\documentclass[11pt]{article} 
\usepackage{float}
\usepackage[utf8]{inputenc}
\usepackage{verbatim,amsmath,amssymb,amsfonts,amsthm,amscd,graphicx,psfrag,epsfig,mathrsfs,stmaryrd}
\usepackage{tikz-cd}
\usepackage[title]{appendix}
\usepackage{amsmath,amssymb,amsfonts}
\usepackage[all]{xy}

\usepackage{mathtools}

\usepackage[alphabetic]{amsrefs}
\usepackage{geometry}
\usepackage{bbm}
\usepackage{bm}
\usepackage{tikz-cd,hyperref} 
\usepackage[]{subcaption}
\usetikzlibrary{positioning,calc,arrows,decorations.pathreplacing,decorations.markings,patterns}
\usepackage{ifthen}

\usepackage{blkarray,stmaryrd}

\newtheorem{theorem}{Theorem}[section]

\newtheorem{theorem-definition}[theorem]{Theorem-Definition}
\newtheorem{theorem-construction}[theorem]{Theorem-Construction}
\newtheorem{lemma-definition}[theorem]{Lemma--Definition}
\newtheorem{lemma-construction}[theorem]{Lemma--Construction}
\newtheorem{lemma}[theorem]{Lemma}
\newtheorem{proposition}[theorem]{Proposition}
\newtheorem{corollary}[theorem]{Corollary}
\newtheorem{conjecture}[theorem]{Conjecture}
\theoremstyle{definition}
\newtheorem{definition}[theorem]{Definition}
\newtheorem{remark}[theorem]{Remark}
\newtheorem{notation}[theorem]{Notation}

\newtheorem{example}[theorem]{Example}

\newcommand{\old}[1]{}

\newcommand{\Z}{{\mathbb Z}}

\newcommand{\R}{{\mathbb R}}
\newcommand{\C}{{\mathbb C}}

\newcommand{\Q}{{\mathbb Q}}
\newcommand{\T}{{\mathbb T}}

\renewcommand{\P}{{\mathbb P}}

\makeatletter
\newcommand{\extp}{\@ifnextchar^\@extp{\@extp^{\,}}}
\def\@extp^#1{\mathop{\bigwedge\nolimits^{\!#1}}}
\makeatother
\newcommand\restr[2]{{
  \left.\kern-\nulldelimiterspace 
  #1 
  \vphantom{\big|} 
  \right|_{#2} 
  }}
\definecolor{calpolypomonagreen}{rgb}{0, 0.6, 0.2}

\newcommand{\ra}{\rightarrow}
\newcommand{\be}{\begin{equation}}
\newcommand{\ee}{\end{equation}}
\newcommand{\bt}{\begin{theorem}}

\newcommand{\et}{\end{theorem}}
\newcommand{\bd}{\begin{definition}}
\newcommand{\ed}{\end{definition}}
\newcommand{\bp}{\begin{proposition}}
\newcommand{\ep}{\end{proposition}}

\newcommand{\bl}{\begin{lemma}}
\newcommand{\el}{\end{lemma}}
\newcommand{\bc}{\begin{corollary}}
\newcommand{\ec}{\end{corollary}}
\newcommand{\bcon}{\begin{conjecture}}
\newcommand{\econ}{\end{conjecture}}
\newcommand{\la}{\label}
\newcommand{\w}{{\rm w}}

\tikzset{ 
	mid arrow/.style={postaction={decorate,decoration={
				markings,
				mark=at position .5 with {\arrow{latex}}
	}}},
	mid rarrow/.style={postaction={decorate,decoration={
				markings,
				mark=at position .5 with {\arrow{latex reversed}}
	}}},
}

\tikzset{qvert/.style={draw,black,circle,fill=gray,minimum size=5pt,inner sep=0pt}  } 
\tikzset{bvert/.style={draw,circle,fill=black,minimum size=5pt,inner sep=0pt}  } 
\tikzset{wvert/.style={draw,circle,fill=white,minimum size=5pt,inner sep=0pt}  } 

  \newcommand{\cX}{{\mathcal{X}}}
\newcommand{\cH}{{\mathcal{H}}}
\newcommand{\cC}{{\mathcal{C}}}
\newcommand{\cL}{{\mathcal{L}}}
\newcommand{\cO}{{\mathcal{O}}}

\begin{document}


\title{The cluster modular group of the dimer model}
\author{Terrence George and Giovanni Inchiostro}

\newcommand{\Addresses}{{
  \bigskip
  \footnotesize
  \textsc{Brown University\\Providence, RI-02906,USA}\\
  \textit{E-mail address}: \texttt{georgete@umich.edu}\\
  \textit{E-mail address}: \texttt{ginchios@uw.edu}
}}

\maketitle
\begin{abstract}
     Associated to a convex integral polygon $N$ is a  cluster integrable system $\mathcal X_N$ constructed from the dimer model. We compute the group $G_N$ of symmetries of $\mathcal X_N$, called the (2-2) cluster modular group, showing that it is a certain abelian group conjectured by Fock and Marshakov. Combinatorially, non-torsion elements of $G_N$ are ways of shuffling the underlying bipartite graph, generalizing domino-shuffling. Algebro-geometrically, $G_N$ is a subgroup of the Picard group of a certain algebraic surface associated to $N$. 
 \end{abstract}

\bibliographystyle{amsxport}

\section{Introduction}
Domino-shuffling is a technique introduced in \cite{EKLP} to enumerate and generate domino tilings of the Aztec diamond graph, and was used to give the first proof of the arctic circle theorem \cite{JPS}. Domino tilings are dual to the dimer model on the square grid. There are generalizations of domino-shuffling, called \textit{(2-2) cluster modular transformations} for other biperiodic bipartite graphs and they comprise the elements infinite order of a group called the \textit{(2-2) cluster modular group}. This group was studied by Fock and Marshakov \cite{FM16}*{Section 7.3} (under the name group of discrete automorphisms) and they gave an explicit conjecture for its isomorphism type. The goal of this paper is to study these generalized shufflings, and in particular, to compute the (2-2) cluster modular group for any biperiodic bipartite graph. 

(2-2) cluster modular transformations give rise to dynamical systems on the space of weights on bipartite graphs as we now explain. Let $\Gamma$ be a bipartite graph on a torus $\mathbb T$ and let $\mathcal L_\Gamma:=H^1(\Gamma,\C^*)$ be the space of weights on $\Gamma$ (cf. section \ref{dimermodel}). There are two types of local rearrangements of bipartite graphs called \textit{elementary transformations} (see Figure \ref{eteqns}). Each elementary transformation has an associated birational map of weights, characterized by the property that it preserves the dimer partition function up to a constant scaling factor (see for example \cite{GK12}*{Theorem 4.7}). Given a sequence of elementary transformations such that the initial and  final graphs are both $\Gamma$ (which we call a \textit{(2-2) cluster transformation}), composing the induced birational maps of weights gives a birational automorphism of $\mathcal L_{\Gamma}$. The cluster transformation is \textit{trivial} if this induced map on weights is the identity. The \textit{(2-2) cluster modular group} is the group of cluster transformations modulo the trivial ones. 

\begin{remark}The word cluster refers to the fact that there is an underlying cluster algebra structure such that the elementary transformations are mutations (see \cite{GK12}). We include the prefix (2-2) because elementary transformations are a special class of mutations at degree $4$ vertices of the underlying quiver, and are often called \textit{2-2 moves}. The full cluster modular group is much larger, but the other mutations are less natural from the point of view of statistical mechanics.
\end{remark}

\begin{figure}
\begin{subfigure}{0.8 \textwidth}
  \centering
  \hspace*{1.9cm}  
  \includegraphics[width=.8\linewidth]{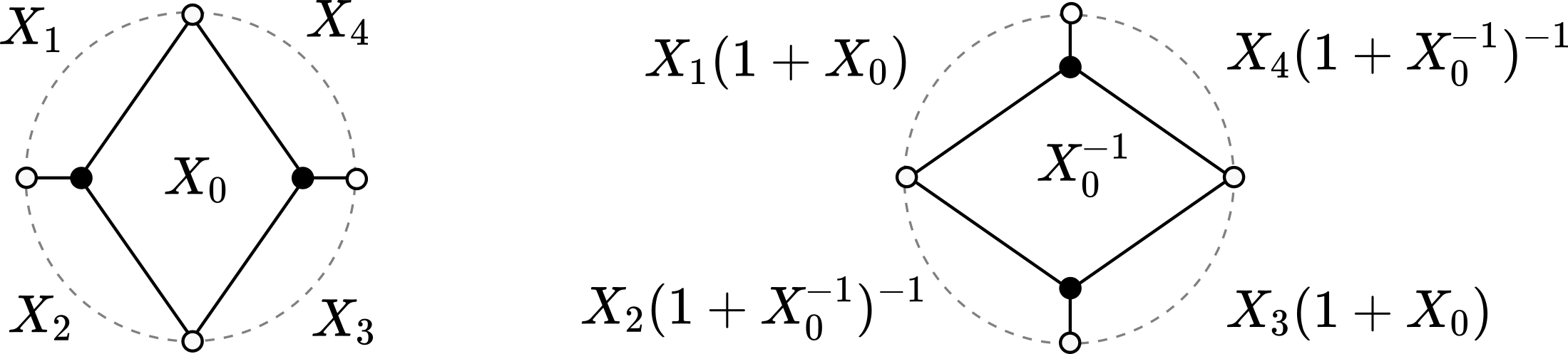}
  \caption{Spider move.}
  \label{eteqns1}
\end{subfigure}
\begin{subfigure}{0.8\textwidth}
  \centering
  \includegraphics[width=.65\linewidth]{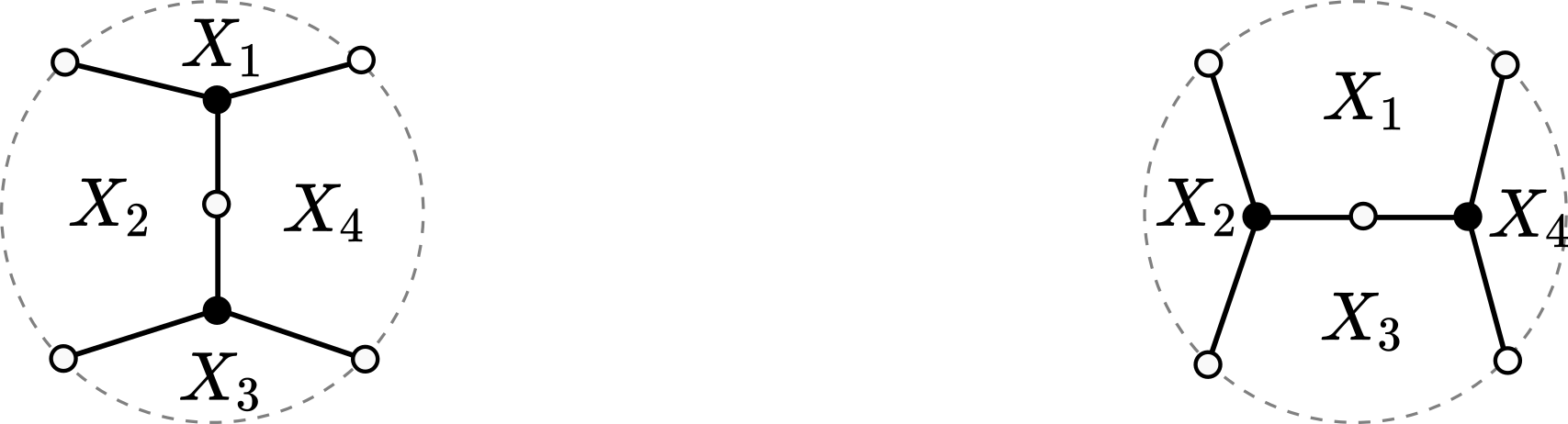}
  \caption{Shrinking/expanding degree $2$ white vertices.}
  \label{eteqns2}
\end{subfigure}
\caption{Elementary transformations along with induced birational maps of weight tori.}
\label{eteqns}
\end{figure}

\begin{figure}
  \centering
  \includegraphics[width=.9\linewidth]{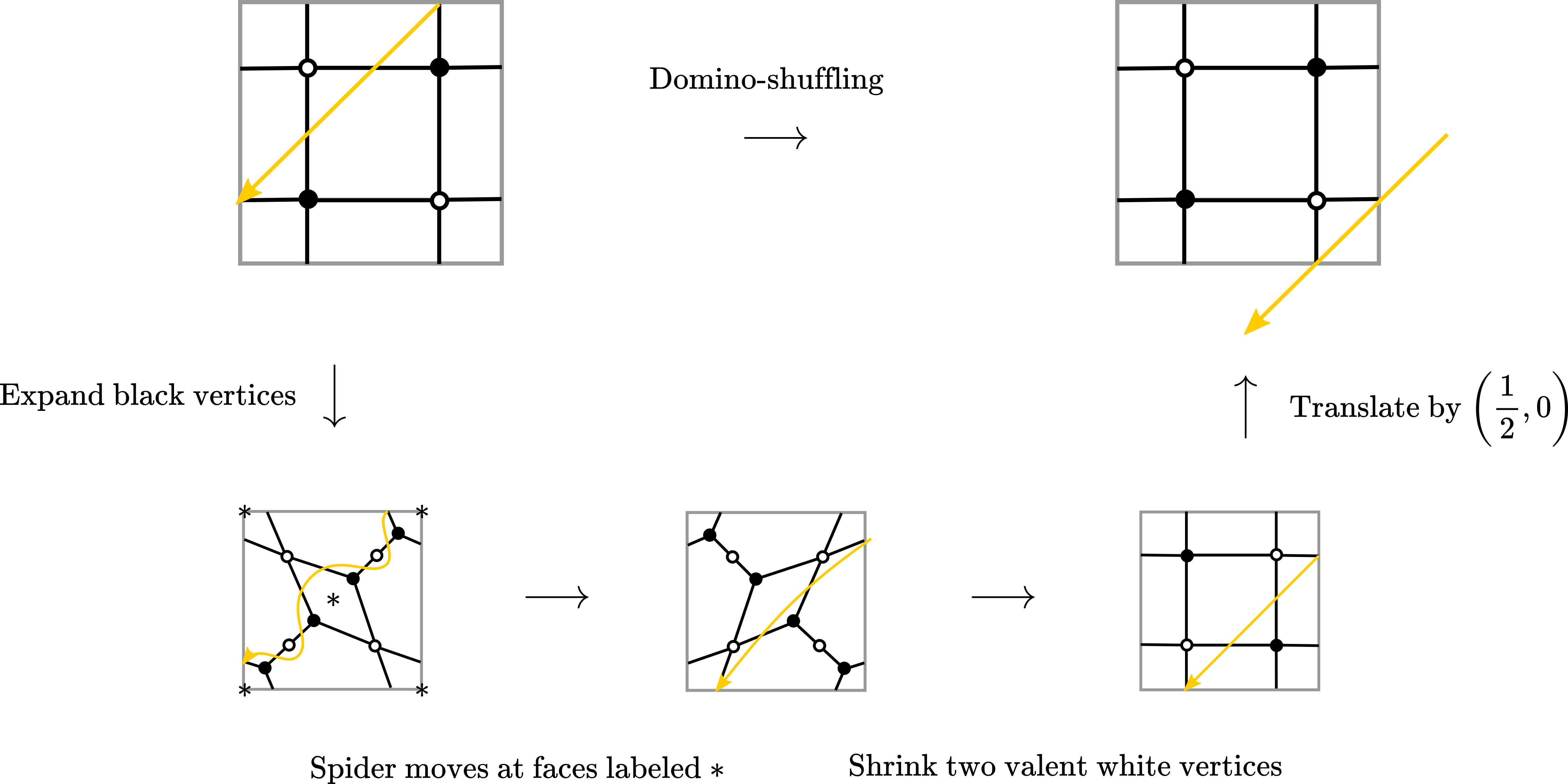}
\caption{The cluster modular transformation called domino-shuffling. }
\label{octrec2}
\end{figure}

\begin{figure}
  \centering
  \includegraphics[width=.3\linewidth]{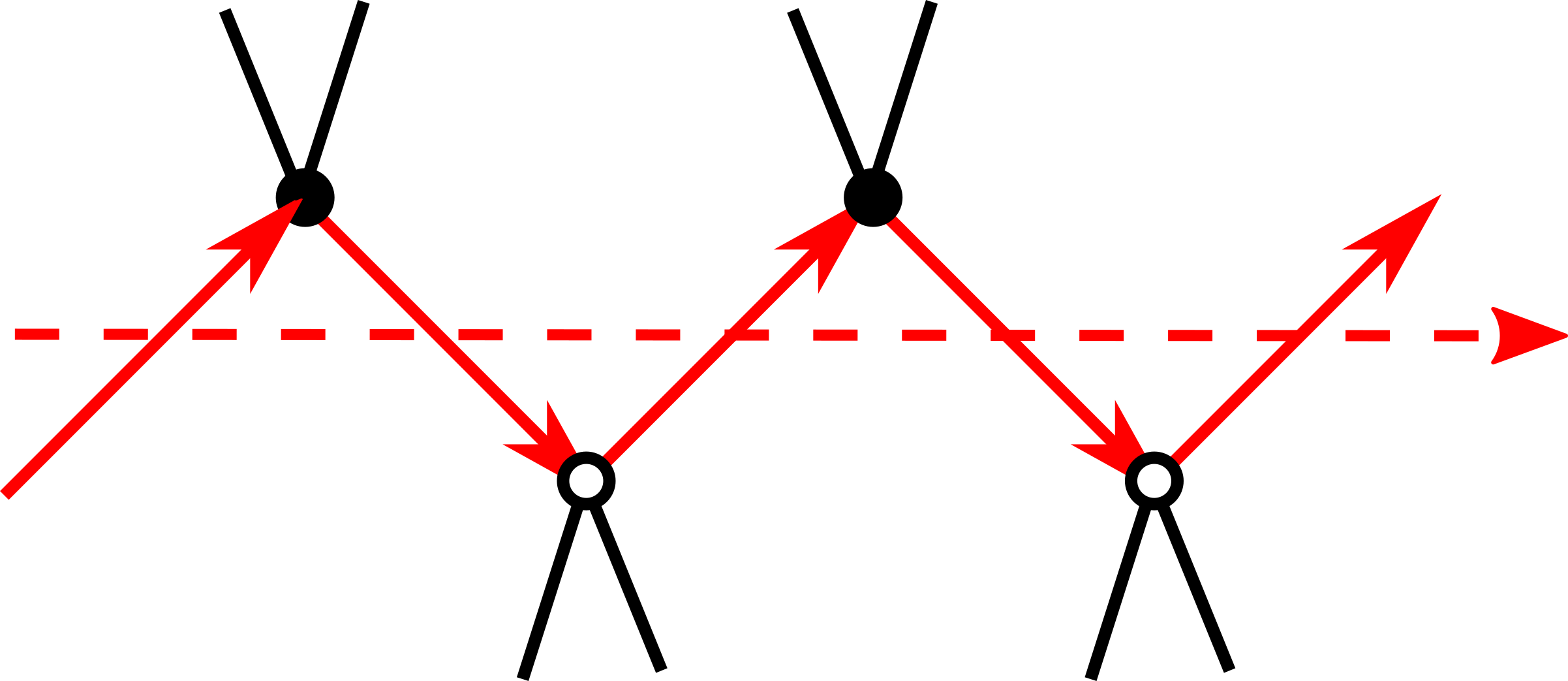}
\caption{A zig-zag path (solid red) and its representation as a path in the medial graph (dashed red).}
\label{zzpath}
\end{figure}

\begin{figure}
\centering
{\includegraphics[width=0.5\textwidth]{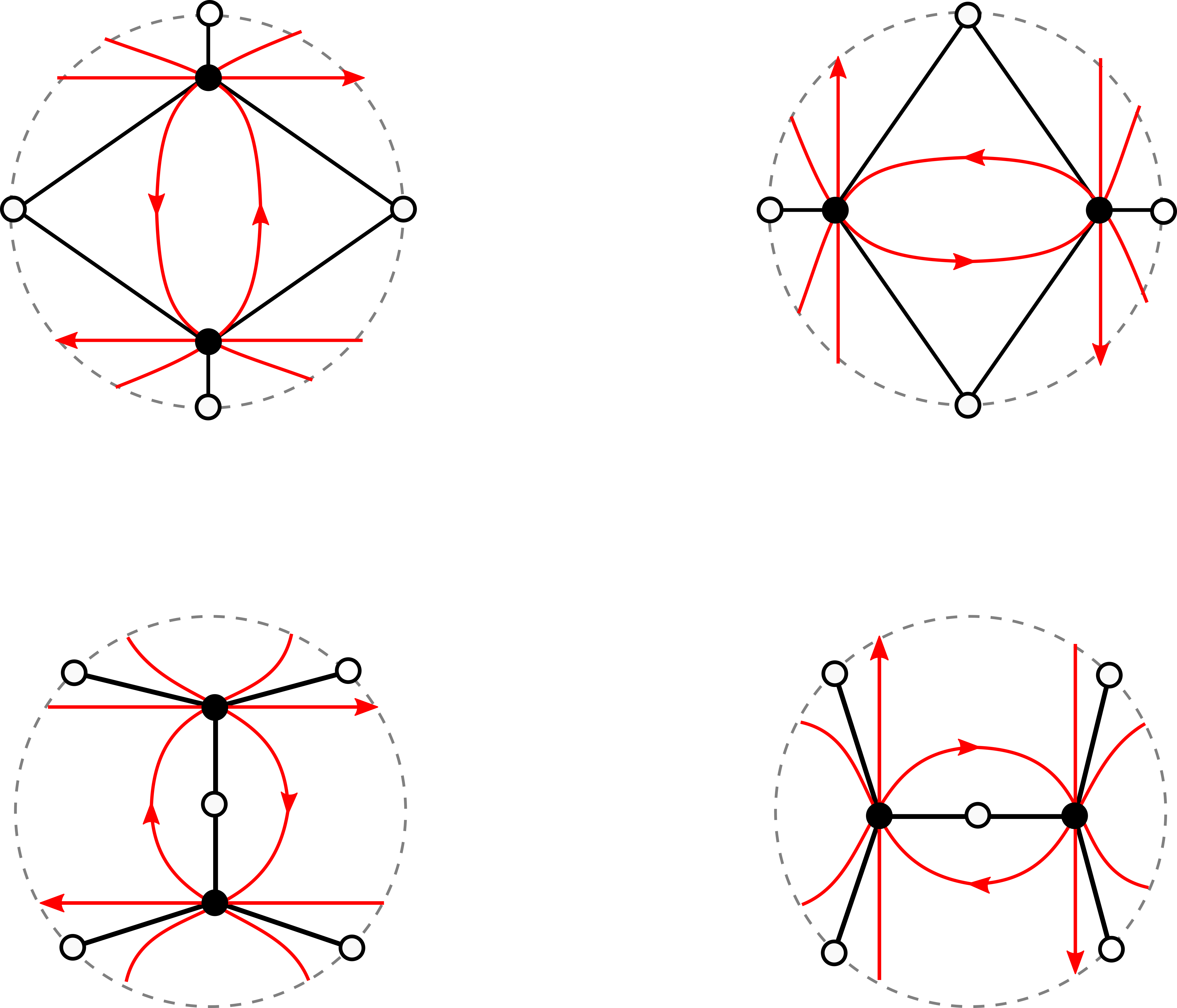}}
\caption{Equivalence of elementary transformations and $2-2$ moves.}\label{eq}
\end{figure}

A \textit{zig-zag path} in $\Gamma$ is a path that turns maximally left at white vertices and maximally right at black vertices (see Figure \ref{zzpath}). Recall that the homology group $H_1(\T,\Z)$ of the torus $\T$ is isomorphic to $\Z^2$. Associated to any bipartite graph on a torus $\T$ is a convex integral polygon $N$ in the plane $H_1(\T,\Z) \otimes_\Z \R \cong \R^2$ called its \textit{Newton polygon}, whose \textit{primitive edge vectors} are given by the homology classes of all zig-zag paths in $\Gamma$. By a primitive edge vector of $N$, we mean a vector contained in an edge of $N$ and oriented in such a way that it is contained in the counterclockwise oriented boundary of $N$, such that its starting and ending points are lattice points (i.e. points in $\Z^2$), and such that there are no other lattice points in its interior. We denote by $E_N$ the set of edges of $N$ (not primitive, so each edge is the union of the primitive edge vectors contained in it). The (2-2) cluster modular group will be completely determined by $N$.  We also point out that elementary transformations have an appealing description in terms of homotopy of zig-zag paths (see Figure \ref{eq} and section \ref{co}). 

\begin{figure}
  \centering
  \includegraphics[width=.2\linewidth]{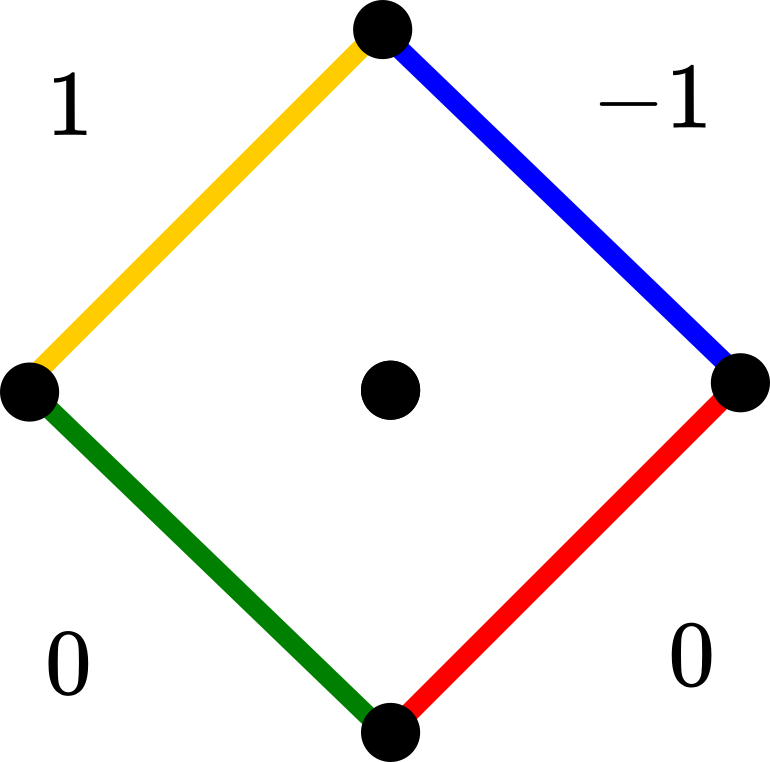}
\caption{The Newton polygon along with the function $f$ for the cluster modular transformation in Figure \ref{octrec2}. The yellow edge corresponds to the yellow zig-zag path in Figure \ref{octrec2}, which is translated one step to the left during the cluster transformation.}
\label{npds}
\end{figure} 
Fock and Marshakov \cite{FM16}*{Section 7.3} constructed a homomorphism from the group of (2-2) cluster transformations to a certain abelian group that we now describe. Let $\widetilde \Gamma$ be the planar biperiodic graph whose quotient under the translation action of $H_1(\T,\Z)$ is $\Gamma$, that is, the preimage of $\Gamma$ in the universal cover of $\T$. We can lift a cluster transformation to an $H_1(\T,\Z)$-periodic sequence of elementary transformations from $\widetilde \Gamma$ to itself. If we superpose $\widetilde \Gamma$ over itself after the cluster transformation, the lift of each zig-zag path is superposed over a lift of a zig-zag path with the same homology class.  To each cluster transformation, we can associate an integer function $f$ on the edges $E_N$ of the Newton polygon $N$ as follows: for any edge $E \in E_N$, the inverse image in the universal cover of the torus of all zig-zag paths corresponding to $E$ (that is all zig-zag paths whose homology classes are in the direction of $E$ when $E$ is oriented counterclockwise along the boundary of $N$) is an infinite collection of ``parallel" zig-zag paths in $\widetilde \Gamma$; let us label them by $(\alpha^i)_{i \in \Z}$, ordered along the direction normal to $E$ and pointing out of $N$. Consider the zig-zag path $\alpha^0$.  After the cluster transformation, if we superpose $\widetilde \Gamma$ over itself, $\alpha^0$ is superposed over a parallel zig-zag path $\alpha^j.$ We define $f(E)$ to be $-j$, which is the number of steps (measured in terms of parallel zig-zag paths) that this zig-zag path (and therefore any zig-zag path $\alpha^i$ parallel to $\alpha^0$) is translated by the cluster transformation. For example, Figure \ref{octrec2} shows the relative positions of a zig-zag path corresponding to the yellow edge $E$ of $N$ in Figure \ref{npds} before and after the cluster transformation corresponding to domino-shuffling from Figure \ref{octrec2}. Since the zig-zag path has been translated one step to the left, we have $f(E)=1$. The evaluations of the function $f$ on the other edges of $N$ are similarly computed (see Figure \ref{npds}). 

The function so defined satisfies (see Section \ref{sec:fockmarsh} for details) 
\be \la{propertystar}
\sum_{E \in E_N}f(E)=0. 
\ee
Let us denote by $\Z^{E_N}_0$ the group of integer functions on $E_N$ satisfying (\ref{propertystar}). Since we passed to the universal cover of $\T$, there is an ambiguity in superposing $\widetilde \Gamma$ over itself because we can translate by $H_1(\T,\Z)$. Therefore to make $f$ a well-defined function of the cluster transformation, we should consider it as an element of the quotient $$\Z^{E_N}_0/H_1(\mathbb T,\Z),$$ where the embedding of $H_1(\T,\Z)$ used in the quotient is given by number of steps that zig-zag paths in $\widetilde{\Gamma}$ are translated by when $\widetilde{\Gamma}$ is translated by elements of $H_1(\mathbb T,\Z)$:
\begin{align*}
    H_1(\mathbb T,\Z) &\hookrightarrow \Z^{E_N}_0\\
    \gamma &\mapsto \left(E \mapsto \langle E,\gamma \rangle_\T \right),
\end{align*}
where $\langle \cdot, \cdot \rangle_\T$ is the intersection form in $H_1(\T,\Z)$. The assignment of the function $f$ to a cluster transformation is a group homomorphism 
\[
\psi:\{\text{Cluster transformations}\} \ra \Z^{E_N}_0/H_1(\T,\Z).
\]
\begin{remark}
Our terminology differs from that of Fock and Marshakov \cite{FM16}, so we provide a translation. The (2-2) cluster modular group is their \textit{group of discrete automorphisms} $\mathcal G_\Delta$, where they use $\Delta$  to denote the Newton polygon, and (2-2) cluster modular transformations are called \textit{discrete flows}.
\end{remark}
Our main result is the following conjecture of Fock and Marshakov \cite{FM16} with a minor modification when $N$ contains no lattice points in its interior.

\begin{theorem}[cf. Theorem \ref{maincor}]\label{mainthm2}
If the Newton polygon $N$ contains at least one interior lattice point, the homomorphism $\psi$ gives an isomorphism of the (2-2) cluster modular group with
\[
\Z^{E_N}_0/H_1(\mathbb T,\Z).
\]
 If $N$ contains no interior lattice points, the (2-2) cluster modular group is a smaller finite group.
\end{theorem}
In particular, the rank of the (2-2) cluster modular group depends only on the number of edges of $N$.
\begin{corollary}\label{rank1}
When $N$ contains an interior lattice point, the rank of the (2-2) cluster modular group is $|E_N|-3$, where $|E_N|$ is the number of edges of the polygon $N$. When $N$ has no interior lattice points, the rank is zero.
\end{corollary}
Informally, while the collection of all zig-zag paths undergoes a complex sequence of moves, if we restrict attention to the set of zig-zag paths in a specific homology direction, no two zig-zag paths in this set can cross during a cluster transformation. Therefore this set of zig-zag paths as a whole undergoes a translation. The function $f$ defined above records these translations, and remarkably,  we can essentially reconstruct the entire cluster transformation from $f$. 

The proof of Theorem \ref{mainthm2} has two parts. In Section \ref{sectionsurj}, we show that every element of $\Z^{E_N}_0/H_1(\mathbb T,\Z)$ arises from a cluster transformation. This part of the proof is purely combinatorial.

Translations by elements of $H_1(\mathbb T,\Z)$ clearly give rise to trivial cluster transformations. The second part of the proof of the Theorem shows that these are the only trivial cluster transformations. It is difficult to directly check if the induced birational map of weights is the identity. However, integrability of the space of weights $\mathcal L_\Gamma$ means that there is a local reparameterization such that the birational map of weights induced by cluster transformations are linearized. 

Associated to a polygon $N$ is a certain compactification $X_N$ of $(\C^*)^2$ called a \textit{toric surface} (see for example \cite{CLS11}).
Kenyon and Okounkov \cite{KO} defined the \textit{spectral transform} of $wt \in\mathcal L_\Gamma$ to be a triple $(C, S, \nu)$, where $C \subset X_N$ is a curve called the \textit{spectral curve} and $S$ is a divisor of degree $g$ equal to the genus of $C$, that is a formal linear combination of $g$ points in $C$, and $\nu$ is a bijection between zig-zag paths and the points at infinity of $C$ (i.e. the points in $C \cap (X_N \setminus (\C^*)^2)$). The curve $C$ is the vanishing locus of a Laurent polynomial $P(z,w)$ which is a homology-class-weighted version of the partition function for dimer covers. Fock \cites{F15} proved that the spectral transform is birational, allowing us to view $(C,S,\nu)$ as a local reparameterization of $\mathcal L_\Gamma$. For a fixed curve $C$, the Jacobi inversion theorem states that the space of degree $g$ effective divisors in $C$ is birational to a $g$-dimensional complex torus called the Jacobian variety of $C$. In this parameterization, every  cluster transformation leaves $C$ invariant and is a translation of the divisor $S$ in the Jacobian variety of $C$. This translation depends only on the function $f$ associated to the cluster transformation and was described explicitly by Fock \cite{F15} (see Figure \ref{fockfig} for an illustration and Proposition \ref{fockthm} for a precise statement). 

\begin{figure}
  \centering
  \includegraphics[width=.6\linewidth]{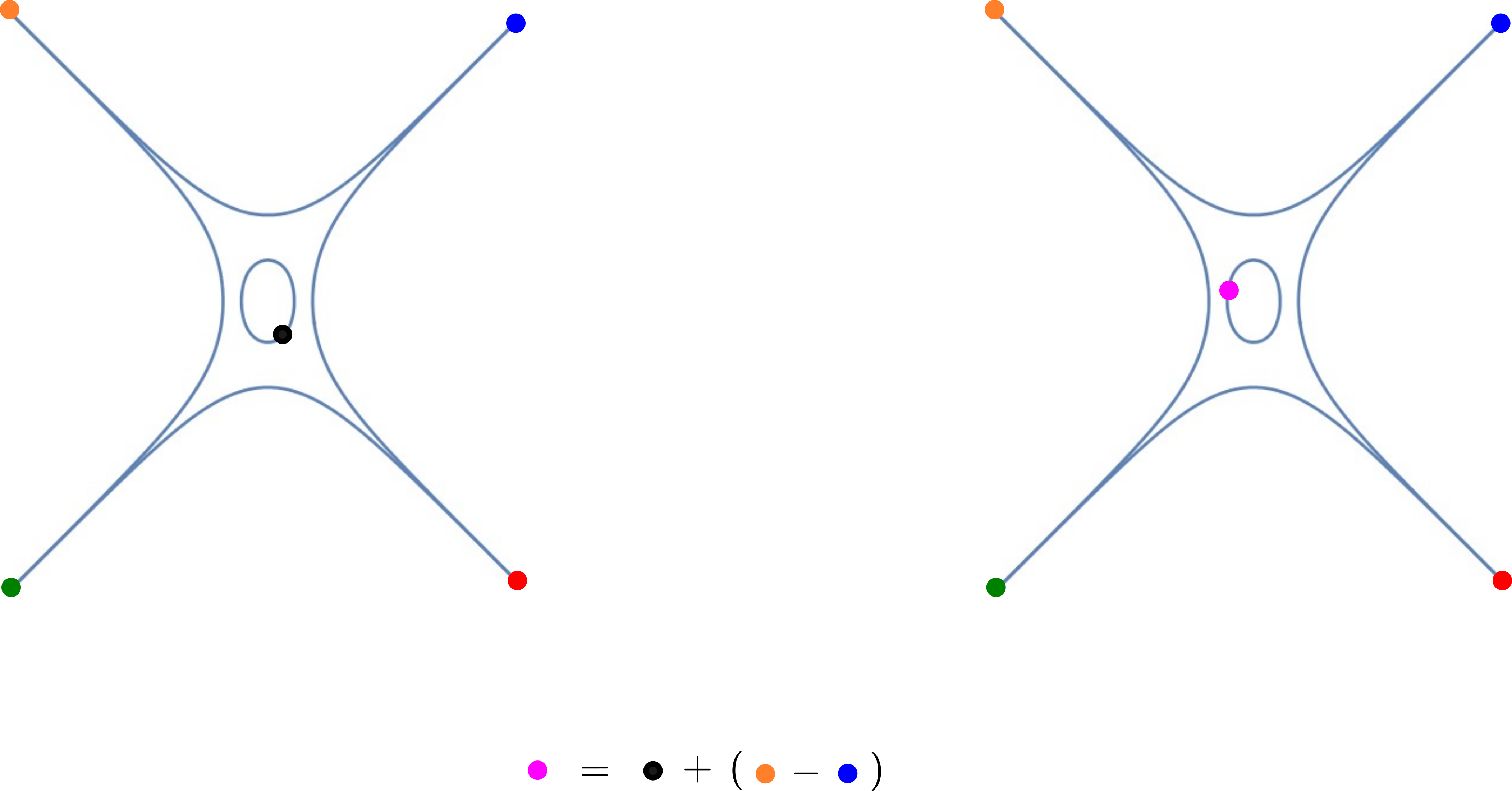}
\caption{The black point on the left is the divisor $S$ on the amoeba of the spectral curve. The points at infinity of the curve are in bijection with zig-zag paths and coloured according to Figure \ref{npds}. The cluster transformation in Figure \ref{octrec2} maps the black point to the pink point. Fock \cite{F15} shows that this map is the translation shown below the figure in the Jacobian variety of the spectral curve. This translation is determined by the function $f$ shown in Figure \ref{npds}.}
\label{fockfig}
\end{figure}
Therefore the question of which cluster transformations are non-trivial can be answered by looking at which translations on the Jacobian variety of $C$ are non-trivial. Under the standard equivalence between divisors and line bundles in algebraic geometry, a translation in the Jacobian corresponds to tensoring by a certain line bundle, so we need to understand when certain line bundles on $C$ are non-trivial. The following Theorem provides the answer. 
\begin{theorem}[cf. Theorem \ref{theorem:appendix1}]
Suppose $N$ contains an interior lattice point. If $L$ is a non-trivial line bundle on the toric surface $X_N$ associated to $N$, then for a generic spectral curve $C$, we have $L|_C \ncong \mathcal O_C$. 
\end{theorem}

We can now explain why the absence of an interior lattice point in $N$ makes the cluster modular group finite. The genus of a generic curve $C$ defined as the vanishing locus of a Laurent polynomial $P(z,w)$ is equal to the number of interior lattice points of the Newton polygon $N$ of $P(z,w)$ (see for example \cite{CLS11}*{Proposition 10.5.8}). Therefore if $N$ has no interior lattice points, then a generic spectral curve $C$ has genus $0$, and therefore is isomorphic to $\P^1$. The Jacobian variety of $\P^1$ is $0$, so every cluster transformation is determined by its action on the finite data $\nu$. See Example \ref{eg:bf} for an example of $N$ with no interior lattice points.

The (2-2) cluster modular group has been studied in the mathematical physics literature by Eager and Franco \cite{EF}, where it is called the \textit{space of Seiberg duality cascades}. They provide a description that is equivalent to that of Fock and Marshakov. We comment on this further in Section \ref{sec:ef}. 

In the last paragraph of \cite{FM16}*{Section 7.3}, Fock and Marshakov provide an alternate description of $\Z^{E_N}_0/H_1(\mathbb T,\Z)$ as the group of divisor classes on the toric surface $X_N$ that restrict to degree $0$ divisors on a generic spectral curve $C$. However this is only true as stated for polygons whose sides are all primitive, that is, no side contains a lattice point other than the end points (see Example \ref{eg:prim}). Recently Treumann, Williams and Zaslow  \cite{TWZ18} gave a different version of linearization of cluster modular transformations under the spectral transform, replacing the toric variety $X_N$ by a toric stack $\mathcal X_N$. 

\begin{proposition}[cf. Proposition \ref{propp2}]\label{mainthm1}
When the Newton polygon $N$ contains an interior lattice point, the (2-2) cluster modular group can be identified with certain subgroup of $\text{Pic}(\mathscr X_N)$.
\end{proposition}

We end the introduction by describing the (2-2) cluster modular groups for some small Newton polygons.

\paragraph{Triangles}
For triangular $N$, \cite{IU15}*{Proposition 11.3} tells us that there is a unique bipartite graph in $\T$ with Newton polygon $N$ and its lift to the plane is the honeycomb lattice. Since this graph does not admit any elementary transformations, the only cluster modular transformations are translations.

\paragraph{Quadrilaterals}

Corollary \ref{rank1} tells us that the cluster modular group has rank one. The dimer models that have quadrilateral Newton polygons coincide with those that arise from Speyer's ``crosses and wrenches" construction \cite{Speyer}. The octahedron recurrence studied there is the (essentially unique) non-torsion cluster modular transformation (on the $\mathcal A$ cluster variety). Other incarnations of cluster modular transformations for quadrilateral $N$ are Hirota's bilinear difference equation \cite{Miwa}, the domino-shuffling algorithm \cites{EKLP,Propp03} (see Example \ref{eg:ds}), the shuffling studied in \cite{BF18} (see Example \ref{eg:bf}) for the suspended pinch-point graph and the pentagram map \cite{FM16}*{Section 8.5}. Another large class of examples with quadrilateral Newton polygons arise from the $Y^{p,q}$ theories in mathematical physics (see for example \cite{Franco2006BraneDA}*{Section 9.3.1}).

The octahedron recurrence can be used to compute arctic curves \cites{PS06, DFS14}. We observed in \cite{G18} that part of the data needed for this technique of computing arctic curves is a cluster modular transformation along with edge-weights that are periodic under the induced birational map. We hope that understanding the cluster modular group will help generalize this method beyond the quadrilateral Newton polygon case. Since higher degree polygons have cluster modular groups with rank greater than one by Corollary \ref{rank1}, we expect a family of arctic curves, one for each cluster modular transformation of infinite order.

\paragraph{Higher degree polygons}

Cluster modular transformations for the  del Pezzo quiver $dP_2$, which has a pentagon Newton polygon, were explicitly studied in \cite{GLVY16}. The $dP_3$ quiver with a hexagonal Newton polygon has been studied in \cites{LMNT14, LM17, LM19}. The cube recurrence studied in \cites{CS04, PS06} arises as the restriction to the resistor network subvariety of a cluster modular transformation on the $dP_3$ graph \cite{GK12}*{Section 6.3}.

\paragraph{Acknowledgments.} 
We are grateful to Dan Abramovich, Melody Chan, Rick Kenyon, Gregg Musiker, Harold Williams and Xufan Zhang. We also thank the anonymous referees for many helpful comments and suggestions.

\section{Background}

\paragraph{Some basic notation.}

Let $\T$ be a topological torus, and let $T:=H_1(\T,\Z)^\vee \otimes_\Z \C^*$ be the algebraic torus with group of characters $H_1(\T,\Z)$. Here $H_1(\T,\Z)^\vee$ denotes the dual group $\text{Hom}_\Z(H_1(\T,\Z),\Z)$. Given an convex integral polygon $N \subset H_1(\T,\R)$, that is, a convex polygon whose vertices are in $H_1(\T,\Z) $, we denote by $V_N$ and $E_N$ the vertices and edges of $N$ respectively.

Let $\Sigma \subset H_1(\T,\Z)^\vee \otimes _\Z \R$ denote the dual fan of $N$. Let $\Sigma(r)$ denote the $r$-dimensional faces of $\Sigma$. Let $u_\rho $ be the primitive integral vector along the ray $\rho \in \Sigma(1)$. Let $E_\rho $ denote the edge of $N$ that is dual to $\rho$. Let $|E_\rho|$ be its \textit{integral length}, defined as the number of primitive integral vectors in $E_\rho$.

\subsection{Combinatorial objects}\la{co}

See \cite{GK12} for further background on the objects described in this section.

\paragraph{Bipartite torus graphs.}

\begin{figure}
\centering
{\includegraphics[width=0.45\textwidth]{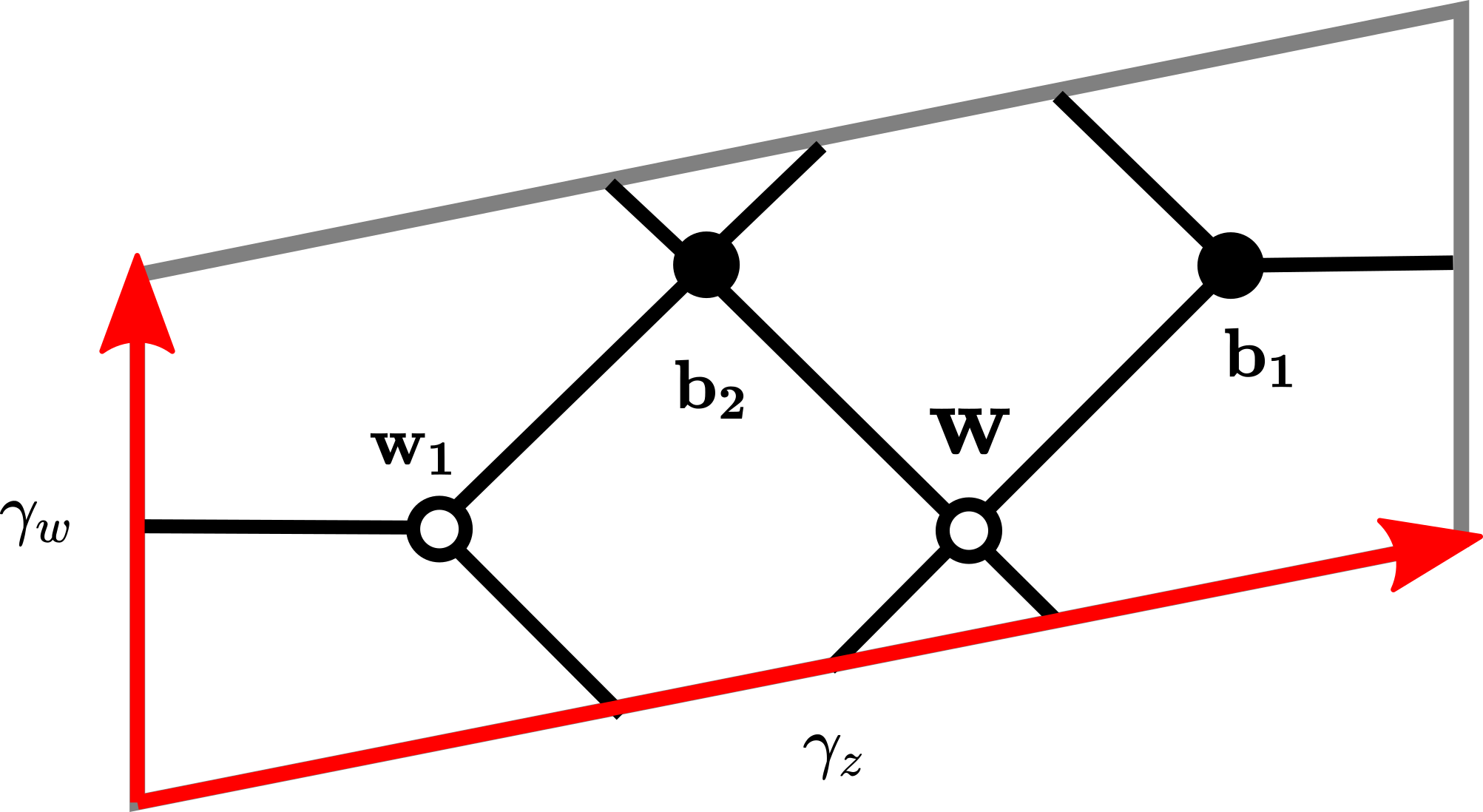}}
\caption{A fundamental parallelogram for a bipartite torus graph. The generators $\gamma_z, \gamma_w$ of $H_1(\T,\Z)$ are as  shown.}\label{bfgraph}
\end{figure}

A \textit{bipartite graph} is a graph whose vertices are colored black or white, such that each edge is incident to a black and a white vertex. A \textit{bipartite torus graph} is a bipartite graph $\Gamma$ embedded in $\mathbb T$ such that the faces of $\Gamma$, that is, the connected components of $\mathbb T - \Gamma$, are contractible. We denote by $B(\Gamma)$ and $W(\Gamma)$ the black and white vertices of $\Gamma$ respectively.

\begin{figure}
\centering
{\includegraphics[width=0.8\textwidth]{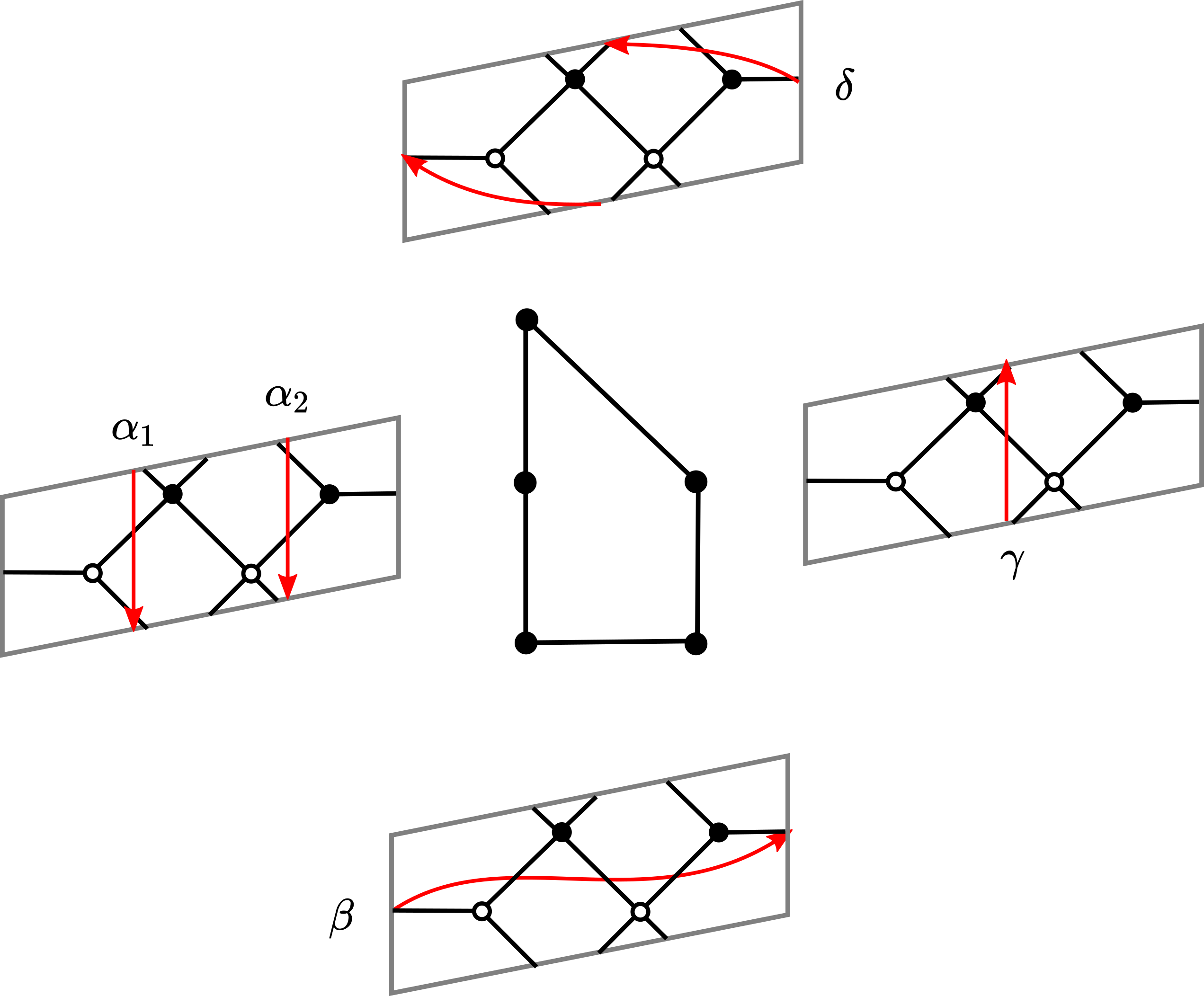}}
\caption{The Newton polygon and zig-zag paths for the graph in Figure \ref{bfgraph}.}\label{bfzzpaths}
\end{figure}

\paragraph{Zig-zag paths and minimality.}
A \textit{zig-zag path} in $\Gamma$ is an oriented path in $\Gamma$ that turns maximally left at white vertices and maximally right at black vertices.  We usually represent a zig-zag path by an oriented path in the medial graph that passes consecutively through the edges of the zig-zag path (see Figure \ref{zzpath}). $\Gamma$ is said to be  \textit{minimal} if, in the preimage $\widetilde{\Gamma}$ of $\Gamma$ in the universal cover $H_1(\T,\R)$ of $\T$, zig-zag paths have no self-intersections and there are no parallel bigons, that is, pairs of zig-zag paths oriented the same way intersecting at two points. The unique convex integral polygon $N(\Gamma) \subset H_1(\mathbb T,\R)$ whose primitive integral edges are given by the homology classes of zig-zag paths in counterclockwise cyclic order is called the \textit{Newton polygon} of $\Gamma$. We usually abbreviate $N(\Gamma)$ to $N$ when the graph is clear from context.

We label the edges of $N$ by rays of the dual fan: the edge corresponding to $\rho \in \Sigma(1)$ is denoted by $E_\rho$. We denote by $Z_\rho$ the set of zig-zag paths whose homology classes are the primitive vectors contained in the edge $E_\rho$.

\begin{example}
Figure \ref{bfgraph} shows a bipartite graph $\Gamma$ in the torus, and Figure \ref{bfzzpaths} shows its zig-zag paths and Newton polygon. It is easily checked that $\Gamma$ is minimal.
\end{example}

\paragraph{Elementary transformations.}
There are two local rearrangements of bipartite torus graphs called \textit{elementary transformations}:
\begin{enumerate}
    \item Spider moves (Figure \ref{eteqns1});
    \item Shrinking/expanding $2$-valent white vertices (Figure \ref{eteqns2}).
\end{enumerate}

We say that two bipartite torus graphs $\Gamma_1$ and $\Gamma_2$ are \textit{topologically equivalent} if there is a sequence of elementary transformations that converts the graph $\Gamma_1$ into $\Gamma_2$. Applying either of the elementary transformations twice gives back the original graph, and therefore this is an equivalence relation on bipartite torus graphs. Elementary transformations are local and do not change homology classes of zig-zag paths. Therefore they leave the Newton polygon invariant and so
\begin{align}\label{npfunction}
\{\text{Minimal bipartite torus graphs}\}/\text{topological equivalence} \xrightarrow[]{\Gamma \mapsto N(\Gamma) } \nonumber \\ \{\text{Convex integral polygons in }H_1(\T,\R)\},
\end{align}
is a well-defined function.

\begin{theorem}[Goncharov and Kenyon, 2012 \cite{GK12}*{Theorem 2.5}]\label{family}
The function in (\ref{npfunction}) which associates to a graph its Newton polygon is a bijection.
\end{theorem}
In other words, for each convex integral polygon in $H_1(\T,\R)$, there is a family of minimal bipartite torus graphs associated to $N$, and any two members of a family are related by elementary transformations.

\paragraph{Triple point diagrams.}

A \textit{triple point diagram} in a disk $\mathbb D$ is a collection of oriented curves called \textit{strands}, defined up to isotopy, such that:

\begin{enumerate}
\item Three strands meet at each intersection point.
\item The end points of each strand are distinct boundary points.
\item The orientations on the strands induce consistent orientations on the complementary regions.
\end{enumerate}
Each strand starts and ends in $\partial \mathbb D$, so if there are $n$ strands, there are $2n$ points in $\partial \mathbb D$, whose orientations alternate ``in" and ``out" around $\partial \mathbb D$. A triple point diagram is \textit{minimal} if strands have no self intersections and parallel bigons.\\

\begin{figure}
\centering
\includegraphics[width=0.6\textwidth]{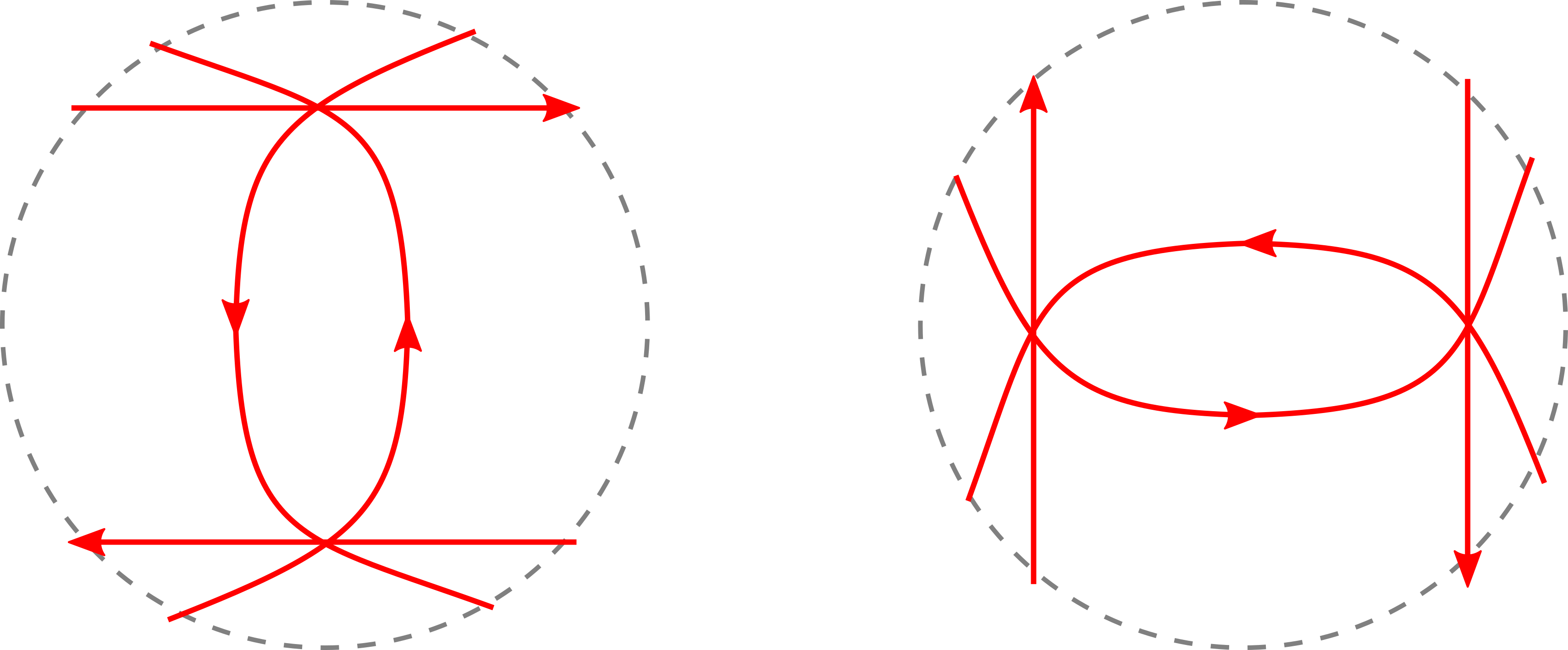}
\caption{The 2-2 move.}\label{2-2}
\end{figure}

There is a local move called a 2-2 move on triple point diagrams (see Figure \ref{2-2}).

\begin{theorem}[Thurston, 2004 \cite{Thur04}, Postnikov, 2006 \cite{Post06}] \label{TT} Suppose we have a disk $\mathbb D$ with $2 n $ points in its boundary alternately labeled ``in" and ``out".
\begin{enumerate}
\item For any of the $n!$ matchings of ``in" and ``out" points, there is a minimal triple point diagram that realizes the matching.
\item Any two minimal triple point diagrams with the same boundary matching of ``in" and ``out" points are related by 2-2 moves.
\end{enumerate}
\end{theorem}

In the course of proving Theorem \ref{TT}, Thurston proves the following result that we will require later.
\begin{proposition}[Thurston, 2004 \cite{Thur04}*{Section 2}]\label{TT2}
Let $\alpha, 
\beta, \gamma$ be three strands that correspond to three consecutive points on the boundary of $\mathbb D$. Then there is a triple crossing diagram (called \textit{standard} in \cite{Thur04}) in which $\alpha, \beta$ and $\gamma$ meet at a triple point just adjacent to the boundary (that is, this is the first triple point of each of these strands as we look along the strand starting at this boundary point).
\end{proposition}

\paragraph{Triple point diagrams in $\mathbb T$.}\label{tpdt2}

A triple point diagram in $\mathbb T$ is a collection of oriented curves called \textit{strands} in $\mathbb T$, determined up to isotopy, such that:
\begin{enumerate}
    \item Three strands meet at each intersection point.
    \item No strand is a homologically trivial loop in $\T$.
    \item The orientations on the strands induce consistent orientations on the complementary regions.
\end{enumerate}
A triple point diagram in $\mathbb T$ is \textit{minimal} if the lift of any strand to the plane has no self-intersections and the lifts of any two strands to the universal cover form no parallel bigons.

\paragraph{Equivalence of triple point diagrams and bipartite torus graphs in $\mathbb T$.}

We recall the equivalence between minimal triple point diagrams in $\mathbb T$ and minimal bipartite torus graphs from \cite{GK12}:
\begin{enumerate}
\item To convert a minimal bipartite torus graph to a triple point diagram, we first expand all black vertices with degree greater than or equal to $4$ by moves inverse to shrinking a degree $2$ white vertex to get a graph in which all black vertices have degree $3$. Then we draw all zig-zag paths so that the black complementary regions are now triangles. Finally we shrink all these black triangle regions into points to get a triple point diagram.

\item To construct a bipartite graph from a triple point diagram, we start by resolving each triple point into a counterclockwise triangle. Put a black vertex in each complementary region that is oriented counterclockwise and a white vertex in each complimentary region that is oriented clockwise. Edges between black and white vertices are given by the vertices of the resolved triple point diagram. The faces of the bipartite graph will be the regions where the orientations alternate.
\end{enumerate}
Under this correspondence, we have:
\begin{align*}
 \text{Minimal bipartite torus graphs} &\longleftrightarrow \text{Minimal triple point diagrams in }\T,\\
    \text{Zig-zag paths} &\longleftrightarrow \text{Strands,}\\
    \text{Elementary transformations} &\longleftrightarrow \text{(2-2) moves}.
\end{align*}

\subsection{The dimer model}\la{dimermodel}

In this section, we introduce the dimer model, mostly following \cite{GK12}.
\paragraph{Weights on bipartite torus graphs.}
We associate to $\Gamma$ the torus of \textit{weights}
$$
\mathcal L_\Gamma:=H^1(\Gamma,\C^*).
$$
A $1$-cocycle representing $wt \in H^1(\Gamma,\C^*)$ is called an \textit{edge-weight}. For $L \in H_1(\Gamma,\Z)$, we denote the pairing of cohomology and homology by $wt(L)$.

For a face $f$ of $\Gamma$, we denote by $\partial f$ the counterclockwise oriented boundary of $f$. We define the \textit{face variables}
$$X_f:=wt(\partial f).$$  
They satisfy the unique relation $\prod_f X_f=1$, arising from the relation $\sum_f  \partial f =0$ in $H_1(\Gamma,\Z)$.

\paragraph{Conjugated surface and the $\epsilon$ form.} Given a bipartite torus graph $\Gamma$, by puncturing each face, we obtain a ribbon graph. Alternately, we can think of the ribbon graph as being obtained from $\Gamma$ by thickening the edges of $\Gamma$. From this ribbon graph, we can construct a new ribbon graph $\widehat \Gamma$ cutting each edge in the middle and gluing it back with a twist. Equivalently, a ribbon structure is the same thing as a cyclic ordering of edges around each vertex of $\Gamma$, and the new ribbon graph $\widehat \Gamma$ is obtained by reversing the cyclic order at each white vertex. From the description in terms of twisting edges, we see that the process of constructing the conjugated surface interchanges boundaries of faces and zig-zag paths. Gluing in the disks along the boundary components of $\widehat \Gamma$ (which are in bijection with zig-zag paths of $\Gamma$), we obtain a surface $\widehat S$ of genus $g$ called the \textit{conjugated surface}, where $g$ is the number of interior lattice points in $N$. 

Since $\widehat \Gamma $ is homotopy equivalent to $\Gamma$, we can define a skew-symmetric bilinear form $\epsilon :H_1(\Gamma,\Z) \times H_1(\Gamma,\Z)$ as follows: If $L_1,L_2 \in H_1(\Gamma,\Z),$ using the homotopy equivalence of $\widehat \Gamma$ with $\Gamma$, we can identify them with homology classes in $H_1(\widehat \Gamma,\Z)$. Using the embedding $\widehat \Gamma \hookrightarrow \widehat S$, they are loops in $\widehat S$. Let $\langle\cdot,\cdot\rangle_{\widehat S}$ denote the intersection form on $\widehat S$. Define $\epsilon(L_1,L_2):=\langle L_1, L_2 \rangle_{\widehat S}$.

\paragraph{Mutations.} Elementary transformations $s:\Gamma_1 \ra \Gamma_2$ bipartite torus graphs induce birational maps of weights $\mu_s:\mathcal L_{\Gamma_1 } \dashrightarrow \mathcal L_{\Gamma_1 }$ described below. In both cases, there is a canonical identification, which we also call $s$, of $H_1(\Gamma_1,\Z)$ with $H_1(\Gamma_2,\Z)$.
\begin{enumerate}
    \item Spider move at face $f$: We define $\mu_s$ by:
    $$
    \mu_s(wt)(L)=wt(s^{-1}(L))\left(1+wt(f)^{-\text{sign }\epsilon(s^{-1}(L),\partial f)}\right)^{-\epsilon(s^{-1}(L),\partial f)}.
    $$
  See Figure \ref{eteqns1} for how the weights of the faces involved transform. 
    \item Shrinking/expanding degree two white vertices: see Figure \ref{eteqns2}.  We define $\mu_s$ by:
    $$
    \mu_s(wt)(L)=wt(s^{-1}(L)).
    $$
\end{enumerate}

\paragraph{The dimer cluster variety $\mathcal X_N$.} Suppose $N$ is a convex integral polygon in $H_1(\T,\R)$. By theorem \ref{family}, there is a family of minimal bipartite torus graphs with Newton polygon $N$ that are related by elementary transformations. Associated with each graph $\Gamma $ in the family is its torus of weights $\mathcal L_\Gamma$. Gluing the $\mathcal L_\Gamma$ using the birational maps induced by the elementary transformations, we obtain a space $\mathcal X_N$ called the {\it dimer cluster variety}.

\paragraph{The (2-2) cluster modular group.}
We say that two bipartite torus graphs $\Gamma_1$ and $\Gamma_2$ are \textit{isotopic} if there is an isotopy in $\T$ relating $\Gamma_1$ and $\Gamma_2$. 
A \textit{(2-2) cluster transformation} $t:\Gamma_0 \ra \Gamma_n$ is a sequence:
$$
\Gamma_0 \xrightarrow[]{s_0} \Gamma_1 \xrightarrow[]{s_1} \cdot \cdot \cdot \xrightarrow[]{s_{n-1}} \Gamma_n,
$$
where each $s_i$ is an elementary transformation or an isotopy in $\T$. A (2-2) cluster transformation $t$ induces a birational map $\mu_t$ of weight tori by composition:
$$
\mu_t:= \mu_{s_{n-1}} \circ \cdot \cdot \cdot \circ \mu_{s_0}:\mathcal L_{\Gamma_0} \ra \mathcal L_{\Gamma_n}. 
$$
A (2-2) cluster transformation $\Gamma \ra \Gamma $ is called \textit{trivial} if the induced birational map of weight tori is the identity. The groupoid $\mathcal G_N$ whose objects are minimal bipartite torus graphs $\Gamma$ with Newton polygon $N$ and morphisms are (2-2) cluster transformations modulo trivial (2-2) cluster transformations is called the \textit{(2-2) cluster modular groupoid} of $\mathcal X_N$. The fundamental group $G_{N}$ of $\mathcal G_{N}$ is called the \textit{(2-2) cluster modular group} and its elements are called \textit{(2-2) cluster modular transformations}. Although we need a base point $\Gamma$ to define the fundamental group $G_N$, a different choice of base point gives an isomorphic group. Elements of $G_N$ are also called \textit{discrete flows} in \cite{FM16}.

\paragraph{Dimer covers.} A \textit{dimer cover} or \textit{perfect matching} of $\Gamma$ is a collection of edges of $\Gamma$ such that each vertex of $\Gamma$ is incident to exactly one edge in the collection. By orienting each edge from its black vertex to its white vertex, we can view each dimer as a $1$-chain in $\Gamma$. Fix a dimer cover $M_0$ which we call the \textit{reference dimer cover}. Then we can associate to each dimer cover $M$ a homology class $[M-M_0] \in H_1(\T,\Z)$ and weight $wt([M-M_0])$. The Newton polygon $N$ has the following description in terms of dimer covers.

\begin{proposition}[\cite{GK12}*{Theorem 3.12}]\label{nplemma}
Suppose $\Gamma$ is a minimal bipartite torus graph with Newton polygon $N$. Up to a translation in $H_1(\T,\R)$, we have:
$$
N=\text{Convex-hull }\{ [M-M_0]: M \text{ is a dimer cover of $\Gamma$}\}.
$$
\end{proposition}

\paragraph{Kasteleyn theory.} Let $R$ be a fundamental rectangle for $\T$. Let $\gamma_z,\gamma_w$ be the oriented sides of $R$ generating $H_1(\T,\Z)$, as shown in Figure \ref{bfgraph}. To each edge $e$ of $\Gamma$, we associate a character
\be \la{edgeph}
\varphi(e)=z^{( e,\gamma_w )} w^{(e,-\gamma_z)},
\ee
where we consider the edge $e$ to be oriented from its black vertex to its white vertex and $(\cdot,\cdot)$ is the local intersection number.\\

$\kappa\in H^1(\Gamma,\C^*)$ is called a  \textit{Kasteleyn sign} if:
\begin{enumerate}
    \item $\kappa(L)=\pm 1$ for all $L \in H_1(\Gamma,\Z)$.
    \item $\kappa(\partial f)=(-1)^{l/2+1}$, if $f$ is a face of $\Gamma$ containing $l$ edges in its boundary.
\end{enumerate}
The \textit{Kasteleyn matrix} 

 \begin{align*}  K(z,w)&:\C[z^{\pm 1}, w^{\pm 1}]^{B(\Gamma )} \ra  \C[z^{\pm 1}, w^{\pm 1}]^{W(\Gamma )}
\end{align*}
is defined as 
$$
 K(z,w)_{\text{w}, \text b}=\sum_{e \in E(\Gamma)\text{ incident to } \text w, \text b} wt(e) \kappa(e)z^{( e,\gamma_w )} w^{( e,-\gamma_z)},
$$
where $\kappa, wt$ are any 1-cocycles representing their cohomology classes. 

\begin{theorem}[Kasteleyn 1963, \cite{Kast63}]\label{Kastthm} We have
\[
\frac{\text{det }K(z,w)}{wt(M_0) z^{( M_0,\gamma_w )} w^{( M_0,-\gamma_z)}}= \sum_{M \text{dimer cover of $\Gamma$}} \text{sign}(M)  wt([M-M_0]) (z,w)^{[M-M_0]},
\]
where $\text{sign}(M) \in \{\pm 1\}$ is a sign that depends on the homology class $[M-M_0]$ and $\kappa$.
\end{theorem}
The Laurent polynomial
\begin{align*}
    P(z,w):=\frac{\text{det }K(z,w)}{wt(M_0) z^{( M_0,\gamma_w )} w^{( M_0,-\gamma_z)}}
\end{align*}
is called the \textit{characteristic polynomial}, and its vanishing locus $C_0:=\{(z,w) \in (\C^*)^2: P(z,w)=0\}$ is called the  \textit{(open) spectral curve}. Note that while the Kasteleyn matrix depends on the choice of 1-cocycles representing the cohomology classes $wt,\kappa$ and the choice of a reference matching $M_0$, the spectral curve is independent on these choices. By Proposition \ref{nplemma}, the Newton polygon of $P(z,w)$ coincides with the Newton polygon of $\Gamma$. 

\subsection{A construction of Fock and Marshakov.}\la{sec:fockmarsh}
In this section, we describe the construction of a homomorphism from the group of cluster transformations to an abelian group due to \cite{FM16}*{Section 7.3}. Let $\Z^{\Sigma(1)}_0$ be the group of integer valued functions $f$ on $\Sigma(1)$ such that $\sum_{\rho \in \Sigma(1)} f(\rho)=0$. Let $\langle \cdot,\cdot \rangle_\T :H_1(\mathbb T,\Z) \times H_1(\mathbb T,\Z) \ra \Z$ be the intersection pairing in $\mathbb T$. We have an embedding
\begin{align*}
    j:H_1(\mathbb T,\Z) &\hookrightarrow \Z^{\Sigma(1)}_0\\
    \gamma &\mapsto \left( \sum_{\alpha \in Z_\rho} \langle [\alpha], \gamma \rangle_\T \right)_{\rho \in \Sigma(1)}.
\end{align*}

Let $\Gamma$ be a bipartite torus graph and let $T$ be its triple point diagram. A cluster transformation $\Gamma \ra \Gamma$ is equivalent to a sequence of triple point diagrams 
\be \label{tcdseq}
T=T_0 \ra T_1 \ra \cdot \cdot \cdot \ra T_{n-1} \ra T_n \cong T,
\ee

where $T_{i+1}$ is obtained from $T_i$ by either performing a 2-2 move or $T_{i+1}$ is related to $T_i$ by an isotopy in $\T$. Let $\{\alpha^i\}$ be the set of strands in $T$. The sequence (\ref{tcdseq}) can be interpolated by a one parameter family of curves $\alpha^i(t)$ in $\T$, where $t\in [0,1]$ such that $\alpha^i(0)=\alpha^i$ and such that the intersections remain triple at all but $n-1$ parameter values where we have a quadruple intersection in the course of a 2-2 move. Using the isomorphism of triple point diagrams $T=T_0 \cong T_n$, we glue the end points of the parameter interval $[0,1]$ to get an $S^1$. During the course of the sequence (\ref{tcdseq}), each strand $\alpha$ in $T$ traces out a $2$-chain $S_\alpha:=\{(u,t):u \in \alpha(t), t \in S^1\}$ in $\mathbb T \times S^1$.\\

Let $Z_\rho=\{\alpha^i_\rho\}_{i=1}^{|E_\rho|}$ be the strands in $T$ corresponding to $\rho \in \Sigma(1)$. The cluster transformation maps each strand $\alpha \in Z_\rho$ bijectively to another strand in $Z_\rho$, and therefore $\partial (\sum_i S_{\alpha_\rho^i})=0.$ Moreover, $\sum_\rho \sum_i S_{\alpha_\rho^i}$ is a $2$-boundary: it is the boundary of the $3$-chain in $\T \times S^1$ traced out by the regions of $\T$ corresponding to white vertices of $\Gamma$. Therefore  we have
\be \la{sumzero}
\sum_\rho \sum_i [S_{\alpha_\rho^i}]=0, \text{ in $H_2(\T \times S^1, \Z)$}.
\ee

Let $(\gamma_z,\gamma_w)$ be the  basis for $H_1(\mathbb T,\Z)$ from Figure \ref{bfgraph} and suppose $\gamma_t$ is a generator of $H_1(S^1,\Z)$. By the K\"unneth formula \cite{Hatcher}*{Theorem 3.16 and Example 3.18}, we have $H_2(\T \times S^1,\Z) \cong \Lambda^2 _\Z[\gamma_x,\gamma_z,\gamma_w]$.
If a strand $\alpha^i_\rho \in Z_\rho$ with $[\alpha_\rho^i]=X_\rho \gamma_z+Y_\rho \gamma_w$ is translated by $a_\rho \gamma_z + b_\rho \gamma_w$ during the sequence (\ref{tcdseq}), then
\begin{align}\la{scoord}
[S_{\alpha_\rho^i}]&= (X_\rho \gamma_z + Y_\rho \gamma_w) \wedge (a_\rho \gamma_z + b_\rho \gamma_w + \gamma_t) \nonumber \\
&=(b_\rho X_\rho- a_\rho Y_\rho)\gamma_z \wedge \gamma_w + X_\rho \gamma_z \wedge \gamma_t + Y_\rho \gamma_w \wedge \gamma_t.
\end{align}

Define a function
\begin{align*}
g:\Sigma(1) &\ra \Z\\
\rho &\mapsto |E_\rho| (b_\rho X_\rho-a_\rho Y_\rho).
\end{align*}

 Informally, each zig-zag path in $\Gamma$ is translated in the universal cover to a parallel zig-zag path by the cluster transformation. $g(\rho)$ is the number of steps in the direction of $\rho$ that any zig-zag path in $Z_\rho$ is translated. Writing (\ref{sumzero}) in coordinates using (\ref{scoord}), we get
 \begin{align*}
     \left(\sum_{\rho \in \Sigma(1)} |E_\rho|(b_\rho X_\rho-a_\rho Y_\rho)\right)\gamma_z \wedge \gamma_w +\left(\sum_{\rho \in \Sigma(1)} X_\rho \gamma_z + Y_\rho \gamma_w \right)\wedge \gamma_t=0.
 \end{align*}
 We have $\left(\sum_{\rho \in \Sigma(1)} X_\rho \gamma_z + Y_\rho \gamma_w \right)=0$ because this is the sum of counterclockwise oriented edges of the Newton polygon. Since $\sum_{\rho \in \Sigma(1)} |E_\rho|(b_\rho X_\rho-a_\rho Y_\rho)=0$, we get  
  $g \in \mathbb Z^{\Sigma(1)}_0$. \\

The above construction gives us a group homomorphism $\psi$ defined as the composition
\be \label{psi}
\{\text{Cluster transformations } \Gamma \ra \Gamma\} \ra \Z^{\Sigma(1)}_0 \ra \mathbb Z^{\Sigma(1)}_0/j H_1(\T,\Z).
\ee
Fock and Marshakov \cite{FM16} conjectured that $\psi$ gives an isomorphism of the (2-2) cluster modular group with $\mathbb Z^{\Sigma(1)}_0/j H_1(\T,\Z).$ We will prove this by showing in Section \ref{sectionsurj} that $\psi$ is surjective, and that the kernel of $\psi$ consists precisely of trivial cluster transformations in Section \ref{sec:triv}.
\subsection{Algebraic geometry background}\label{subsection_ag}

Throughout this paper, the main reference for the algebraic geometry concepts we will use is Hartshorne's book \cite{Hart}.
We will be mainly dealing with normal projective surfaces (see \cite{Bea} or \cite[Chapter V]{Hart} for a reference): up to removing a finite set of points (the \emph{singular locus}), one can think of them as 2-dimensional complex manifolds, embedded in $\mathbb{P}^n$ (that for us will be $\mathbb{P}^n_{\mathbb{C}}$). Similarly, a \emph{curve} will be a purely 1-dimensional projective variety (for example, the locus where $X^2Z = Y^3$ in $\mathbb{P}^2$). A smooth curve is just a compact Riemann surface. We now introduce some notations and a definition that will be useful later.

\begin{notation}
If $X$ is a scheme with a sheaf $\mathcal{F}$ on $X$ and $i\in \mathbb{N}$, we will denote by $h^i(\mathcal{F}):=\operatorname{dim}_{\mathbb{C}}(H^i(X,\mathcal{F}))$.
\end{notation}
\begin{definition}
A surface $X\subseteq \mathbb{P}^n$ is \textit{ruled by lines} if for every point $p\in X$ there is a line of $\mathbb{P}^n$ passing through $p$.
\end{definition}

\subsubsection{Line bundles and divisors on curves}\la{sec:curves}
In this Section, we summarize some results on  algebraic curves that we will need in Section \ref{sec:tct}. For further details, see \cite{ACGH}*{Chapter I}. By a \textit{curve} $C$, we mean a one dimensional projective variety. Generally we will deal with smooth curves i.e. compact Riemann surfaces. The key to studying the geometry of $C$ is to understand rational (i.e. meromorphic) functions on it, which leads to the notions of line bundles and divisors. A \textit{(Weil) divisor} on $C$ is a formal linear combination of points in $C$, that is a sum of the form
\[
D=\sum_{i}n_i p_i, \quad n_i \in \Z, p_i \in C.
\]
The number $n_i$ is called the \textit{multiplicity} of $p_i$ in $D$. The divisors in $C$ form a group under addition, graded by the \textit{degree} homomorphism, defined by
\[
\text{deg}(D)=\sum_i n_i.
\]
If $f$ is a rational function on $C$, it defines its divisor of zeroes and poles
\[
\text{div }f=\sum_{p \in C} \text{ord}_p(f) p,
\]
where $\text{ord}_p(f)$ is the order of vanishing of $f$ at $p$. Such divisors are called \textit{principal divisors} and are always of degree $0$. Two divisors $D$ and $D'$ are said to be \textit{linearly equivalent} if their difference is a principal divisor. The group $\text{Cl}(C)$ of divisors modulo linear equivalence is called the \textit{divisor class group} of $C$. Note that since principal divisors have degree $0$, the degree homomorphism descends to $\text{Cl}(C)$. We denote by $\text{Div}^d(C)$ (resp. $\text{Cl}^d(C)$) the set of degree $d$ divisors (resp. divisor classes). 

Associated to the divisor $D$ is the line bundle $\mathcal O_C(D)$ on open $U \subset C$ by
\[
H^0(U,\mathcal O_C(D)):=\{\text{Rational functions on $U$ such that }\restr{\text{div }f + D}{U} \geq 0\}.
\]
Here we are making the standard identification of a line bundle with its locally free sheaf of sections. Define the \textit{Picard group} $\text{Pic}(C)$ as the group of line bundles on $C$ with the group operation given by tensor product. The map $D \mapsto \mathcal O_C(D)$ is a group isomorphism of $\text{Cl}(C)$ with $\text{Pic}(C)$. 

Let $K_C$ denote the \textit{canonical divisor} class of $C$ i.e. the divisor class such that $\mathcal O_C(K_C)$ is the cotangent line bundle of $C$. It is a basic fact that $h^0(C,K_C) = g$, where $g = \frac 1 2 \text{rank} H_1(C,\Z)$ is the \textit{genus} of $C$. Let $\omega_1,\dots,\omega_g$ be a basis for the space of $1$-forms $H^0(C,K_C)$. We define the \textit{period map} 
\begin{align*}
   \pi: H_1(C,\Z) & \ra \C^g\\
    \sigma &\mapsto \left( \int_\sigma \omega_i \right)_{i=1}^g.
\end{align*}
The \textit{Jacobian} $J(C)$ of $C$ is the complex torus $\C^g/H_1(C,\Z)$.

Fix a base point $p_0$ and define the \textit{Abel map}
\begin{align*}
    u:C &\ra J(C)\\
    p &\mapsto \left(\int_{p_0}^p \omega_i \right)_{i=1}^g,
\end{align*}
where the integral is over an arbitrary path from $p_0$ to $p$. Since we quotient out $H_1(C,\Z)$ in $J(C)$, the map $u$ is well-defined. The definition of the Abel map extends to divisors by linearity. We have:
\begin{theorem}[Abel's theorem] \la{thm:abel}
 Two divisors $D$ and $D'$ are linearly equivalent if and only if $u(D)=u(D')$.
\end{theorem}
As a consequence of Abel's theorem, we get that the Abel map $u:\text{Div}^d(C) \ra J(C)$ factors through an injective map $\phi:\text{Cl}^d(C) \ra J(C)$ for all $d$. We call a divisor \textit{effective} if it has nonnegative multiplicity at each point of $C$. We denote by $C^{(d)}$ the set of effective divisors of degree $d$. We have:
\begin{theorem}[Jacobi inversion theorem] \la{thm:jacobi}
 The Abel map $u:C^{(g)} \ra J(C)$ is surjective and birational.
\end{theorem}
In other words, given a generic point $q$ in $J(C)$, there is a unique degree $g$ effective divisor $D$ such that $u(D)=q$. As a consequence, we also see that the map $\phi:\text{Cl}^d(C) \ra J(C)$ is a bijection for each $d$.

In what follows, we will often use \emph{ample} and \emph{very ample} line bundles. These line bundles provide an intrinsic way to understand projective embeddings. We briefly introduce them here, and we refer the reader to \cite{Hart} for a more complete reference. A very ample line bundle $L$ on $X$ is a line bundle such that there exists an embedding $i: X\hookrightarrow \mathbb{P}^n$ for a certain $n$, such that $i^*\mathcal{O}_{\mathbb{P}^n}(1) \cong L$. An ample line bundle is a line bundle such that a positive tensor power of it is very ample. Given an embedding $i: X\hookrightarrow \mathbb{P}^n$, a \emph{hyperplane section} of $i$ is the zero locus of a section $i^*H\in H^0(X,i^*\mathcal{O}_{\mathbb{P}^n}(1))$, where $H\in H^0(X,\mathcal{O}_{\mathbb{P}^n}(1))$. Geometrically, the hyperplane section $i^*H$ is the intersection of $X$ with the hyperplane $H$. If $x_0,x_1,\dots,x_n$ are homogeneous coordinates on $\P^n$, then for example we can take $H=x_0$, so the locus of points in $X$ that map to points of the form $[0,a_1,\dots,a_n]  \in \P^n$ is a hyperplane section.

\subsubsection{Toric surfaces}\label{subsection_toric_surfaces}
In this subsection we include some notions that we will use on toric varieties that we will use later. We redirect the reader to the book \cite{CLS11} for a complete treatment.
A \textit{toric surface} $X$ is a normal algebraic surface that contains a torus $(\C^*)^2$ as a dense open subvariety, such that the action of $(\mathbb{C}^{* })^{n}$ by multiplication on itself extends to an action of $(\mathbb{C}^{*})^{2}$ on $X$. For example, $\mathbb{P}^2$ is a toric variety. Indeed the dense torus $(\C^*)^2 \subset \mathbb{P}^2$ is the set of points of the form $[a_0,a_i,a_2]$ such that $a_0,a_1,a_2 \in \C^*$. Another example is $\P^1 \times \P^1$, whose dense torus is the set of points of the form $([a_0,a_1],[b_0,b_1])$ with $a_0,a_1,b_0,b_1 \in \C^*$.

In what follows we will only be interested in normal and projective toric surfaces. We denote by $M$ the group of characters of $(\mathbb{C}^{*})^{2}$, i.e. the group of homomorphisms $(\mathbb{C}^{*})^{2} \to \mathbb{C}^{*}$ (for us $M$ will be $H_1(\T,\Z)$). Then $M$ is isomorphic to $\mathbb{Z}^2$, with the isomorphism sending $(a_1,a_2) \in \mathbb{Z}^2$ to the homomorphism sending $(\lambda_1,\lambda_2)\mapsto \lambda_1^{a_1}\cdot \lambda_2^{a_2}$.

Given a set of characters $\chi_0,\dots,\chi_m$ of $M$, we have a morphism $(\mathbb{C}^*)^{ n} \to \mathbb{P}^m$ sending $x\mapsto [\chi_0(x),\dots,\chi_n(x)]$. In particular, for every convex integral polygon $N\subseteq M\otimes \mathbb{R}$, we can take the set of characters to be the lattice points contained in $N$. This gives a morphism $(\mathbb{C}^{*})^{ n}\to \mathbb{P}^m$ as above, where $m+1$ is the number of lattice points contained in $N$. The closure of the image of $(\mathbb{C}^{*})^{ 2} \to \mathbb{P}^m$ is a toric surface
(the image of the map $(\mathbb{C}^{*})^{ 2}\to \mathbb{P}^m$ is an open subset of its closure)\cite{CLS11}*{Proposition 2.1.2}. Moreover, since every convex integral polygon is very ample, the toric surface defined above is normal \cite{CLS11}*{Corollary 2.2.19}. Therefore it has isolated singularities, as normal varieties are smooth in codimension one. 

\begin{remark}\label{remark_sections_are_lattice_pts}
We can also understand the previous paragraph also as follows. Consider the action of $(\mathbb{C}^{*})^{ 2}$ on $\mathbb{P}^m$ defined as $t*[a_0,\dots,a_m]:=[\chi_0(t)\cdot a_0,\dots,\chi_m(t)\cdot a_m]$. Our toric variety is the closure of the orbit of $[1,\dots,1]$. With this action of
$(\mathbb{C}^{*})^{ 2}$ on $\mathbb{P}^m$, the sections $x_i$ of $H^0(\mathbb{P}^m, \mathcal{O}_{\mathbb{P}^m}(1))$ are $(\mathbb{C}^{*})^{ 2}$-equivariant (they have character
$\chi_i$). In particular, if we pull-back the sections $\{X_i\}_{i=0}^n\subseteq H^0(\mathcal{O}_{\mathbb{P}^m}(1))$ to the torus, they correspond to monomials of the form $p_i = z^{a_i}w^{b_i} \in \mathbb{C}[z^{\pm 1},w^{\pm 1}]$. Then the set of points $\{(a_i,b_i)\}_{i=0}^n$ are the lattice points of $N$. 
\end{remark} 

Therefore, a convex integral polygon $N$ gives rise to a projective toric surface $X_N$, along with an ample divisor $D_N$, such that $H^0(X_N,D_N)$ is the vector space of Laurent polynomials with Newton polygon {contained in} $N$. Therefore the linear system $|D_N|$ is identified with curves defined by vanishing of Laurent polynomials with Newton polygon {contained in} $N$. We will require the following two facts:
\begin{itemize}
    \item A generic curve $C = V(P)$ for $P\in H^0(X_N,D_N)$ has genus $g$ equal to the number of interior lattice points in $N$ (see \cite[Proposition 10.5.8]{CLS11}). 
    \item The complement of the algebraic torus in $X_N$ is a union of $\mathbb P^1$s, called \textit{lines at infinity}, parameterized by the edges of $N$, and intersecting according to the combinatorics of $N$ \cite{CLS11}*{Theorem 3.2.6}.
\end{itemize}

  In what follows, we will denote the line at infinity corresponding to $E_\rho \in E_N$ by $D_\rho$. For $C \in |D_N|$, we have $|C \cap D_\rho|=|E_\rho|$, where the points in $C \cap D_\rho$ are counted with multiplicity.

\begin{remark}
A subpolygon of a polygon induces a rational map of the associated toric surfaces. Indeed given $\{\chi_0,\dots,\chi_m\}$ characters of $(\mathbb{C}^{*})^{ 2} $, and given $1<k\le m$, we can consider the two maps $(\mathbb{C}^{*})^{ 2} \to \mathbb{P}^m$ and $(\mathbb{C}^{*})^{ 2} \to \mathbb{P}^k$ where the first one is induced by $\{\chi_0,\dots,\chi_m\}$ and the second one by $\{\chi_0,\dots,\chi_k\}$. There is a rational map $\mathbb{P}^m\dashrightarrow \mathbb{P}^k$ that sends $[a_0,\dots,a_m]\mapsto [a_0,\dots,a_k]$, that makes the following diagram commutative:
$$\xymatrix{ & \mathbb{P}^m\ar@{..>}[dd]\\(\mathbb{C}^{*})^{ 2} \ar[ru]\ar[rd] & \\ & \mathbb{P}^k }$$
\end{remark}

\section{Surjectivity of $\psi$}\label{sectionsurj}

In this section we show that the group homomorphism $\psi$ of Fock and Marshakov defined in (\ref{psi}) is surjective. Given an element of $f \in \Z^{\Sigma(1)}_0$, we will construct a cluster transformation $t_f$ such that $\psi(t_f)=f$.

\subsection{A construction of Goncharov and Kenyon}
We recall the construction of minimal bipartite torus graphs with Newton polygon $N$ from \cite{GK12}. We require that the graph has two  additional properties that are not explicitly mentioned in \cite{GK12}, but are immediate consequences of the construction. Suppose the torus $\mathbb T$ is constructed by gluing opposite sides of a rectangle $R$. We label the north, west, south and east sides of $R$ by $\partial R_N,\partial R_W,\partial R_S,\partial R_E$ respectively. For each ray $\rho \in \Sigma(1)$, let $X_\rho \gamma_z+Y_\rho \gamma_w$ the primitive edge vector in the direction of $E_\rho$, where $\gamma_z,\gamma_w$ are the generators of $H_1(\T,\Z)$ that are given by the sides of $R$ oriented as in Figure \ref{bfgraph}. For each $\rho \in \Sigma(1)$, draw loops $\{\alpha_\rho^i\}_{i=1}^{|E_\rho|}$ in $\T$, each with homology class $X_\rho \gamma_z+Y_\rho \gamma_w$ so that the total number of intersections of any loop with the boundary of $R$ is minimal. Isotope the loops in $\T$ so that:
\begin{enumerate}
    \item The intersections of the loops with each side of $R$ alternate in orientation, ``in" and ``out".
    \item The west-most point on $\partial R_N$ is an ``out" point.
    \item We do not introduce any new intersection points of loops with $\partial R$ during the isotopy.
\end{enumerate}
Using Theorem \ref{TT}, we can isotope the loops in $R$ to obtain a minimal triple crossing diagram in $R$ with the same boundary matching. Using the procedure outlined in Section \ref{tpdt2}, we convert it to a minimal bipartite torus graph. 
\begin{proposition}[\cite{GK12}]\label{gkpropn}
For a convex integral polygon $N$, there is a minimal bipartite torus graph $\Gamma$ with Newton polygon $N$ satisfying:
\begin{enumerate}
    \item The west-most intersection point of a strand with $\partial R_N$ is an ``out" point.
    \item The number of intersections of each zig-zag path with the boundary of $R$ is the smallest possible for a minimal triple point diagram with Newton polygon $N$.
\end{enumerate}
\end{proposition}

We require the following lemma that is contained in the proof of \cite{GK12}*{Theorem 2.5}. We include the proof of the second statement, because it is short and illustrative of the type of arguments we will make later.

\begin{figure}

  \centering
  \includegraphics[width=0.6\linewidth]{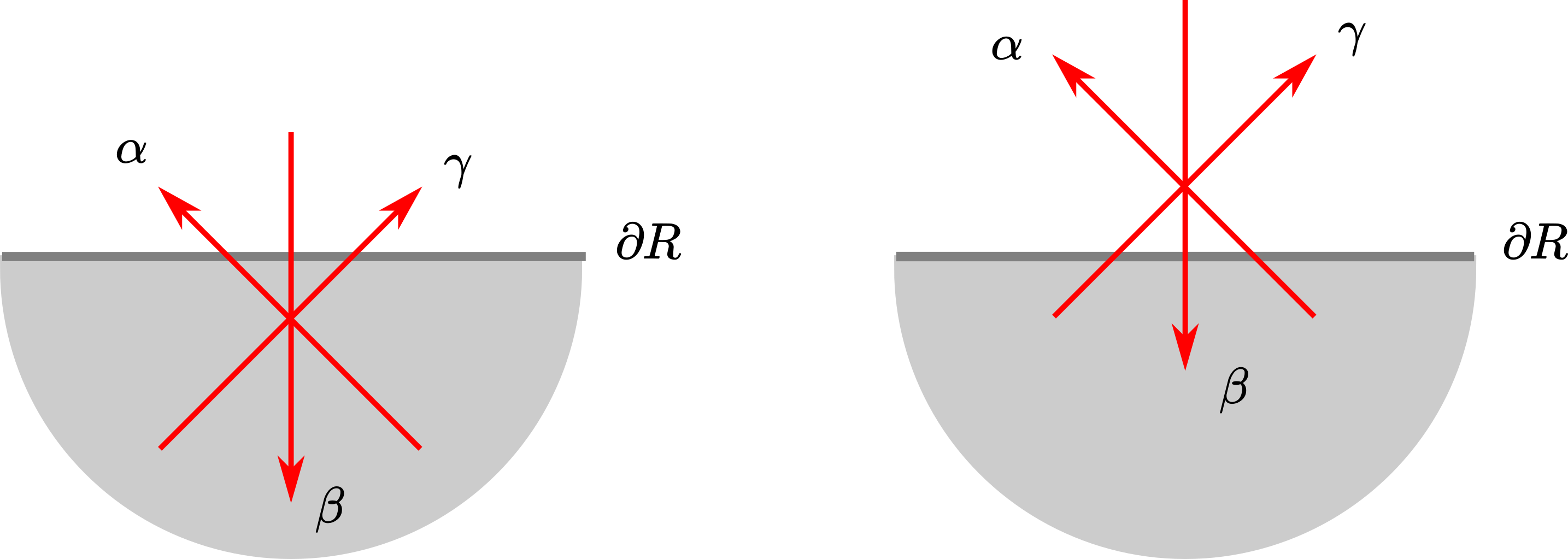}

\caption{Permuting boundary intersections.}
\label{permute}
\end{figure}

\begin{lemma}\label{exch}
Suppose $T$ is a triple point diagram in $\T$. The relative order along the boundary of $R$ of strands associated to the same ray of $\Sigma$ is fixed. The relative order of two incoming or outgoing strands associated to different edges of $N$ can be interchanged by $2$-$2$ moves and isotopy.
\end{lemma}
\begin{proof}

Suppose $\alpha$ and $\gamma$ are two consecutive ``out'' strands in $T$ that correspond to different rays of $\Sigma$. Then by the alternating property, there is an ``in'' strand $\beta$ of $T$ between them. Since $\alpha$ and $\gamma$ belong to different rays of $\Sigma$, they must cross at a triple point inside $R$. By Proposition \ref{TT2}, there is a triple point diagram $T'$ in which the three strands $\alpha, \beta, \gamma$ meet at a triple point just adjacent to the boundary. By Theorem \ref{TT}, we can use $2$-$2$ moves and isotopy to convert $T$ into $T'$. Then we isotope this triple point across the boundary $\partial R$, which permutes boundary intersections of $\alpha$ and $\gamma$, as illustrated in Figure \ref{permute}.
\end{proof}

\paragraph{Change of basis for $H_1(\Gamma,\Z)$.}
Let $X_\rho \gamma_z + Y_\rho \gamma_w$ be the homology class of a zig-zag path in $Z_\rho$ in the basis $(\gamma_z,\gamma_w)$ of $H_1(\T,\Z)$ from Figure \ref{bfgraph}. Changing the basis, or equivalently, changing the fundamental rectangle $R$ of $\T$ corresponds to the action of $SL(2,\Z)$ on $H_1(\T,\Z)$. 
$SL(2,\Z)$ is generated by 
$$
{\bf S}:=\begin{pmatrix}0&-1\\1&0 \end{pmatrix}, \qquad   {\bf T}:=\begin{pmatrix}1&-1\\0&1 \end{pmatrix}.$$ 
Let ${\bf g}\cdot R$ denote the fundamental parallelogram with boundary formed by the vectors ${\bf g} \cdot \gamma_z$ and ${\bf g}\cdot \gamma_w$. We describe the action of some elements of $SL_2(\Z)$ explicitly.
\begin{enumerate}
    \item In the basis $({\bf S}\cdot \gamma_z, {\bf S}\cdot \gamma_w)$, the vector $a \gamma_z + b \gamma_y$ has coordinates
    $$
    {\bf S}^{-1}\begin{pmatrix}a\\b \end{pmatrix}=\begin{pmatrix}b\\-a \end{pmatrix}.
    $$
    
    Therefore the new coordinates are obtained from the old coordinates by rotating clockwise by $\frac{\pi}2$.
    
    \item In the basis $({\bf T}\cdot \gamma_z, {\bf T}\cdot \gamma_w)$, the vector $a \gamma_z + b \gamma_y$ has coordinates
    $$
    {\bf T}^{-1}\begin{pmatrix}a\\b \end{pmatrix}=\begin{pmatrix}a+b\\b \end{pmatrix}.
    $$
    Therefore ${\bf T}$ is a shear mapping.
    
    \item Define 
    $${\bf U}:=-{\bf T}{\bf S} {\bf T}=\begin{pmatrix}1&0\\-1&1 \end{pmatrix}.$$
    In the basis $({\bf U}\cdot \gamma_z, {\bf U}\cdot \gamma_w)$, the vector $a \gamma_z + b \gamma_y$ has coordinates
    $$
    {\bf U}^{-1}\begin{pmatrix}a\\b \end{pmatrix}=\begin{pmatrix}a\\a+b \end{pmatrix}.
    $$
    Therefore ${\bf U}$ is also shear mapping.
\end{enumerate}

\subsection{Proof of surjectivity.} \la{sec:surj}
The main result of this section is:
\begin{theorem}\label{surj}
The group homomorphism 
$$
\psi:\{\text{Cluster transformations } \Gamma \ra \Gamma\} \ra \mathbb Z^{\Sigma(1)}_0/j H_1(\T,\Z),
$$
defined in (\ref{psi}) is surjective.
\end{theorem}

The rest of this section is devoted to the proof of Theorem \ref{surj}. Let $\rho, \sigma \in \Sigma(1)$ be two consecutive rays in counterclockwise cyclic order. Since the functions $\delta_{\rho}-\delta_{\sigma}$ generate $\Z^{\Sigma(1)}_0$, it suffices to show that there is a cluster transformation $t$ such that $\psi(t)=\delta_{\rho}-\delta_{\sigma}$.\\

  Let $(X_\rho,Y_\rho)$ and $(X_\sigma,Y_\sigma)$ be the homology classes of strands in $Z_\rho, Z_\sigma$ respectively in the basis $(\gamma_z,\gamma_w)$. Changing the basis by repeatedly using ${\bf T}$ or ${\bf U}$, we may assume that $(X_\rho,Y_\rho)$ is neither horizontal nor vertical. Then, rotating if necessary using ${\bf S}$, we can assume that $X_\rho,Y_\rho>0$. Now making another change of basis by repeatedly using ${\bf T}$ or ${\bf U}$, we may assume that $(X_\sigma,Y_\sigma)$ is not horizontal or vertical either. For example, if $(X_\rho,Y_\rho)=(0,-1)$ and $(X_\sigma,Y_\sigma)=(1,0)$, we can do the following change of basis:
 $$
 (0,-1), (1,0) \xmapsto[]{{\bf T}} (-1,-1),(1,0)  \xmapsto[]{{\bf S}^2}(1,1),(-1,0) \xmapsto[]{{\bf U}}(1,2),(-1,-1).
 $$
 
The strategy of the proof is similar to the proof of lemma \ref{exch}. We create a simple configuration of strands near the boundary of $R$ using isotopy and $2-2$ moves, and then push this configuration past $\partial R$.\\

Using Proposition \ref{gkpropn}, we obtain a minimal triple point diagram $\mathfrak T$ in a fundamental rectangle $R$ of $\T$ such that:
\begin{enumerate}
    \item $(X_\rho,Y_\rho) \in \Z_{>0}^2$.
    \item $X_\sigma,Y_\sigma \neq 0$.
\end{enumerate}

Since in what follows we will have occasion to deal with strands in both $\T$ and $R$, let us call strands in $\T$ zig-zag loops and reserve the term ``strand" for strands in $R$, to avoid confusing the two notions. The strands in $R$ are the components of the intersections of zig-zag loops with the interior of $R$. Let $U_\rho$ denote the set of strands whose zig-zag loops correspond to the edge $E_\rho$ of $N$. By minimality of $\mathfrak T$, two strands in $U_\rho$ do not intersect and therefore the partition of the boundary intersection points by the strands in $U_\rho$ is a ``parallel crossing". Therefore there is a (strict) linear order $<_\rho$ on $U_\rho$, where strands are ordered from smallest to largest in the direction of the ray $\rho$. Let us denote by $\alpha$ the $<_\rho$-largest strand in $U_\rho$. Similarly let $\beta$ be the $<_\sigma$-smallest strand in $U_\sigma$. Since $(X_\rho,Y_\rho) \in \Z_{>0}^2$, the strand $\alpha$ is the north-west-most among all strands corresponding to $\rho$. 

\begin{lemma}\la{westtonorth}
The strand $\alpha$ has its ``in" boundary point on $\partial R_W$ and its ``out" boundary point on $\partial R_N$.
\end{lemma}
\begin{proof}
Since $X_\rho,Y_\rho>0$, there is a strand associated to $\rho$ that intersects $\partial R_N$ and a strand associated to $\rho$ that intersects $\partial R_W$. By assumption, $\alpha$ is the north-west-most strand associated to $\rho$, and therefore both of its end points are in $\partial R_N \cup \partial R_W$. Its end points cannot both be on the same side of the boundary of $R$, because the zig-zag loop containing $\alpha$ has smallest possible number of intersections with $\partial R$ (property 2 in Proposition \ref{gkpropn}). Since $X_\rho,Y_\rho>0$, its ``in" boundary point must be on $\partial R_W$ and its ``out" boundary point must be on $\partial R_N$ (again by property 2 in Proposition \ref{gkpropn}). 
\end{proof}

\begin{lemma}\la{keylem}
Starting from $\mathfrak T$ and using 2-2 moves and isotopy in $\T$, we can obtain a new triple point diagram $\mathfrak S$ in $\T$, such that:
\begin{enumerate}
    \item The strands in $U_\rho$ have been cyclically shifted in the direction of $\rho$ (so that $\alpha$ is now $<_\rho$-smallest).
    \item The strands in $U_\sigma$ have been cyclically shifted in the direction of $-\sigma$ (so that $\beta$ is now $<_\sigma$-largest).
    \item The linear orders  of strands corresponding to all other rays are unchanged. 
\end{enumerate}

\end{lemma}

 \begin{proof}
By using lemma \ref{exch}, we can permute the boundary points to make the intersection points of $\alpha$ with $\partial R$ the north-most ``in" point in $\partial R_W$ and the west-most ``out" point in $\partial R_N$. By property 1 in Proposition \ref{gkpropn}, the west-most intersection point of a strand in $\mathfrak T$ with $\partial R_N$ is an ``out" point. Therefore the end-points of $\alpha$ are the north-most intersection point in $\partial R_W$ and the west-most intersection point in $\partial R_N$ respectively. Now we have to deal with four cases, depending on which quadrant $(X_\sigma,Y_\sigma)$ lies in.

\begin{enumerate}

\begin{figure}
\begin{subfigure}{.5 \textwidth}
  \centering
  \includegraphics[width=.5\linewidth]{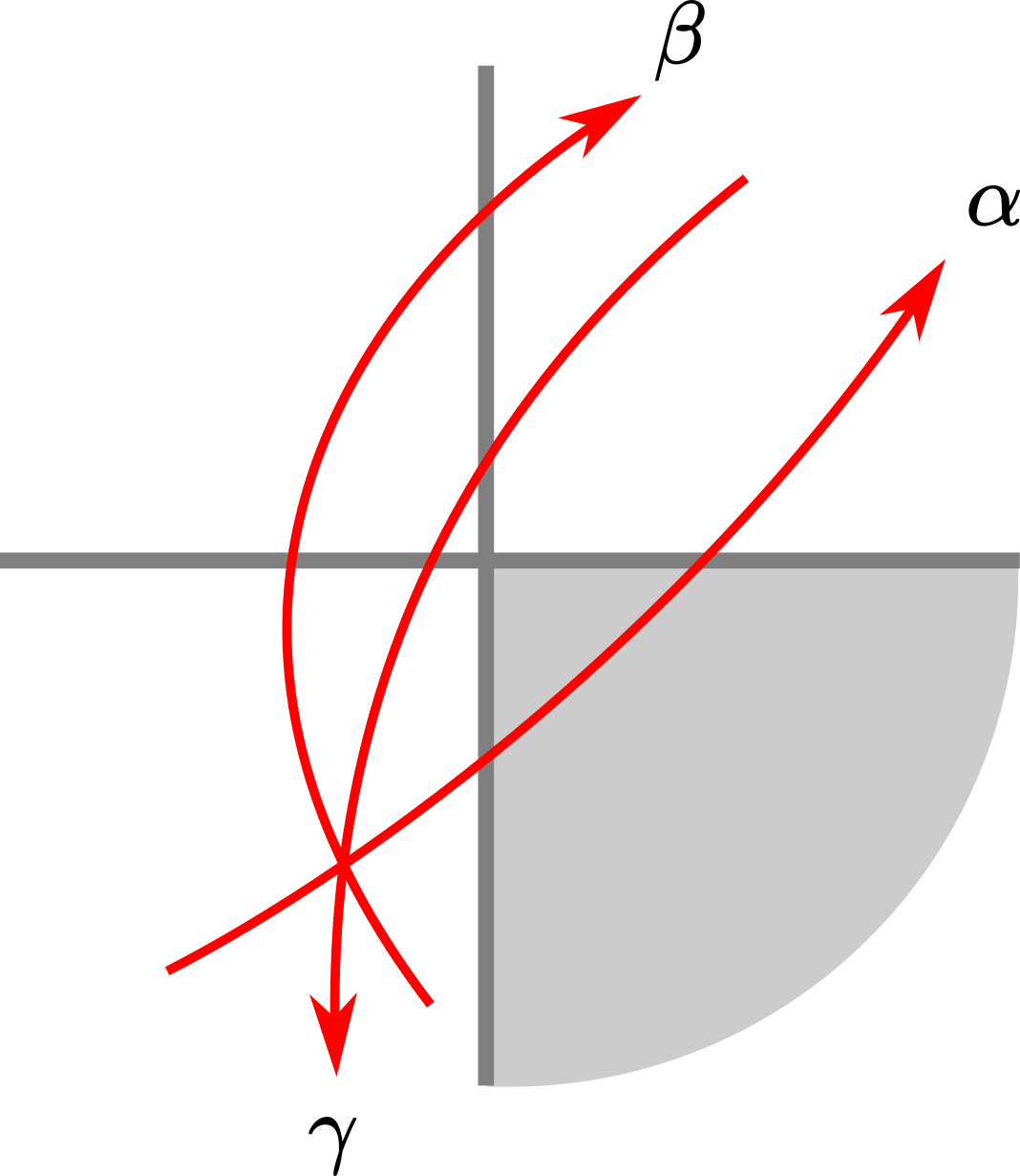}
  \caption{Initial configuration.}
  \label{before1}
\end{subfigure}%
\begin{subfigure}{.5\textwidth}
  \centering
  \includegraphics[width=.5\linewidth]{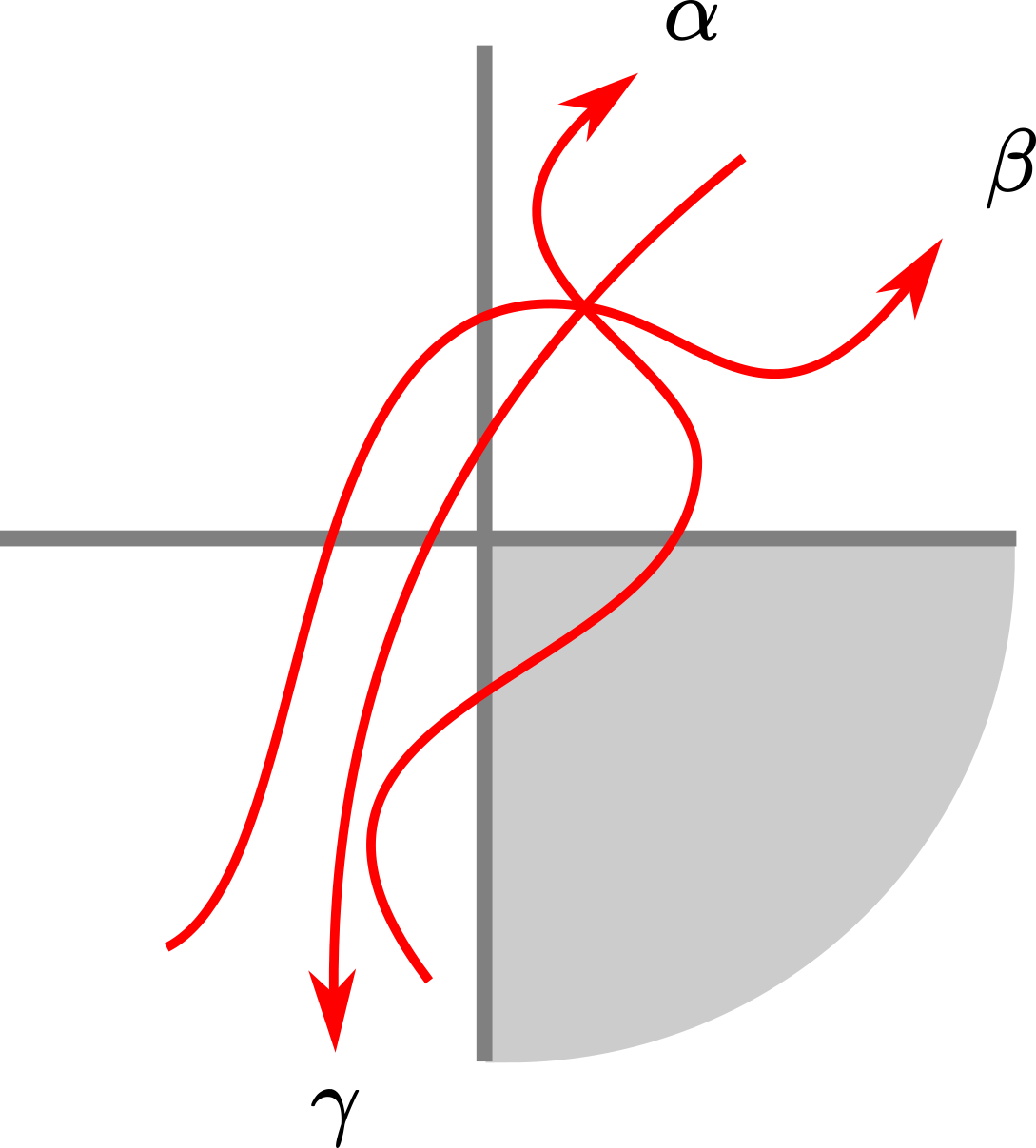}
  \caption{Configuration after isotopy.}
  \label{after1}
\end{subfigure}
\caption{Isotoping the local configuration of strands past the northwest corner of $R$ in case 1.}
\end{figure}
    \item  $X_\sigma,Y_\sigma>0$.
    
  Since $N$ is a closed polygon, there must exist a ray $\tau \in \Sigma(1)$ such that if $(X_\tau,Y_\tau)$ are the coordinates of a zig-zag path in $Z_\tau$, we have $Y_\tau<0$. Making a change of basis using ${\bf T}$, we can further assume $X_\tau<0$ without affecting the assumptions already in place. Since $X_\sigma,Y_\sigma>0$, the strand $\beta$ is the south-east-most among all strands associated to $\sigma$. By an argument similar to the proof of Lemma \ref{westtonorth}, $\beta$ has its ``out" point on $\partial R_N$ and ``in" point on $\partial R_W$. Permuting boundary points using Lemma \ref{exch}, we make the intersections of $\beta$ with $\partial R$ the south-most ``in" point in $\partial R_E$ and the east-most ``out" point in $\partial R_N$.\\

Since the total homology of all zig-zag loops is zero, the total intersection number of the loops with any side of $R$ is zero, that is, we have an equal number of ``in" and ``out" points in any side of $R$, alternating in orientation as we move along the side. By our assumptions on $\alpha$ and $\mathfrak T$, the intersection point of $\alpha$ with $\partial R_N$ is the west-most point in $\partial R_N$ and its orientation is ``out". Therefore, the east-most point in $\partial R_N$ is an ``in" point, which means there is an ``in" point to the east of $\beta$ in $\partial R_N$. For the same reason, there is an ``out" point south of $\beta$ in $R_W$. Permuting boundary intersections using Lemma \ref{exch}, we can make the south-east-most strand $\gamma$ corresponding to $\tau$, which by the argument in Lemma \ref{westtonorth}) has a boundary point on each of these sides, pass through both these points. Using Theorem \ref{TT}, we can make $\gamma$ and $\beta$ run parallel to the boundary. Again using Theorem \ref{TT}, we can make the three strands $\alpha, \beta, \gamma$ meet just adjacent to the northeast corner of $R$ to obtain the local picture shown in Figure \ref{before1}. We isotope the triple point across the corner to obtain the configuration in Figure \ref{after1}. This achieves the shift of cyclic orders for $\rho, \sigma$ without changing the cyclic orders of strands corresponding to other rays.

\item $X_\sigma,Y_\sigma<0$.
\begin{figure}
\begin{subfigure}{.5 \textwidth}
  \centering
  \includegraphics[width=.5\linewidth]{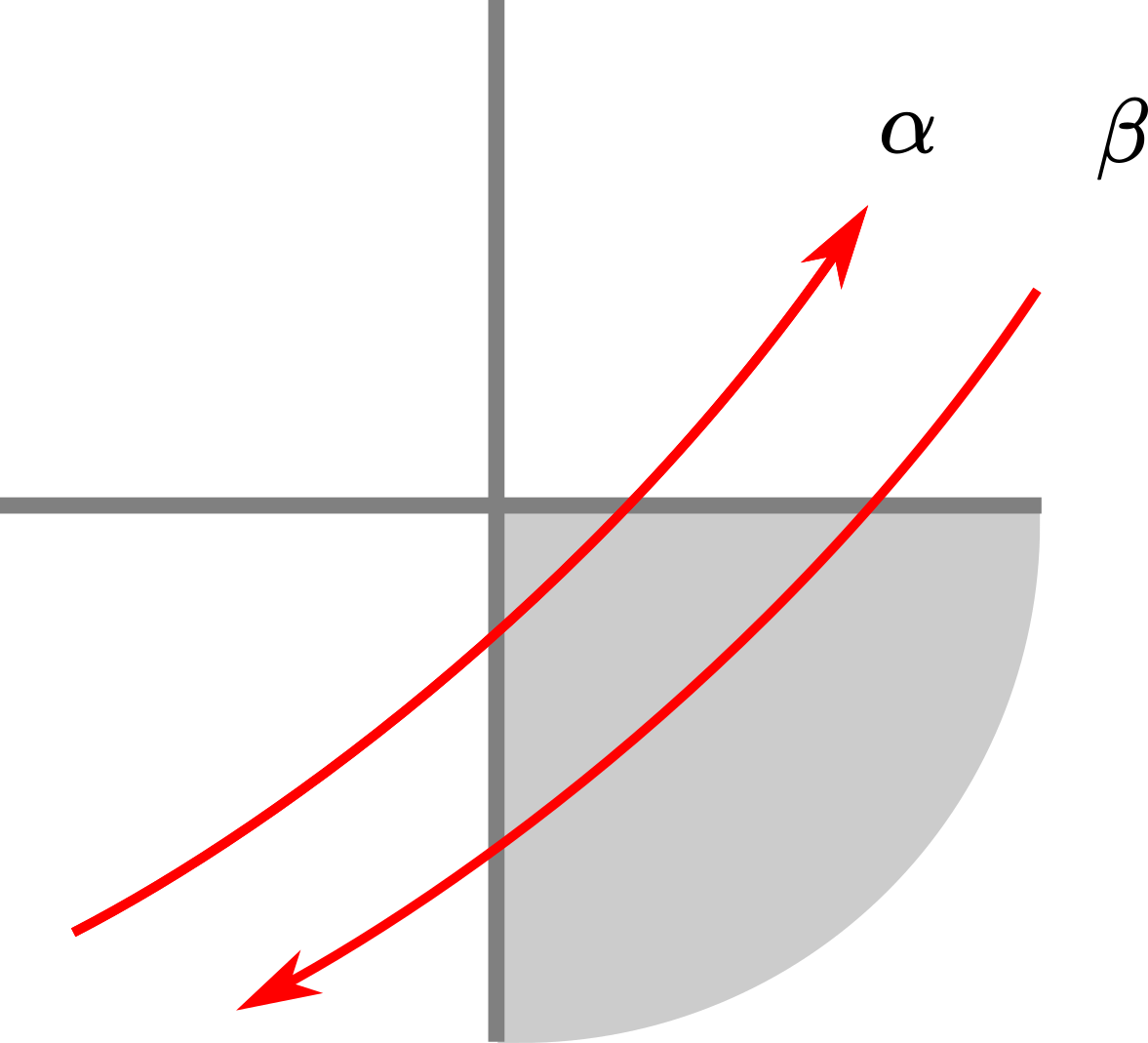}
  \caption{Initial configuration.}
  \label{before2}
\end{subfigure}%
\begin{subfigure}{.5\textwidth}
  \centering
  \includegraphics[width=.5\linewidth]{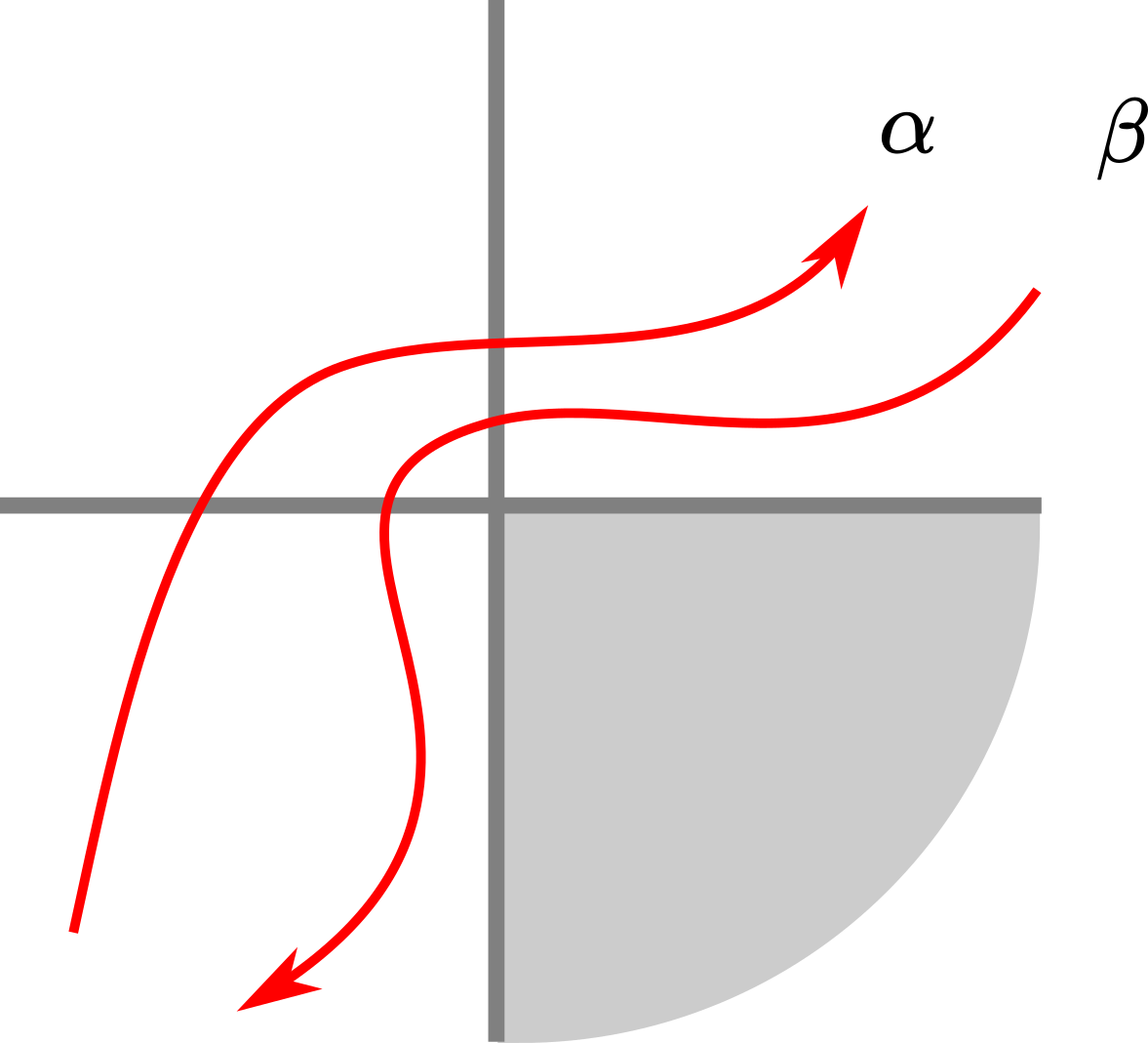}
  \caption{Configuration after isotopy.}
  \label{after2}
\end{subfigure}
\caption{Isotoping the strands past the northwest corner of $R$ in case 2.}
\end{figure}

The strand $\beta$ is the north-west-most among all strands associated to $\sigma$. By the argument in Lemma \ref{westtonorth}, it has an ``in" boundary point on $\partial R_N$ and an ``out" boundary point on $\partial R_N$. Permuting boundary intersections using lemma \ref{exch} we can make the strand $\beta$ the west-most ``in" strand in $\partial R_N$ and the north-most ``out" strand in $\partial R_W$.  Now we use Theorem \ref{TT} to make $\alpha,\beta$ run parallel to the boundary to obtain the local picture shown in Figure \ref{before2} near the northwest corner of $R$. We then isotope to get the configuration in Figure \ref{after2}.\\

\item $X_\sigma<0,Y_\sigma>0$. 

We can use ${\bf T}\in SL(2,\Z)$ to make $X_\sigma>0$, reducing to case 1.\\

\item $X_\sigma>0,Y_\sigma<0$.

This case cannot occur because of convexity of $N$.
\end{enumerate}
 
\end{proof}

\begin{proof}[Proof of Theorem \ref{surj}]
We need to find a sequence of 2-2 moves and isotopy $t:\mathfrak T  \ra \mathfrak T$ such that $\psi(t)=\delta_{\rho}-\delta_{\sigma}$. Using lemma \ref{keylem}, we obtain a triple point diagram $\mathfrak S$ in $\T$. Now we use lemma \ref{exch} to permute the boundary points so that if a pair of ``in" and ``out" points in $\mathfrak T$ is connected by a strand that corresponds to a ray $\tau$, then the corresponding ``in" and ``out" points of $\mathfrak S$ are also connected by a strand corresponding to $\tau$. If $\tau \neq \rho, \sigma$, then this is the same strand as in $\mathfrak T$. When $\tau $ either $\rho$ or $\sigma$, this is a cyclically shifted strand. Let $\mathfrak U$ be the triple point diagram in $R$ thus obtained from $\mathfrak S$. Now we apply Theorem \ref{TT} to convert $\mathfrak U$ to $\mathfrak T$ using a sequence of 2-2 moves and isotopy in R. We define $t$ to be the sequence of 2-2 moves and isotopy $\mathfrak T \rightarrow \mathfrak S \ra \mathfrak U \ra \mathfrak T$. By construction, $\psi(t)=\delta_{\rho}-\delta_{\sigma}$, and the theorem is proved.

\end{proof}

\section{Trivial seed cluster transformations}\la{sec:tct}

By Theorem \ref{surj}, the homomorphism $\psi$ is surjective. To complete the proof of Theorem \ref{mainthm2}, we need to find the kernel of $\psi$. 

\paragraph{The spectral transform.}
We follow \cite{GK12}*{Section 7}.  A \textit{spectral data} is a triple $(C,S,\nu)$ where:

\begin{enumerate}
    \item $C$ is a curve in $|D_N|$.
    \item $S$ is a degree $g$ effective divisor in $C$, where $g$ is the number of interior lattice points of $N$.
    \item $\nu=\{\nu_\rho\}_{\rho \in \Sigma(1)}$ is a collection of bijections $\nu_\rho: Z_\rho \xrightarrow[]{\sim} C \cap D_\rho$ (recall that $|E_\rho|=|Z_\rho|=|C \cap D_\rho|$). 
    
\end{enumerate}
Let $\mathcal S_N$ be the moduli space parameterizing the spectral data related to $N$.\\

Fix a minimal bipartite graph $\Gamma$ with Newton polygon $N$, and a white vertex ${\bf w}$ of $\Gamma$. There is a rational map, called the \textit{spectral transform}, defined by Kenyon and Okounkov \cite{KO}, 
\begin{align*}
    \kappa_{\Gamma, {\bf w}}:\mathcal X_N (\C) & \dashrightarrow \mathcal S_N\\
    wt &\mapsto (C,S,\nu),
\end{align*}
as follows:
\begin{enumerate}
    \item $C$ is the closure of $C_0$ in $X_N$, and is called the \textit{spectral curve}. By Theorem \ref{Kastthm}, $C \in |D_N|$. The points in $C \setminus C_0 = \bigcup_{\rho \in \Sigma(1)} C \cap D_\rho$ are called the \textit{points at infinity}.
    \item $S$ is a degree $g$ effective divisor in $C_0$ defined as follows: Consider the following exact sequence of sheaves given by the Kasteleyn operator:
\begin{align*}  0 \ra \bigoplus_{\text b \in B(\Gamma)} \mathcal O_{(\C^*)^2} \xrightarrow[]{K(z,w)}\bigoplus_{\text w \in W(\Gamma)}\mathcal O_{(\C^*)^2} \ra \text{coker }K(z,w) \ra 0.
\end{align*}

    When $C_0$ is smooth, which is true when $wt$ is generic, coker $K(z,w)$ is the pushforward of a line bundle $\mathcal L$ on $C_0$. The image of the section $\delta_{\bf w}$ of $\bigoplus_{\text w \in W(\Gamma)}\mathcal O_{(\C^*)^2}$ in coker $K(z,w)$ restricts to a section of $\mathcal L$. The divisor $S$ is defined to be the divisor of zeroes of this section. It is a degree $g$ effective divisor (see \cite{KO}*{Theorem 1} for a proof when $X_N=\P^2$ and \cite{GGK} for the general case.)
    
    \item $\nu$ is the bijection between zig-zag paths and points at infinity defined by the following property: $\nu(\alpha)$ is the point in $C \cap D_\rho$ where $\restr{K(z,w)}{\alpha}$ is singular, where $\restr{K(z,w)}{\alpha}$ denotes the Kasteleyn matrix of the zig-zag path $\alpha$ viewed as a bipartite graph in $\T$. The coordinates of $\nu(\alpha)$ are determined by $wt(\alpha).$
\end{enumerate}

The following important result was observed by Goncharov and Kenyon in \cite{GK12} and proved by Fock. 
\begin{theorem}[Fock, 2015 \cite{F15}] \label{bir}
The spectral transform is birational.
\end{theorem}

\paragraph{The discrete Abel map.}

Let $\widetilde \Gamma$ be the preimage of $\Gamma$ in the universal cover of $\T$. Let $\text{Div}_\infty(C)$ denote the divisors at infinity of $C$, that is $\Z$-linear combinations of the points at infinity. Following Fock \cite{F15}, we define the \textit{discrete Abel map}
\begin{align*}
{\bf d_0}:\text{Vertices of }\widetilde \Gamma \ra \text{Div}_\infty(C),
\end{align*}
using the following rules: For a choice of white vertex $\text w$ of $\widetilde \Gamma$, we have the normalization ${\bf d_0}({ \text w})=0$, and for any path $\gamma$ from $\text v_1$ to $\text v_2$, we have
\be \la{damprop}
{\bf d_0}(v_2)-{\bf d_0}(v_1)=\sum_{\text{Zig-zag paths }\alpha }\langle \alpha,\gamma \rangle \nu(\alpha),
\ee
where $\langle \cdot , \cdot \rangle$ is the intersection form on the universal cover of $\T$.
${\bf d_0}$ can be effectively computed by the following procedure: If $\text{bw}$ is an edge, with zig-zag paths $\alpha,\beta$ containing $\text{bw}$, then 
$$
{\bf d_0}(\text w)={\bf d_0}(\text b)-\nu(\alpha)-\nu(\beta).
$$

We have an embedding
\begin{align}
H_1(\T,\Z) &\hookrightarrow  Div_\infty(C) \nonumber\\
\gamma &\mapsto \text{div}_{C} (z,w)^\gamma=\sum_{\text{Zig-zag paths }\alpha }\langle \alpha,\gamma \rangle \nu(\alpha), \la{homologydivisor}
\end{align}
 where $(z,w)^\gamma$ denotes the character of $T=(\C^*)^2$ associated to $\gamma \in H_1(\T,\Z)$. ${\bf d}$ is $H_1(\T,\Z)$-equivariant:
 $$
 {\bf d_0}(v+\gamma)={\bf d_0}(v)+\gamma,
 $$
so that although ${\bf d_0}(\text v)$ is not well-defined for a vertex $\text v$ of $\Gamma$, the divisor class $[{\bf d_0}(\text v)]$ is the same for all lifts of $\text v$ to $\widetilde{\Gamma}$ and therefore well-defined. Therefore we define
\begin{align*}
    {\bf d}: V(\Gamma) &\ra \text{Cl}(C)\\
    \text v &\mapsto [{\bf d_0}(\widetilde {\text v})],
\end{align*}
where $\widetilde{\text v}$ is any lift of $\text v$ in $\widetilde \Gamma$, and $\text{Cl}(C)$ is the divisor class group of $C$.

\begin{example} Consider the bipartite torus graph in Figure \ref{bfgraph} with zig-zag paths labeled as in Figure \ref{bfzzpaths}. The discrete Abel map normalized so that ${\bf d}({\bf w})=0$ is as follows.
\begin{align*}
    {\bf d}({\bf w_1})&=[-\nu(\alpha_1)+\nu(\gamma)],\\
     {\bf d}({\bf b_1})&=[\nu(\alpha_2)+\nu(\beta)],\\
     {\bf d}({\bf b_2})&=[\nu(\beta)+\nu(\gamma)],
\end{align*}
where the zig-zag paths are labeled as in Figure \ref{bfzzpaths}.
\end{example}

\paragraph{Elementary transformations and induced discrete Abel maps.}

\begin{figure}
\centering
{\includegraphics[width=0.9\textwidth]{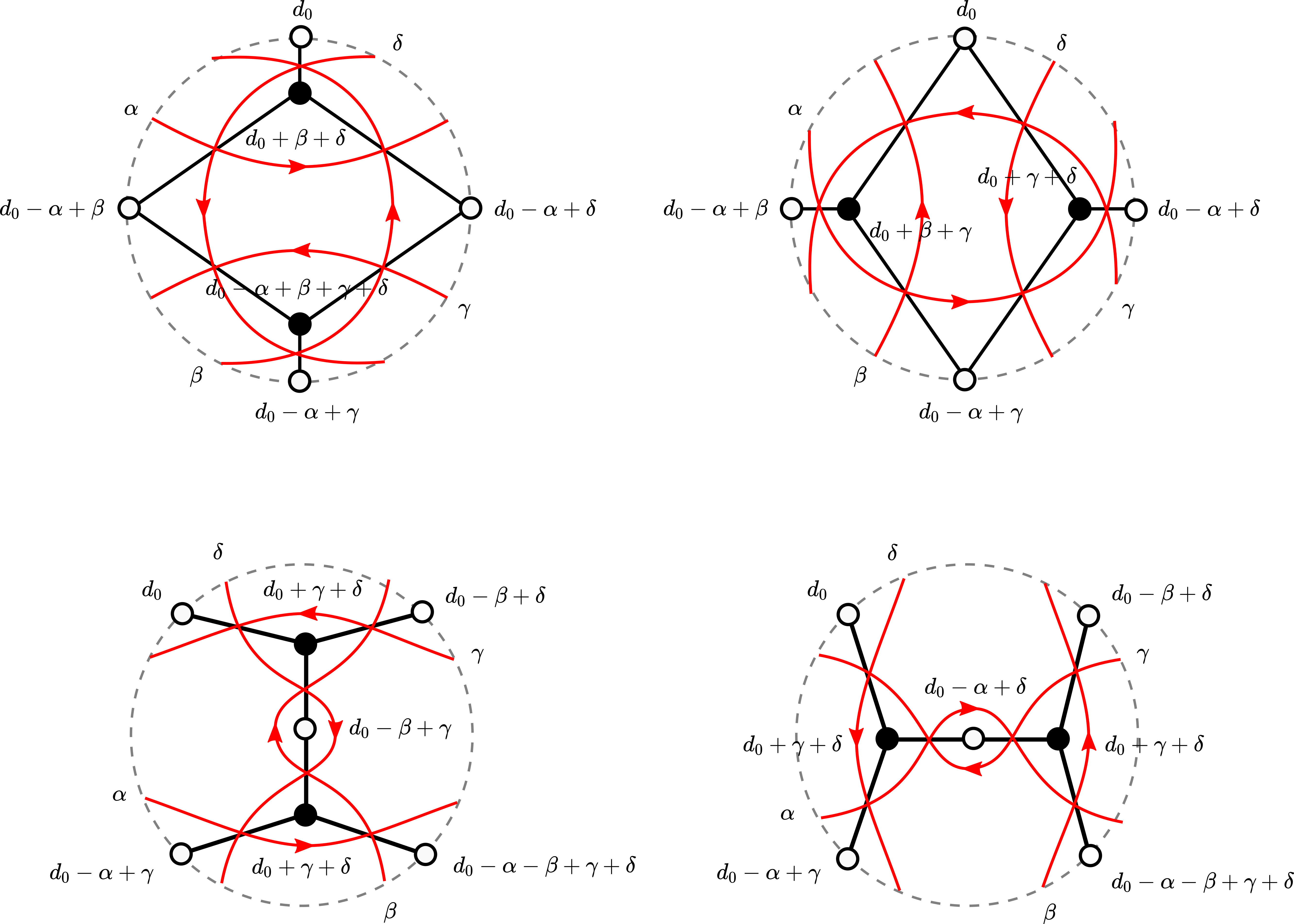}}
\caption{Induced discrete Abel maps.}\label{etdam}
\end{figure}
We now describe how $\nu$ changes under isotopy and elementary transformations and use this to define induced discrete Abel maps. 
\begin{enumerate}
    \item Suppose $s:\Gamma \ra \Gamma$ is an automorphism of $\Gamma$ induced by an isotopy in $\T$. $s$ induces a bijection of the set of zig-zag paths $Z$ with itself that preserves $Z_\rho$. Let 
    \begin{align*}
    \mu_s:\mathcal L_\Gamma &\ra \mathcal L_\Gamma\\
    wt &\mapsto wt \circ s^{-1},
    \end{align*}
    be the induced birational map of weights. If $\alpha$ is a zig-zag path in the graph $\Gamma$ after the isotopy, the point at infinity $\nu_s(\alpha)$ associated with it is determined by $\mu_s(wt(\alpha))=wt(s^{-1}(\alpha))$. Therefore we have $\nu_{s}(\alpha):= \nu(s^{-1}(\alpha)).$\\
    
    If ${\bf d}$ is a discrete Abel map on $\Gamma$, we define an \textit{induced discrete Abel map} ${\bf d}_s$ on $\Gamma$ by the rule: ${\bf d}_s(\text v):={\bf d}(s^{-1}(\text v))$.

    \item Let $s:\Gamma_1 \ra \Gamma_2$ be an elementary transformation and let $Z_1$ and $Z_2$ denote the zig-zag paths of $\Gamma_1$ and $\Gamma_2$ respectively. Let $\mu_s:\mathcal L_{\Gamma_1} \ra \mathcal L_{\Gamma_2}$ be the induced map of weights. $s$ induces a bijection $Z_1 \xrightarrow[]{\sim} Z_2$ between zig-zag paths of $\Gamma_1$ and $\Gamma_2$ such that $\mu_s(wt(\alpha))=wt(s^{-1}(\alpha))$ for all zig-zag paths $\alpha \in Z_2$. Therefore we have $\nu_{s}(\alpha):= \nu(s^{-1}(\alpha)).$\\
    
    Suppose ${\bf d}_1$ is a discrete Abel map on a graph $\Gamma_1$. An elementary transformation $s:\Gamma_1 \ra \Gamma_2$ induces a discrete Abel map ${\bf d}_2$ on $\Gamma_2$ as follows: the elementary transformation only changes $\Gamma_1$ in a disc. The induced discrete Abel map ${\bf d}_2$ is defined to be equal to ${\bf d}_1$ outside the disc, and extended to the interior of the disc using (\ref{damprop}) and $\nu_s$ (see Figure \ref{etdam}).
\end{enumerate}

If $t$ is a sequence of graph isomorphisms and elementary transformations, we get an induced $\nu_t$ and ${\bf d}_t$ by composing.

\paragraph{(2-2) Cluster modular transformations and the spectral transform.}
Let $t: \Gamma \ra \Gamma$ be a (2-2) cluster modular transformation and let $\mu_t$ denote the induced birational automorphism of $\mathcal X_N$. Suppose ${\bf d}$ is a discrete Abel map on $\Gamma$. Let $\nu_t$ be the induced bijction between zig-zag paths and points at infinity and ${\bf d}_t$ the induced discrete Abel map. Any two discrete Abel maps on $\Gamma$ differ only by their normalization, so ${\bf d}-{\bf d}_t$ is a degree zero divisor class of $C$. The following result of Vladimir Fock will play a key role in determining which cluster modular transformations are trivial.

\begin{theorem}[Fock, 2015 \cite{F15}*{Proposition 1}]\label{fockthm}
The following diagram commutes:
\begin{center}
\begin{tikzcd}[column sep = large]
\mathcal X_N \arrow[r, dashed, "\kappa_{\Gamma,{\bf w}}"] \arrow[d,dashed,"\mu_t"] 
& \mathcal S_N \arrow[d,dashed]  \\
\mathcal X_N \arrow[r, dashed, "\kappa_{\Gamma,{\bf w}}"]
& \mathcal S_N
\end{tikzcd},
\end{center}
where the map on the left is $(C,S,\nu) \mapsto (C,S_t,\nu_t)$, where $S_t$ is the degree $g$ effective divisor satisfying
\be
S_t = S+{\bf d}({\bf w})-{\bf d}_t({\bf w}), \text{ in Cl }^g(C). \la{fockcondition}
\ee
By the Jacobi inversion theorem (Theorem \ref{thm:jacobi}), if $S$ is generic, the divisor $S_t$ is uniquely determined by the condition (\ref{fockcondition}).

\end{theorem}
Applying the Abel map $u:\text{Cl}^g(C) \ra J(C)$ to (\ref{fockcondition}), we get $u(S_t) = u(S)+ u({\bf d}({\bf w})-{\bf d}_t({\bf w}))$, which shows that $\mu_t$ becomes a translation by $u({\bf d}({\bf w})-{\bf d}_t({\bf w}))$ in $J(C)$ under the spectral transform.

\subsection{The homomorphism $\psi$ and the discrete Abel map} 
In this section, we prove the following proposition.
\begin{proposition}\label{coruniq}
The birational automorphism $\mu_t$ of $\mathcal X_N$ induced by a cluster transformation $t$ factors through $\psi$:
\begin{center}
    \begin{tikzcd}
\{\text{Cluster transformations } \Gamma \ra \Gamma\} \arrow[r, "\psi"] \arrow[rd,"t \mapsto \mu_t"'] & \mathbb Z^{\Sigma(1)}_0/j H_1(\T,\Z) \arrow[d]\\
& \text{Bir}(\mathcal X_N)
\end{tikzcd},
\end{center}
where $\text{Bir}(\mathcal X_N)$ is the group of birational automorphisms of $\mathcal X_N$.
\end{proposition}

\begin{proof}
We show that the induced discrete Abel map ${\bf d}_t$ and the induced bijection $\nu_t$ for a cluster transformation $t$ are both determined by $\psi(t)$. By Fock's Theorem \ref{fockthm}, the induced birational map $\mu_t$ is determined by ${\bf d}_t$ and $\nu_t$.\\

Let $t:\Gamma=\Gamma_0 \ra \Gamma_1 \ra \cdot \cdot \cdot \ra \Gamma_{n-1} \ra \Gamma_n \cong \Gamma$, be a cluster transformation  where $\Gamma_{i+1}$ is obtained from $\Gamma_i$ by an elementary transformation or $\Gamma_{i+1}$ is isomorphic to $\Gamma_i$ by an isotopy in $\T$. Let $\rho  \in \Sigma(1)$. Let ${\bf d}_0$ be a discrete Abel map on $\widetilde \Gamma$ with the normalization ${\bf d_0}({\bf \widetilde w})=0$, where ${\bf \widetilde w}$ is a chosen lift of ${\bf w}$. Let $\restr{\bf d_0}{C \cap D_\rho}$ denote restriction of the divisor ${\bf d_0}$ to points at infinity associated to $\rho$. The set of all zig-zag paths in the universal cover of $\T$ associated to $\rho$ subdivides the universal cover into a collection of strips $S(a)$ indexed by $a \in \Z^{C \cap D_\rho}$:
$$
S(a):=\restr{\bf d_0}{C \cap D_\rho}^{-1}(a). 
$$

 A elementary transformation or isotopy $s:\Gamma_1 \ra \Gamma_2$ lifts to an $H_1(\T,\Z)$-periodic collection of elementary transformations or isotopy $\widetilde s: \widetilde{\Gamma_1} \ra \widetilde {\Gamma_2}$ in the universal cover of $\T$. Different lifts differ by $H_1(\T,\Z)$, but our construction will be independent of these choices. Suppose ${\bf d}_0$ is a discrete Abel map on $\widetilde{\Gamma_1}$. For $a \in \Z^{C \cap D_\rho}$, let $S_1(a)$ denote the corresponding strip. We define an induced discrete Abel map ${\bf d}_{s,{\bf 0}}$ as follows:
\begin{enumerate}
    \item If $\widetilde s:\widetilde{\Gamma_1} \ra \widetilde{\Gamma_2}$ is induced by an isotopy in $\T$, we define ${\bf d}_{s,{\bf 0}}(\text v):={\bf d}_{\bf 0}(\widetilde s^{-1}(\text v))$.
    \item If $\widetilde s:\widetilde{\Gamma_1} \ra \widetilde{\Gamma_2}$ is an elementary transformation, we define it as in Figure \ref{etdam}.
\end{enumerate}
These are simply $H_1(\T,\Z)$-periodic versions of the induced discrete Abel map on $\T$ defined earlier. By construction, we have the following property: if $S_1(a)$ is the strip whose right boundary is a zig-zag path $\alpha$ of $\widetilde{\Gamma_1}$, then $S_2(a)$ is the strip whose right boundary is $\widetilde s(\alpha)$.\\

Let ${\bf d}_{t,{\bf 0}}$ be the discrete Abel map induced by the cluster transformation $\widetilde t$, obtained by composing. Suppose during $t$ a strand in $Z_\rho$ is translated by $a_\rho \gamma_z+b_\rho \gamma_w$. Then the above property implies that the strip $S_t(a)$ of $\Gamma$ is obtained from $S(a)$ by translating by $a_\rho \gamma_z+b_\rho \gamma_w$. Therefore 
\begin{align}\la{d-dt}
\restr{\left({\bf d}_{t,{\bf 0}} -{\bf d}_{{\bf 0}}\right)}{C \cap D_\rho}=\sum_{\alpha \in Z_\rho}\langle \alpha, \pi_\rho \rangle \nu(\alpha),
\end{align}
where $\pi_\rho$ is any path between a vertex in $S_t(a)$ and a vertex in $S(a)$. Choosing a different path does not affect (\ref{d-dt}). Moreover since (\ref{d-dt}) is unaffected if $a_\rho \gamma_z+b_\rho \gamma_w$ modified by a vector in the span of $X_\rho \gamma_z+Y_\rho \gamma_w$, it is determined by the projection on $(Y_\rho,-X_\rho)$ and therefore by $\psi(t)(\rho) = |E_\rho|(b_\rho X_\rho-a_\rho Y_\rho)$.

Summing over all $\rho \in \Sigma(1)$ and taking divisor classes, we get
\be \la{dchange}
{\bf d}_{t} -{\bf d}=\left[\sum_{\rho \in \Sigma(1)}\sum_{\alpha \in Z_\rho}\langle \alpha, \pi_\rho\rangle \nu(\alpha)\right].
\ee
A different choice of lift $\widetilde \Gamma_i$ of $\Gamma_i$ would modify (\ref{d-dt}) by an element of $H_1(\T,\Z)$, and therefore leave (\ref{dchange}) unchanged.
\\

 We find $\nu_t(\alpha)$ from $\psi(t)(\rho)$ as follows: if $\alpha$ is a zig-zag path in $Z_\rho$, consider a lift $\widetilde \alpha$ of it to the universal cover of $\T$. Suppose $\widetilde \alpha$ is the right boundary of a strip $S$. Since the effect of $t$ on strips is to translate them by $a_\rho \gamma_z+b_\rho \gamma_w$, where $\psi(t)(\rho)=|E_\rho|(b_\rho X_\rho-a_\rho Y_\rho)$, we get that \be \la{recipe} 
\nu_t(\alpha)=\nu(\beta),\ee
where $\beta \in Z_\rho$ is the zig-zag path such that one of its lifts to the universal cover is the right boundary of the strip $S-(a_\rho \gamma_z+b_\rho \gamma_w)$. 
Now that we have found $\nu_t$ and ${\bf d}_t-{\bf d}$, Fock's Theorem \ref{fockthm} gives us the birational map $\mu_t$.
\end{proof}

\begin{figure}
\centering
{\includegraphics[width=0.9\textwidth]{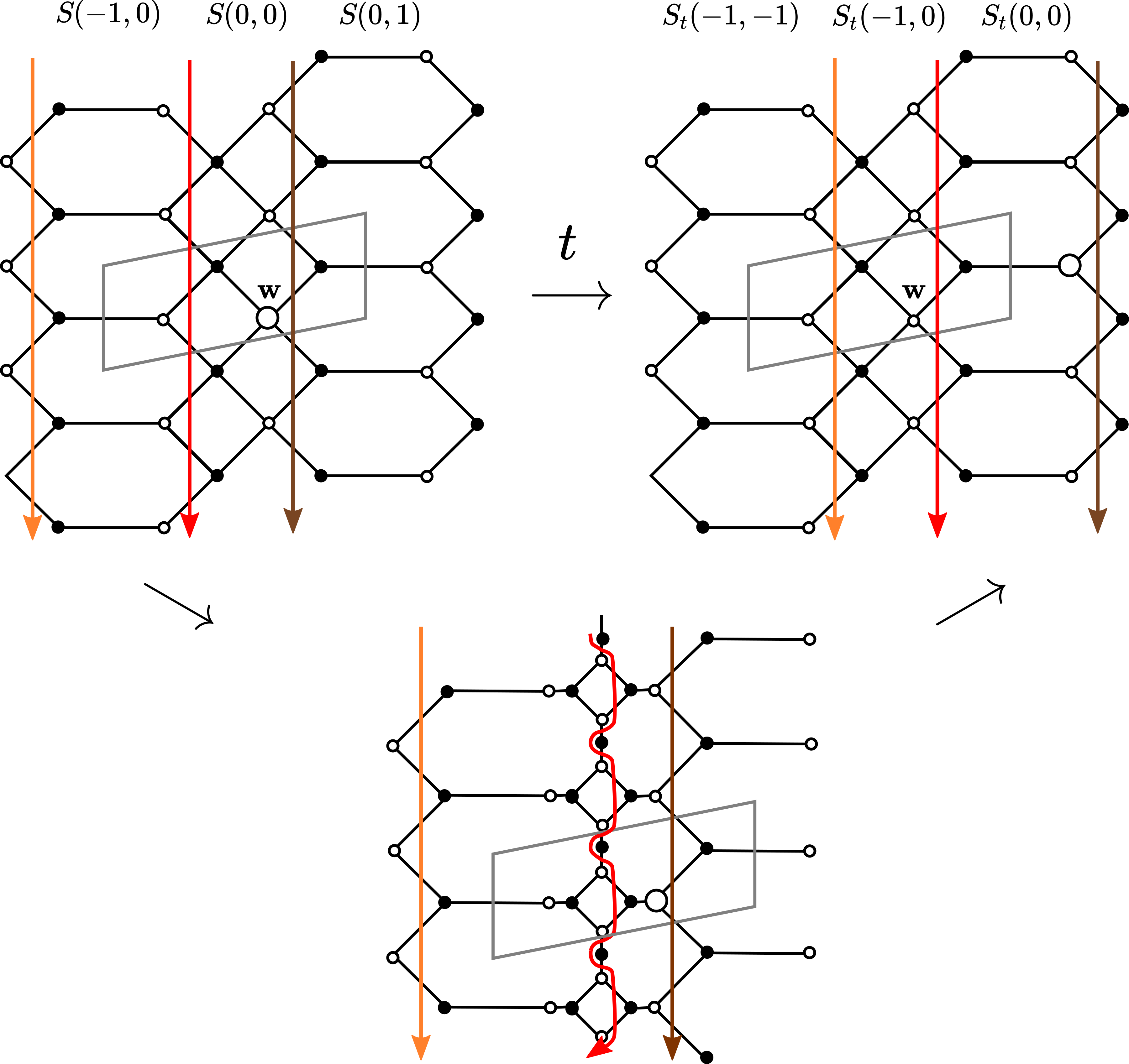}}
\caption{A shuffling algorithm of Borodin and Ferrari. In the first step, we do a spider move at the face immediately on the left of ${\bf w}$ and in the second step we all contract degree $2$ black vertices and translate. We have drawn one of the white vertices (the vertex ${\bf w}$ in the original graph) larger than the others to illlustrate the translation. The zig-zag paths associated with $\rho$ are translated right by $\frac{\gamma_z}{2}$ during $t$. The strips associated with $\rho$ are labeled at the top.}\label{borodinferrari}
\end{figure}

\subsection{Triviality of cluster transformations} \la{sec:triv}

From Fock's Theorem \ref{fockthm}, a cluster transformation $t$ is trivial if and only if $\nu_t=\nu$ and ${\bf d}({\bf w})-{\bf d}_t({\bf w}) = 0 \text{ in Pic}^0(C)$ for a generic curve $C \in |D_N|$. As a reality check, we observe that translation by $\gamma \in H_1(\T,\Z)$ is trivial: it induces $\nu_t=\nu$ and $d_t({\bf w})= d({\bf w})-[\text{div }(z,w)^\gamma] = d({\bf w}).$  Therefore by Proposition \ref{coruniq}, if $\psi(t)=0$, then $t$ is a trivial cluster transformation, so we have: 

\begin{lemma}
$\text{ker }\psi \subseteq \{\text{Trivial cluster transformations} \}$. 
\end{lemma}

We will show that in non-degenerate situations, this inclusion is an equality. We start with the following simple consequence of Fock's Theorem \ref{fockthm}.
\begin{lemma}\label{torlem}
Let $t$ be a cluster transformation $t$ such that $\psi(t)$ is a non-zero torsion element of $\Z_0^{\Sigma(1)}/jH_1(\T,\Z)$. Then $\mu_t$ is non-trivial.
\end{lemma}
\begin{proof}
From (\ref{recipe}), we see that $\nu_t \neq \nu$. Therefore by Theorem \ref{fockthm}, $\mu_t$ is non-trivial.
\end{proof}

We need the following technical result proved in Section \ref{sec:appendix}.
\begin{theorem}\la{theorem:appendix1}
Suppose $N$ has an interior lattice point. If $L$ is a non-trivial line bundle on the toric surface $X_N$ associated to $N$, then for a generic spectral curve $C$, we have $L|_C \ncong \mathcal O_C$. 
\end{theorem}
In other words, if $N$ has an interior lattice point, a generic spectral curve witnesses the non-triviality of line bundles on $X_N$. When $N$ has no interior lattice points, this fails: consider $N=\text{Conv}\{(0,0),(1,0),(0,1),(1,1)\}$ whose toric surface is $X_N = \P^1 \times \P^1$. The line bundles $\mathcal O(n,-n), n \in \Z$, are trivial on every spectral curve $C$ since they are all isomorphic to $\P^1$ and $\restr{\mathcal O(n,-n)}{C}$ has degree $0$.

The main theorem of the paper is:

\begin{theorem}\label{maincor}
If $g \neq 0$, the cluster modular group is 
$$
G_N \cong \Z^{\Sigma(1)}_0/jH_1(\T,\Z).
$$
When $g=0,$ we have
$$G_N \cong \Z^{\Sigma(1)}_0/\{f \in \Z^{\Sigma(1)}_0:  f(\rho) \text{ is divisible by } |E_\rho|  \text{ for all }\rho \in \Sigma(1) \}.$$

\end{theorem}

\begin{proof}
When $g=0,$ $S=\varnothing$, so $\mu_t$ is determined by the action of $t$ on $\nu$. Therefore $t$ is trivial if an only if $\nu_t=\nu$, which happens if and only if $\psi(t)(\rho)$ is divisible by $|E_\rho|$ for all $\rho \in \Sigma(1)$.

When $g \neq 0$, if $t$ is a cluster transformation such that $\psi(t) \neq 0$, then either:
\begin{enumerate}
    \item $\psi(t)$ is a non-zero torsion element: It is non-trivial by Lemma \ref{torlem}.
    \item $\psi(t)$ is not a torsion element: Consider the cluster transformation $t^n$ obtained by iterating $t$, where 
    $$n=k \prod_{\rho \in \Sigma(1)} |E_\rho|, k \in \Z.$$
    Then from Theorem \ref{fockthm} applied to $t^n$, we see that the induced map of spectral data by $t^n$ is given by $(C,S,\nu) \mapsto (C,S',\nu)$, where $S'$ is the generically unique degree $g$ effective divisor satisfying 
    
\be \label{divchangelem}
S' = S+\restr{D}{C},\ee
where $$D=n \sum_\rho \frac{\psi(t)(\rho)}{|E_\rho|}D_\rho$$ is a divisor at infinity of $X_N$. For sufficiently large $k$, $\mathcal O_{X_N}(D)$ is a line bundle on $X_N$ \cite{CLS11}*{Proposition 4.2.7}. Since $\psi(t)$ is not a torsion element, $D$ is not a torsion element of the divisor class group of $X_N$ either; indeed if $l D$ is a principal divisor for some $l \in \Z$, then $lD$ is the divisor of a character $(z,w)^\gamma$ for some $\gamma \in H_1(\T,\Z)$. However this means that $\psi(t^{ln}) \in jH_1(\T,\Z)$, contradicting the assumption that $\psi(t)$ is not a torsion element.

Therefore $\mathcal O_{X_N}(D) \ncong \mathcal O_{X_N}$, and so by Theorem \ref{theorem:appendix1}, we get $\restr{\mathcal O_{X_N}(D)}{C} \ncong \mathcal O_C$ for a generic spectral curve $C$. Therefore by (\ref{divchangelem}), $t^n$ is a not a trivial cluster transformation. Since $\mu_{t^n} = \mu_t^n$, $t$ is also not a trivial cluster transformation.
\end{enumerate}
Therefore $\text{ker }\psi=\{\text{Trivial cluster transformations } \Gamma \ra \Gamma\}$. By Theorem \ref{surj}, $\psi$ is surjective, so by the first isomorphism theorem \cite{Lang}*{I, \S 3}, the cluster modular group is
$$
G_N \cong \Z^{\Sigma(1)}_0/jH_1(\T,\Z).
$$

\end{proof}
Now we compute two examples of shuffling algorithms to illustrate our general results.
\begin{example}[A shuffling algorithm of Borodin and Ferrari] \la{eg:bf}

\begin{figure}
\centering
{\includegraphics[width=0.6\textwidth]{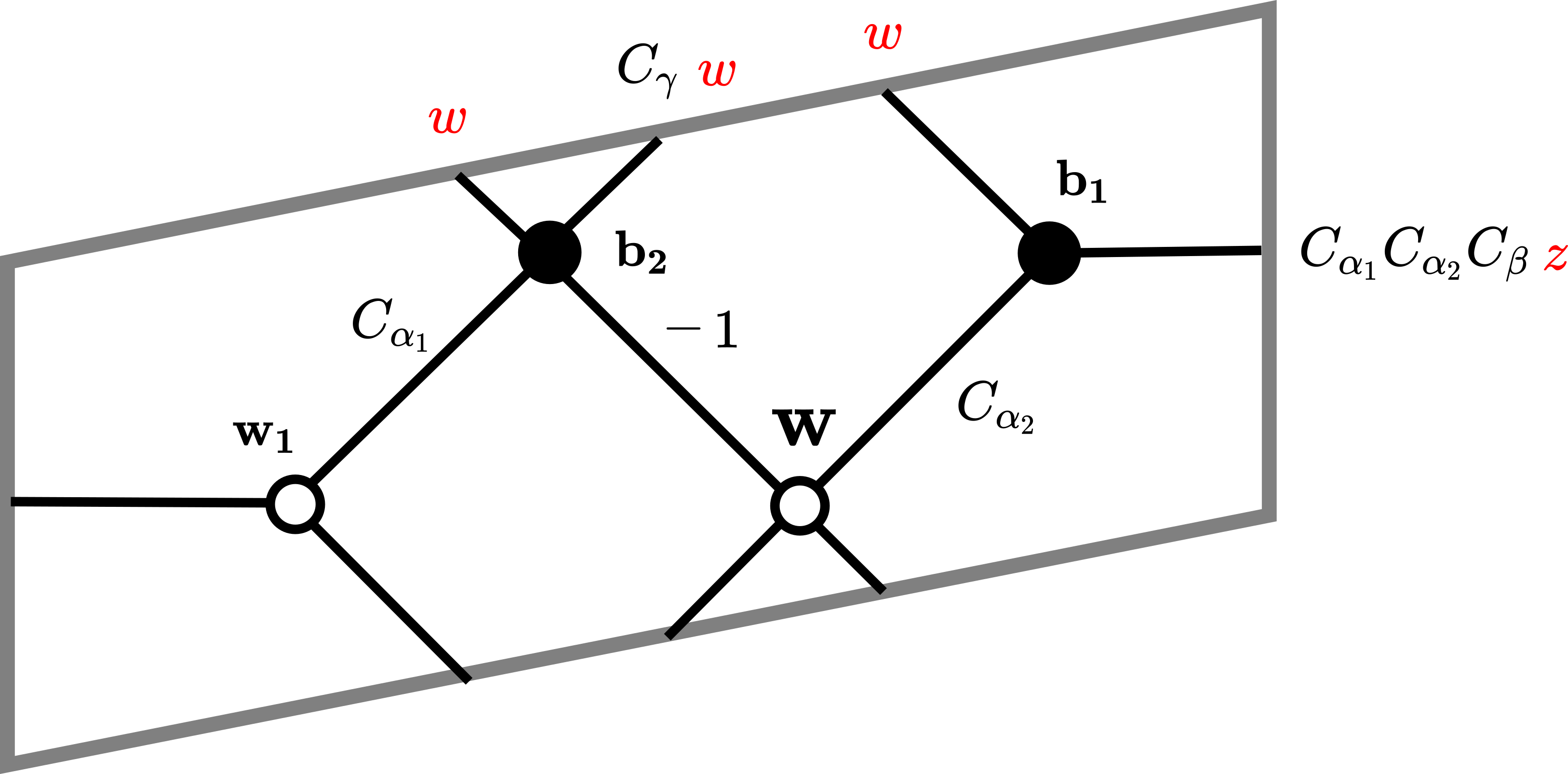}}
\caption{A cocycle representing $wt$, along with Kasteleyn signs and the characters $\phi(e)$ (red).}\label{bfgraphweights}
\end{figure}

The cluster transformation $t$ shown in Figure \ref{borodinferrari} for the graph $\Gamma$ in Figure \ref{bfgraph} was  studied by Borodin and Ferrari in \cite{BF18}. This example has appeared earlier in the physics literature, where it is known as the suspended pinch point. Suppose $wt \in \mathcal L_\Gamma$. Let $C_{\alpha}:=wt(\alpha)$ denote the monodromy of $wt$ around a zig-zag path $\alpha$. Then $\{C_{\alpha_1}, C_{\alpha_2}, C_\beta,C_\gamma\}$ is a set of coordinates for $\mathcal L_\Gamma$. A cocycle representing $wt$ in this basis is shown in Figure \ref{bfgraphweights}. The Kasteleyn matrix and spectral curve are:
\begin{align}
K(z,w)&=\begin{blockarray}{ccc}
{\bf b_1} & {\bf b_1} \\
\begin{block}{(cc)c}
  C_{\alpha_2}+  w & C_{\alpha_1}C_{\alpha_2}C_\beta z & {\bf w}\\
  -1+C_\gamma w & C_{\alpha_1}+ w & {\bf w_1}\\
\end{block}
\end{blockarray},\nonumber \\
C&=\{C_{\alpha_1} C_{\alpha_2} + C_{\alpha_1} w + C_{\alpha_2} w + w^2 + C_{\alpha_1} C_{\alpha_2}C_{\beta} z - C_{\alpha_1} C_{\alpha_2} C_{\beta} C_\gamma w z=0\}. \label{spectralcurve}
\end{align}
From (\ref{spectralcurve}), we recover the Newton polygon shown in Figure \ref{bfzzpaths}. Since the Newton polygon has no interior lattice points, $C$ is a genus $0$ curve and therefore the divisor $S$ in the spectral transform is irrelevant. There are two points at infinity of $C$ associated to the ray $\rho = \mathbb R_{\geq 0}(1,0)$:
\begin{align*}
P_1 &: z\ra 0, w=-C_{\alpha_1},\\
P_2 &: z \ra 0, w=-C_{\alpha_2},
\end{align*}
and only one point at infinity of $C$ associated to each of the other rays:
\begin{align*}
P_3 &: z \ra \frac{-1}{C_\beta}, w \ra 0;\\
P_4 &: z \ra \infty, w=\frac{1}{C_{\gamma}},\\
P_5 &: z,w \ra \infty, w/z=C_{\alpha_1} C_{\alpha_2} C_{\beta} C_\gamma.
\end{align*}
The bijection $\nu$ is:
$$
(\alpha_1,\alpha_2,\beta,\gamma,\delta) \mapsto (P_1,P_2,P_3,P_4,P_5).
$$
The spectral transform is:
\begin{align*}
    \mathcal L_\Gamma &\xrightarrow[]{\kappa_{\Gamma, {\bf w}}} \mathcal S_N\\
    (C_{\alpha_1}, C_{\alpha_2},C_\beta,C_\gamma) &\mapsto (C,\nu),
\end{align*}
and the induced map $\mu_t:\mathcal L_\Gamma \ra \mathcal L_\Gamma$ is:
$$
 (C_{\alpha_1}, C_{\alpha_2},C_\beta,C_\gamma) \mapsto (C_{\alpha_2}, C_{\alpha_1},C_\beta,C_\gamma),
$$
and the induced bijection $\nu_t$ is:
$$
(\alpha_1,\alpha_2,\beta,\gamma,\delta) \mapsto (P_2,P_1,P_3,P_4,P_5).
$$
Since $t$ exchanges the weights of the zig-zag paths $\alpha_1$ and $\alpha_2$, and also exchanges the points at infinity associated with these zig-zag paths, it follows that the diagram in Theorem \ref{fockthm} commutes in this example.

Let us choose the basis $(u_\alpha,u_\beta,u_\gamma)$ for $\Z^{\Sigma(1)}_0$, where the rays of the dual fan $\Sigma$ are labeled by zig-zag paths (see Figure \ref{bfzzpaths}). In the basis $(\gamma_z,\gamma_w)$ for $H_1(\T,\Z)$, the embedding $j$ is given by 
\[
\begin{pmatrix}
2&0\\
0&1 \\
-1&0\\
\end{pmatrix}.
\]
The homomorphism 
\begin{align*}
   \zeta: \Z^{\Sigma(1)}_0/j H_1(\T,\Z) &\ra \Z \nonumber \\
    (a,b,c) &\mapsto a+2c 
\end{align*}
is an isomorphism. We compute $\psi(t)=(1,0,0)$ and since $\zeta(1,0,0)=1$, we see that $\psi(t)$ generates $\Z^{\Sigma(1)}_0/j H_1(\T,\Z)$.
However, the cluster modular group $G_N$ is smaller. Since $t$ acts on $\mathcal L_\Gamma$ by interchanging the weights of $\alpha_1$ and $\alpha_2$, the cluster transformation $t^2$ is a trivial cluster transformation that is not in $\text{ker }\psi$. Since $\zeta \circ \psi(t^2)=2$, we get $G_N \cong \Z/2 \Z$. 
\end{example}

\begin{example}[Domino-shuffling] \la{eg:ds}
\begin{figure}
\centering
{\includegraphics[width=0.32\textwidth]{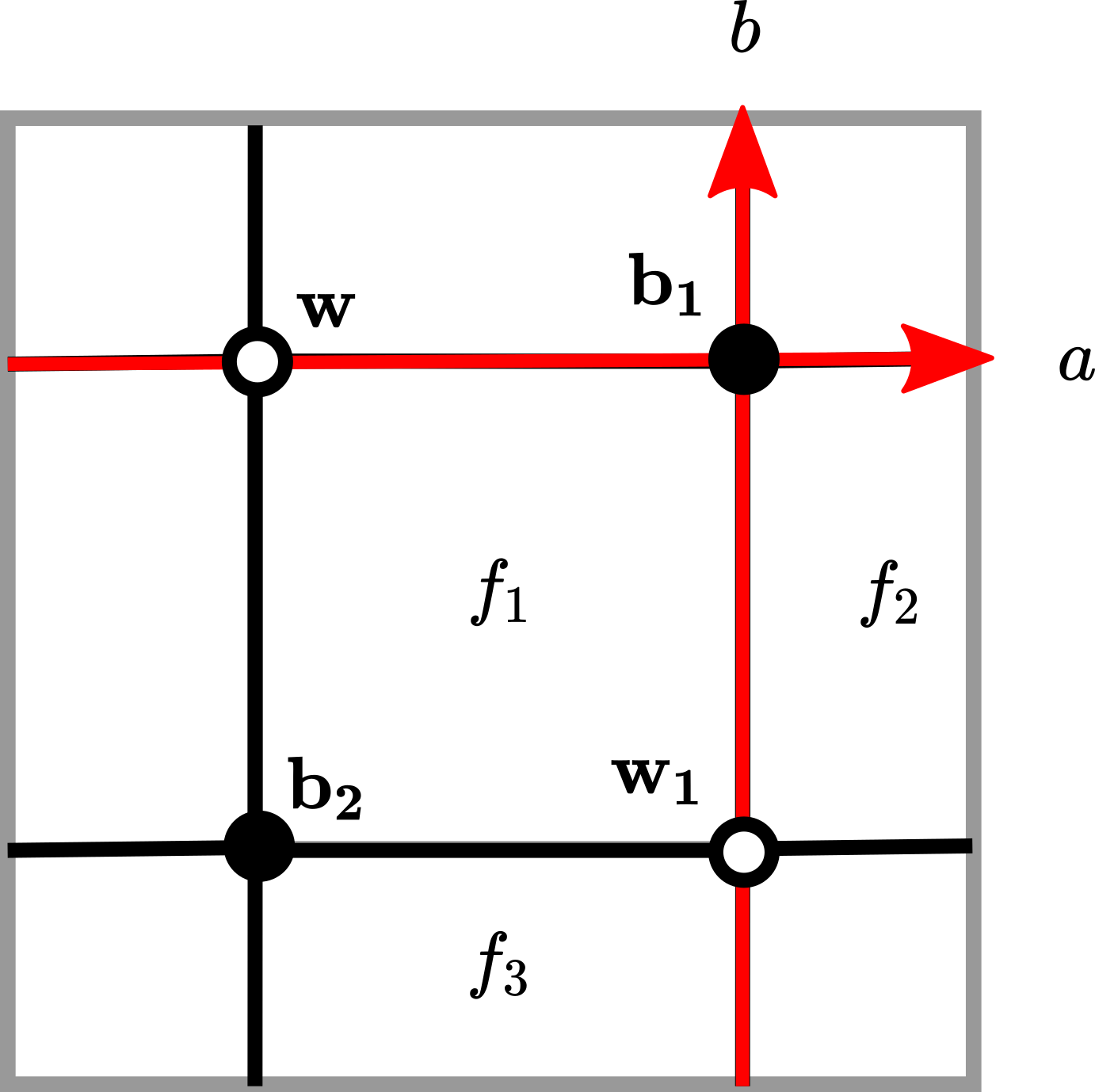}}
\caption{The bipartite torus graph involved in domino-shuffling showing the labels of the faces and the loops $a$ and $b$.}\label{dsgraph}
\end{figure}

\begin{figure}
\centering
{\includegraphics[width=0.4\textwidth]{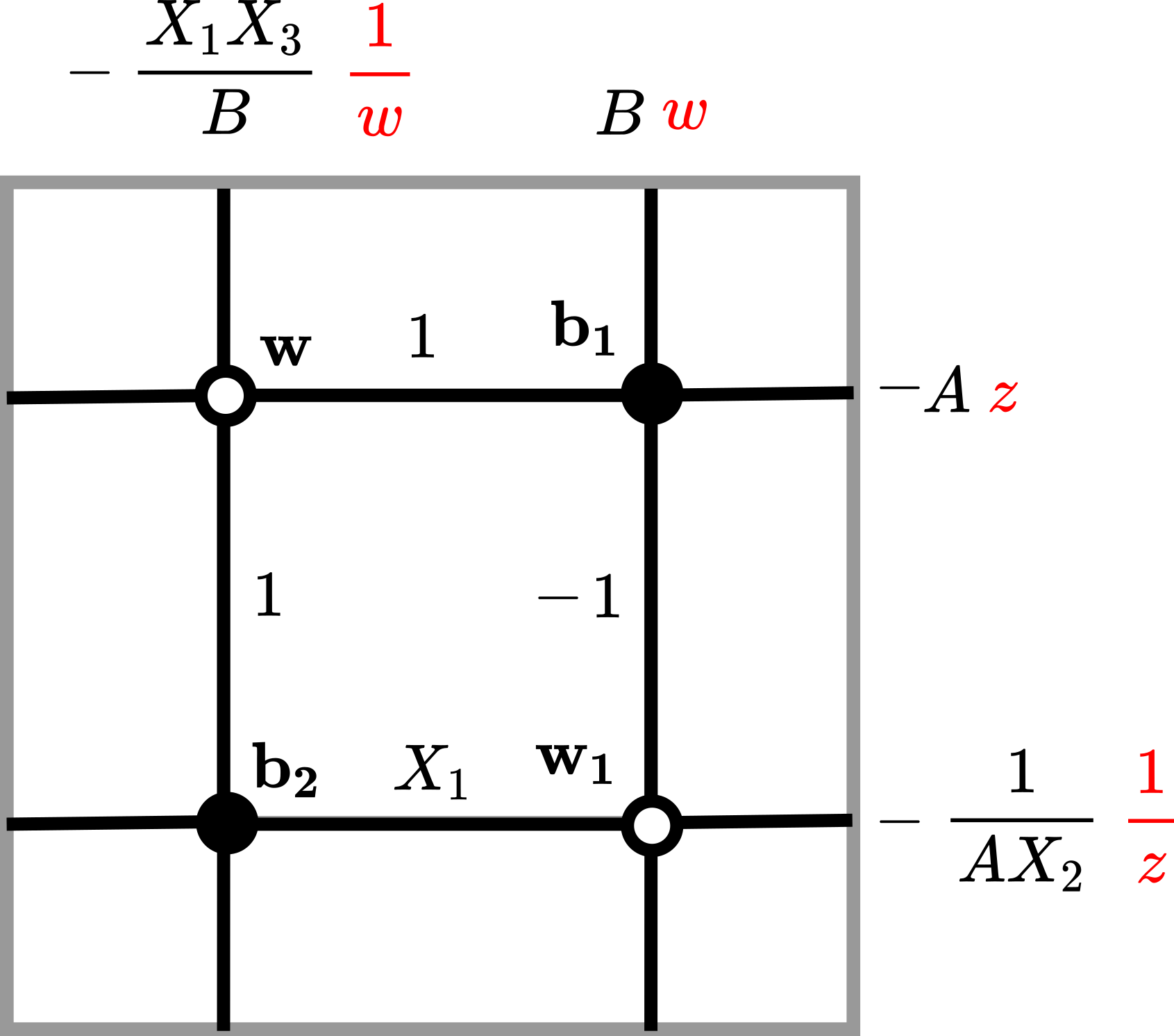}}
\caption{A cocycle representing $wt$, along with Kasteleyn signs and the characters $\phi(e)$ (red) for the graph in Figure \ref{dsgraph}.}\label{dsgraphweights}
\end{figure}

Consider the graph $\Gamma$ in Figure \ref{dsgraph}. The letters $a$ and $b$ are labeling two cycles in $H_1(\Gamma,\Z)$ whose projections to $\T$ generate $H_1(\T,\Z)$. Let $X_i:=wt(f_i), i=1,2,3$ and let $A:=wt(a), B:=wt(b)$. Then $(X_1,X_2,X_3,A,B)$ gives coordinates on $\mathcal L_\Gamma$. A cocycle representing $wt$ is shown in Figure \ref{dsgraphweights}. The Kasteleyn matrix and the spectral curve are
\begin{align}
K(z,w)&=\begin{blockarray}{ccc}
{\bf b_1}& {\bf b_2} \\
\begin{block}{(cc)c}
  1-A z & 1-\frac{X_1 X_3}{B w} & {\bf w}\\
  -1+Bw & X_1-\frac{1}{A X_2 z} & {\bf w_1}\\
\end{block}
\end{blockarray},\nonumber \\
C&=\left(1  + X_1 + \frac{A^2}{X_2} + X_1 X_3\right)- B w - \frac{X_1 X_3}{B w} - \frac{A}{X_2 z} - A X_1 z. \label{spectralcurve2}
\end{align}
There is only one zig-zag path in each direction, so $\nu$ is trivial. Let $\alpha,\beta,\gamma,\delta$ denote the zig-zag paths in counterclockise order starting from the blue edge in Figure \ref{npds}. The divisor $S=(p,q)$ has only one point, which is found by simultaneously solving $\text{adj }K(z,w)_{\bf b w}=0$, for all black vertices ${\bf b}$ of $\Gamma$.
 In this case, we get 
 $$
 (p,q)=\left(\frac{1}{A X_1 X_2}, \frac{1}{B}\right).
 $$
The spectral transform is given by
\begin{align}\la{stexample}
    \mathcal L_\Gamma &\xrightarrow[]{\kappa_{\Gamma, {\bf w}}} \mathcal S_N \nonumber \\
    (X_1,X_2,X_3,A,B) &\mapsto (C,(p,q)).
\end{align}
Consider the cluster transformation $t$ shown in Figure \ref{octrec2}. The induced map $\mu_t:\mathcal L_\Gamma \ra \mathcal L_\Gamma$ is:

 \begin{align}
     X_1 &\mapsto X_2 \frac{(1+X_1)^2}{(1+X_1 X_2 X_3)^2},\nonumber\\
     X_2 &\mapsto X_1^{-1},\nonumber\\
     X_3 &\mapsto X_1 X_2 X_3,\nonumber\\
     A &\mapsto \frac{AX_1(1+X_1X_2X_3)}{1+X_1}, \nonumber\\
     B &\mapsto \frac{B X_2}{(1+X_1)(1+X_1 X_2 X_3)}.\la{mutexample}
 \end{align}
 Let $S_t=(p_t,q_t)$ be the induced divisor. We can find the induced divisor in two different ways, verifying Theorem \ref{fockthm}.
 \begin{enumerate}
     \item We have ${\bf d}_t({\bf w})=[\nu(\alpha)-\nu(\beta)]=[\nu(\gamma)-\nu(\delta)]$, where the second equality comes from $\text{div}_C{z}=-\nu(\alpha)+\nu(\beta)+\nu(\gamma)-\nu(\delta)$. The unique point that satisfies
     $$
     S_t = S+{\bf d}({\bf w})-{\bf d}_t({\bf w}),
$$
      must be
 \be \la{pqt}
 p_t=\frac{1+X_1 X_2 X_3}{A X_2 (1+ X_1) },\quad q_t=\frac{1+X_1 X_2 X_3}{B
X_2 (1+ X_1) },
 \ee
because the rational function 
$$
R(z,w)=\frac{B w(-1+A X_1 X_2 z)}{-A X_1 X_2 X_3 z+B w(-1+A(1+X_1)X_2 z)},
$$
has zeros at $(p,q)$ and $\nu(\delta)$ and poles at $(p_t,q_t)$ and $\nu(\gamma)$, as can be checked using the coordinates of the points at infinity:
\begin{align*}
\nu(\gamma)&: z,w\ra 0, \frac{w}{z}=\frac{-X_1 X_2 X_3}{AB},\\
\nu(\delta)&: z \ra \infty, w \ra 0, z w = \frac{-X_3}{AB}.
\end{align*}
\item We can find the composition $\kappa_{\Gamma, {\bf w}} \circ \mu_t$ using (\ref{stexample}) and (\ref{mutexample}) to get (\ref{pqt}). 
 \end{enumerate}
Now we compute $G_N$ (see also \cite{FM16}*{Section 8.2}). We choose the basis $(u_\alpha,u_\beta,u_\gamma)$ for $\Z^{\Sigma(1)}_0$, where the rays in $\Sigma(1)$ are labeled by zig-zag paths, and the basis $(\gamma_z,\gamma_w)$ for $H_1(\T,\Z)$. The embedding $j$ is given in these bases by
\[
\begin{pmatrix}
-1&-1\\
1&-1\\
1&1
\end{pmatrix}.
\]
We can find the cokernel of $j$ i.e. $\Z^{\Sigma(1)}_0/jH_1(\T,\Z)$ by computing its Smith normal form (see for example \cite{sand}*{Section 2.4.2}). The Smith decomposition of the matrix of $j$ is
\[
\begin{pmatrix}
-1&0&0\\
-1&-1&0\\
1&0&1
\end{pmatrix}
\begin{pmatrix}
-1&-1\\
1&-1\\
1&1
\end{pmatrix}\begin{pmatrix}
1&-1\\
0&1
\end{pmatrix}= \begin{pmatrix}
1&0\\
0&2\\
0&0
\end{pmatrix},
\]
from which we see that we have an isomorphism \begin{align*}
  \zeta: \Z^{\Sigma(1)}_0/jH_1(\T,\Z) &\ra \Z \oplus \Z/2\Z\\
  (a,b,c) &\mapsto (a+c,-a-b + 2 \Z).
\end{align*} 
 We have $\psi(t)=(-1,1,0)$ (see Figure \ref{npds}), so $\zeta \circ \psi(t)= (-1,0+2\Z)$. Therefore $t$ is a generator of infinite order. Consider the cluster modular transformation $s$ given by the translation of $\Gamma$ by $\frac 1 2 (\gamma_z,\gamma_w)$. We have $\psi(s)=(-1,0,1,0)$, so that $\zeta \circ \psi(s)=(0,1+2 \Z)$ is a generator of order $2$. Since $s$ and $t$ are non-trivial cluster transformations, we get that the cluster modular group $G_N \cong \Z \oplus \Z/2\Z$ and is generated by $s$ and $t$.
\end{example} 
 
\section{Proof of Theorem \ref{theorem:appendix1}}\label{sec:appendix}
The goal of this Section is to prove Theorem \ref{theorem:appendix1}. The arguments assume some familiarity with algebraic geometry, we refer to Section \ref{subsection_ag} for more details. All polygons in this Section will be convex, integral and embedded in $\mathbb{R}^2$.

We define a \emph{building block} polygon to be a polygon $\Delta \subseteq \mathbb{R}^2$ that contains exactly one interior lattice point, and at most five lattice points in total. Note that a building block polygon has either three or four edges, and modulo lattice equivalence, there are only four (see Figure \ref{bbpoly}).

\begin{figure}
	\centering
	\begin{tikzpicture}[scale=1] 
		\draw (1,0) -- (0,1) -- (-1,-1) -- (1,0);
		\draw[fill=black] (0,0) circle (2pt);
		\draw[fill=black] (1,0) circle (2pt);
		\draw[fill=black] (0,1) circle (2pt);
		\draw[fill=black] (-1,-1) circle (2pt);
	\end{tikzpicture}\hspace{1cm}
	\begin{tikzpicture}[scale=1] 
		\draw (1,0) -- (0,1) -- (-1,0) -- (0,-1)--(1,0);
		\draw[fill=black] (0,0) circle (2pt);
		\draw[fill=black] (1,0) circle (2pt);
		\draw[fill=black] (0,1) circle (2pt);
		\draw[fill=black] (-1,0) circle (2pt);
		\draw[fill=black] (0,-1) circle (2pt);
	\end{tikzpicture}\hspace{1cm}
	\begin{tikzpicture}[scale=1] 
		\draw (1,0) -- (0,1) -- (-1,1) -- (0,-1)--(1,0);
		\draw[fill=black] (0,0) circle (2pt);
		\draw[fill=black] (1,0) circle (2pt);
		\draw[fill=black] (0,1) circle (2pt);
		\draw[fill=black] (-1,1) circle (2pt);
		\draw[fill=black] (0,-1) circle (2pt);
	\end{tikzpicture}\hspace{1cm}
	\begin{tikzpicture}[scale=1] 
		\draw (1,-1) -- (0,1) -- (-1,-1) -- (1,-1);
		\draw[fill=black] (0,0) circle (2pt);
		\draw[fill=black] (1,-1) circle (2pt);
		\draw[fill=black] (0,1) circle (2pt);
		\draw[fill=black] (-1,-1) circle (2pt);
		\draw[fill=black] (0,-1) circle (2pt);
	\end{tikzpicture}\hspace{1cm}
	\caption{The four building block polygons modulo lattice equivalence.}
	\label{bbpoly}
\end{figure}

Our first goal is to show that given any polygon $N$ that contains an interior lattice point, one can find a building block polygon $\Delta$ with $\Delta \subseteq N$ (Proposition \ref{prop:every:pol:contains:a:bb:pol}). We will find $\Delta$ as the polygon with the least number of lattice points among all polygons which are both contained in $N$ and have at least one interior lattice point. We begin with a few preparatory lemmas.

\begin{lemma}\label{lemma:minimal:pol:has:up:to:4:edges}
Consider a polygon $N$ that contains an interior lattice point. Then there exists a polygon $Q\subseteq N$ which contains an interior lattice point and at most four edges.
\end{lemma}
\begin{proof}

Let $x$ be an interior lattice point of $N$. Choose a polygon $Q\subseteq N$ which is minimal among all polygons contained in $N$ that contain $x$ as an interior lattice point, with the partial order induced by inclusion. We aim to show that $Q$ has either three or four edges. If not, let $a_1,\dots,a_n$ denote the vertices of $Q$, with $n>4$, labeled in clockwise order.
Consider the line segments joining $a_1$ with $a_3$, and $a_3$ with $a_5$ respectively. They divide $Q$ into three smaller polygons, each with fewer lattice points, and these segments intersect only at $a_3$, since by assumption $a_5$ is distinct from $a_3$ and $a_1$. Therefore $x$ is an interior point of either the polygon with vertices $a_1,a_3,a_4,\dots,a_n$ or the polygon with vertices $a_1,a_2,a_3,a_5,a_6, \dots , a_n$. This contradicts the minimality of $Q$.\end{proof}
\begin{lemma}\label{lemma:at:most:an:int:pt}
Consider a polygon $N$ with an interior lattice point. Then there is a polygon $Q\subseteq N$ which contains exactly one interior lattice point.

\begin{figure}
	\centering
	\begin{tikzpicture}[scale=0.8] 
		\draw (2,-0.2) -- (1.4,1) -- (-1,1.2) -- (-2,-0.4) -- (-1.5,-1.5) -- (1,-1.2) -- (2,-0.2);
		\draw [dashed] (2,-0.2) -- (-1.5,-1.5) -- (-1,1.2);
		\draw[fill=black] (0,0) circle (2pt);
		\draw[fill=black] (2,-0.2) circle (2pt);
		\draw[fill=black] (1.4,1) circle (2pt);
		\draw[fill=black] (-1,1.2) circle (2pt);
			\draw[fill=black] (-2,-0.4) circle (2pt);
		\draw[fill=black] (-1.5,-1.5) circle (2pt);
		\draw[fill=black] (1,-1.2) circle (2pt);
		\node (no) at (0.25,0.25) {$x$};
		\node (no) at (2+0.4,-0.2) {$a_5$};
		\node (no) at (-1,1.45) {$a_1$};
		\node (no) at (-1.5-0.25,-1.5-0.25) {$a_3$};
	\end{tikzpicture}
	\caption{Figure for Lemma \ref{lemma:at:most:an:int:pt}.}
\end{figure}

\end{lemma}
\begin{proof}
Consider a subpolygon $Q$ of $N$, which has at least two interior lattice points, $x$ and $y$. Up to shrinking $Q$, we can assume, by Lemma \ref{lemma:minimal:pol:has:up:to:4:edges}, that either $Q$ is a triangle or a quadrilateral.

If $Q$ is a triangle, consider the line through $x$ and $y$. It must meet an edge $\ell$ of $Q$. Let $a$ and $b$ be the vertices of $\ell$. Up to swapping $x$ and $y$, we can assume that the distance between $x$ and $\ell$ is less than the distance between $y$ and $\ell$. Then the triangle with vertices $y,$ $a$
 and $b$, with an interior point (namely $x$) has fewer interior lattice points than $Q$ ($y$ is not an interior point).
 
 If $Q$ is a quadrilateral, consider the two diagonals of $Q$. They have a single intersection point, so there must be a diagonal $\ell$ which does not contain both $x$ and $y$. Then $\ell$ divides $Q$ into two smaller polygons, and one of them must have an interior point.
 
 Therefore, if we consider a polygon which is minimal with respect to inclusion and has an interior lattice point, it must have a exactly one  interior lattice point.\end{proof}

\begin{figure}
	\centering
	\begin{tikzpicture}[scale=0.8] 
		\draw (1.1,2) -- (-1.2,-1.2) -- (2,-1.2) -- (1.1,2);
		\draw [dashed] (-1.2,-1.2) -- (0,0) ;
		\draw [dashed] (0,0) -- (2,-1.2);
		\draw [dashed] (0,0) -- (1.2,-1.2);
		\draw[fill=black] (0,0) circle (2pt);
		\draw[fill=black] (0.6,-0.6) circle (2pt);
		\draw[fill=black] (1.2,-1.2) circle (2pt);
		\draw[fill=black] (1.1,2) circle (2pt);
			\draw[fill=black] (2,-1.2) circle (2pt);
					\draw[fill=black] (-1.2,-1.2) circle (2pt);
		\node (no) at (0.25,0.25) {$y$};
		\node (no) at (0.6+0.25,-0.6+0.25) {$x$};
		\node (no) at (-1.45,-1.2) {$a$};
		\node (no) at (2+0.25,-1.2) {$b$};
		\node (no) at (0.2,-1.5-0.25) {$\text{Consider the triangle $\overline{aby}$}$.};
	\end{tikzpicture} \hspace{1cm} 		\begin{tikzpicture}[scale=0.8] 
		\draw (1.6,1.6)  -- (-1.5,1.5) -- (-1.3,-1.3) --(1.3,-1.3) -- (1.6,1.6);
		
		\draw [dashed] (-1.5,1.5) -- (1.3,-1.3);
		\draw [dashed] (1.6,1.6) --  (-1.3,-1.3);
		\draw[fill=black] (0,0) circle (2pt);
		\draw[fill=black] (1,1) circle (2pt);
		\draw[fill=black] (1.6,1.6) circle (2pt);
		\draw[fill=black] (-1.3,-1.3) circle (2pt);
			\draw[fill=black] (1.3,-1.3) circle (2pt);
					\draw[fill=black] (-1.5,1.5) circle (2pt);
		\node (no) at (0,0.3) {$y$};
		\node (no) at (1,1-0.25) {$x$};
		\node (no) at (-1.5-0.25,1.5+0.25) {$a$};
		\node (no) at (1.6+0.25,1.6+0.25) {$b$};
			\node (no) at (1.3+0.25,-1.3-0.25) {$c$};
		\node (no) at (0.2,-1.5-0.25) {$\text{Consider the triangle $\overline{abc}$}$.};
	\end{tikzpicture}
	\caption{Figure for Proposition \ref{prop:every:pol:contains:a:bb:pol}.}
\end{figure}

 \begin{proposition}\label{prop:every:pol:contains:a:bb:pol}
 Given any polygon $N$ with an interior lattice point, one can find a building block polygon $\Delta$ such that $\Delta \subseteq N$.
 \end{proposition}
 \begin{proof}
 Consider a polygon $Q\subseteq N$. From Lemma
 \ref{lemma:at:most:an:int:pt} and Lemma \ref{lemma:minimal:pol:has:up:to:4:edges}, we can assume, up to shrinking $Q$, that $Q$ has at most four edges, and a single interior point $x$. If $Q$ has four edges and five points, we are done, otherwise there is an edge $\ell$ with a point $y\in \ell$ which is not a vertex. Let $a,b$ be the two vertices of $Q$ not contained in $\ell$. Then the segments $\overline{ya}$ and $\overline{yb}$ intersect only at $y$, and divide $Q$ into three smaller polygons. One of them must contain $x$ in its interior. Therefore if $Q$ is minimal and has four edges, it must be a building block polygon.
 
 If instead $Q$ is a triangle, assume it has more than five lattice points. Then there are two lattice points $p,q$ which are on the boundary of $Q$, but are not vertices. If they belong to the same edge $\ell$, let $a$ be the vertex of $Q$ not contained in $\ell$. Then the segments $\ell_1:=\overline{ap}$ and $\ell_2:=\overline{aq}$ meet only at $a$, and divide $Q$ into smaller polygons. Then there must be one among $\ell_1$ and $\ell_2$ which does not contain $x$, and which is the side of a smaller polygon contained in $Q$ and with an interior point. 
 Similarly if $p$ and $q$ do not belong to the same edge, let $a$ be the vertex not contained in the edge containing $p$. Then the segments $\ell_1:=\overline{ap}$ and $\ell_2:=\overline{qp}$ intersect only at $p$ and divide $Q$ into smaller polygons. One of them must have an interior point.  

\begin{figure}
	\centering
	\begin{tikzpicture}[scale=0.8] 
		\draw (0.1,2)  -- (-1,0) -- (1.5,0) --(1.6,1.6) -- (0.1,2);
		
		\draw [dashed] (0.1,2) -- (0,0);
		\draw [dashed] (1.6,1.6) --  (0,0);
		\draw[fill=black] (0,0) circle (2pt);
		\draw[fill=black] (1,1) circle (2pt);
		\draw[fill=black] (1.6,1.6) circle (2pt);
		\draw[fill=black] (0.1,2) circle (2pt);
			\draw[fill=black] (1.5,0) circle (2pt);
					\draw[fill=black] (-1,0) circle (2pt);
		\node (no) at (0,-0.3) {$y$};
		\node (no) at (1,1-0.25) {$x$};
		\node (no) at (0.1,2+0.25) {$a$};
		\node (no) at (1.6+0.25,1.6+0.25) {$b$};
			\node (no) at  (1.5+0.25,-0.25) {$c$};
		\node (no) at (0.2,-0.5-0.25) {$\text{Consider the polygon $\overline{abcy}$}$.};
		
	\end{tikzpicture}\hspace{0.5cm}
	\begin{tikzpicture}[scale=0.8] 
		\draw (0,2) -- (-1.5,-1.2) --  (1.5,-1.2)--(0,2);
		\draw [dashed] (-.6,-1.2) -- (0,2) ;
		\draw [dashed] (0.6,-1.2) -- (0,2);
		
		\draw[fill=black] (-.6,-1.2) circle (2pt);
		\draw[fill=black] (.6,-1.2) circle (2pt);
		\draw[fill=black] (0,2) circle (2pt);
		\draw[fill=black] (0.7,-0.6) circle (2pt);
			\draw[fill=black] (-1.5,-1.2) circle (2pt);
					\draw[fill=black] (1.5,-1.2) circle (2pt);
		\node (no) at (-0.6,-1.45) {$p$};
		\node (no) at (0.6,-1.45) {$q$};
		\node (no) at (0.6+0.25,-0.6+0.25) {$x$};
		\node (no) at (0,2.25) {$a$};
		\node (no) at (1.5+0.25,-1.2) {$b$};
		\node (no) at (0.2,-1.7-0.25) {$\text{Consider the triangle $\overline{abp}$}$.};
	\end{tikzpicture} 	\hspace{0.5cm}
	\begin{tikzpicture}[scale=0.8] 
		\draw (0,2) -- (-1,-1) --  (2,0)--(0,2);
		\draw [dashed] (-0.5,0.5) -- (.5,-.5) ;
		\draw [dashed] (.5,-.5) -- (0,2);
		
		\draw[fill=black] (2,0) circle (2pt);
		\draw[fill=black] (-1,-1) circle (2pt);
		\draw[fill=black] (0,2) circle (2pt);
			\draw[fill=black] (.5,-.5) circle (2pt);
		\draw[fill=black] (0.25, 0.75) circle (2pt);	
			
					\draw[fill=black] (-0.5,0.5) circle (2pt);
		\node (no) at (.5,-.75) {$p$};
		\node (no) at (-0.5-0.25,0.75) {$q$};
		\node (no) at (0.5, 0.75) {$x$};
		\node (no) at (0,2.25) {$a$};
		\node (no) at (2+0.25,0) {$b$};
		\node (no) at (0.2,-1.7-0.25) {$\text{Consider the polygon $\overline{abpq}$}$.};
	\end{tikzpicture} 
	
	\caption{Figure for Lemma \ref{lemma:taking:subpolygone:is:projection}.}
\end{figure}

\end{proof}

 \begin{lemma}\label{lemma:taking:subpolygone:is:projection}
 Consider the projective toric surface $X_N$ associated to the polygon $N$, and let $Q\subseteq N$ be a subpolygon of $N$, with associated projective toric surface $X_Q$. The two polygons give rise to projective embeddings $X_N \subseteq \mathbb{P}^n$ and $X_Q \subseteq \mathbb{P}^m$, where $n:=|N \cap \Z^2|-1$ and $m:=|Q \cap \Z^2|-1$. There is a linear projection $\mathbb{P}^n \dashrightarrow \mathbb{P}^m$ which gives a $(\C^*)^2$-equivariant rational map $X_N \dashrightarrow X_Q$, which restricts to an isomorphism of the dense torii $(\C^*)^2$.
 \end{lemma}
 \begin{proof}
 Consider the characters $\chi_0,\dots,\chi_n$ corresponding to the lattice points in $N$, and let $\chi_0,\dots,\chi_m$ be those corresponding to the lattice points in $Q$. As in Section \ref{subsection_toric_surfaces}, consider the map $\Phi:(\C^*)^2\to \mathbb{P}^n$ given by $p\mapsto [\chi_0(p),\dots,\chi_n(p)]$. The toric variety $X_N$ is the closure of the image of $\Phi$, and similarly $X_Q$ is the closure of the map $\Phi:(\C^*)^2\to \mathbb{P}^m$ sending $p\mapsto [\chi_0(p),\dots,\chi_m(p)]$. Then the projection from the $n-m$ coordinates corresponding to the characters in $(N \setminus Q) \cap \Z^2$ gives the desired rational map.
 \end{proof}
\begin{theorem}\label{theorem:building:block:not:rbl}
Consider the projective toric surface $X_\Delta$ associated to a building block polygon $\Delta$, and let $\chi_0,\dots,\chi_n$ denote the characters corresponding to the lattice points in $\Delta$. Then there exists no line $\ell$ of $\mathbb{P}^n$, contained in $X_\Delta$, passing through the identity of the dense torus $(\C^*)^2\subseteq X$.
\end{theorem}
\begin{proof}
First observe that either $n=3$ or $n=4$, as $n+1$ is the number of lattice points of $\Delta$. Any line in $\mathbb{P}^n$ is the intersection of $n-1$ \emph{linearly independent} hyperplanes $H_1,\dots,H_{n-1}$.
When we restrict these hyperplanes to the torus $(\C^*)^2$, they can be written as $\restr{(H_i)}{(\C^*)^2}=\sum_j a_{i,j}z^i w^j$ where $(z,w)$ are the coordinates on $(\C^*)^2$.

Assume by contradiction that such a line exists. Then when intersected with $(\C^*)^2$ would give a codimension one subset of $(\C^*)^2$: it would be the zero locus of an homogeneous polynomial $f$. Then we can factor $(H_i)|_{(\C^*)^2}=fg_i$ for $g_i\in\mathbb{C}[z^{\pm},w^{\pm}]$. Consider the Newton polygon $P_f$ associated to $f$, defined as follows. If we write $f:=\sum c_{i,j}z^{i}w^{j}$, then $P_f$ is the convex hull of the points $(i,j)$ such that $c_{i,j}\neq 0$. Similarly, if we denote by $P_i$ the Newton polygon of $g_i$, then from 
\cite{Ost76}*{Theorem VI} we have that the Minkowski sum $P_f + P_i$ has vertices corresponding to the characters in $fg_i=(H_i)|_{(\mathbb{C}^{\times})^2}$. Recall (see Section \ref{subsection_toric_surfaces}) that these correspond to lattice points of $\Delta$. In particular,
\begin{center} $P_f + P_i$ is a subpolygon of $\Delta$ for $i=1,\dots,n-1$.
\end{center}
Now we show that a building block polygon cannot contain $n-1$ polygones as above.

First observe that since $f$ vanishes on the identity, there are at least two distinct pairs $(i,j)$ such that $c_{i,j}\neq 0$. In particular, $P_f$ contains a segment $s$.
Then, since the hyperplanes $H_i$ are linearly independent, the set $\bigcup_i P_i$ has at least $n-1$ elements, say ${x_1, \dots, x_{n-1}}$. Therefore the $n-1$ translates $x_i +s$ of $s$ must be contained in $\Delta$. Checking the four building block polygons in Figure \ref{bbpoly}, we see that this is not possible.
\end{proof}
\begin{corollary}\label{corollary:surface:with:interior:pt:not:rbl}
Any projective toric surface $X_N$ whose polygon $N$ contains an interior lattice point is not ruled by lines.
\end{corollary}
\begin{proof}From Proposition \ref{prop:every:pol:contains:a:bb:pol}, there is a building block polygon $\Delta \subseteq N$ corresponding to a toric surface $X_\Delta$. From Lemma \ref{lemma:taking:subpolygone:is:projection}, there is a rational map $\pi : X_N\dashrightarrow X_\Delta$. This rational map sends a line not contained in the indeterminacy locus to either a point or a line. So if $X_N$ is ruled by lines, there is a line $\ell$ passing through the identity of the torus in $X_N$. Then $\pi(\ell)$ is a line through the identity of the torus in $X_\Delta$, contradicting Theorem \ref{theorem:building:block:not:rbl}.
\end{proof}
We now to prove the main technical result of this section:
\begin{proposition}\label{prop:appendix}
Consider a projective toric surface $X_N$ associated to the polygon $N$ that contains at least one interior lattice point. Suppose  $L$ is a non-trivial line bundle on $X_N$. Then there is a hyperplane section $C \in |D_N|$ such that $C$ is irreducible and smooth, and such that $L|_C$ is not trivial.
\end{proposition}
\begin{proof}We prove the desired statement by contradiction.
The proof is divided into several steps.

\textit{Step 1.} We show that there is a pencil $\Pi\cong \mathbb{P}^1\subseteq \mathbb{P}(H^0(\mathbb{P}^n,\mathcal{O}_{\mathbb{P}^n}(1)))$ such that for every $x\in \Pi$, the corresponding hyperplane section is irreducible, and the generic member is smooth.

It suffices to show that the locus in $\mathbb{P}(H^0(\mathbb{P}^n,\mathcal{O}_{\mathbb{P}^n}(1)))$ corresponding to reducible hyperplane sections has closure of codimension at least 2. Indeed, otherwise it would contain an $(n-1)$-dimensional subset parameterizing reducible curves. 
Now we use the following:

\begin{lemma}[\cite{Lop91}*{Lemma II.2.4}]\label{lopez}
An irreducible non-degenerate surface $S \subset \mathbb P^n, n \geq 3$, has an $(n-1)$-dimensional family of reducible hyperplane sections if and only if $S$ is either ruled by lines, or is the Veronese surface, or its general projection in $\mathbb P^4$, or its general projection in $\mathbb P^3$ (the Steiner surface).
\end{lemma}
The Veronese surface is toric and  corresponds to the Newton polygon with vertices $$\text{Conv}\{(0,0), (2,0),(0,2)\},$$ and so has no interior lattice points. Since the number of interior lattice points in the Newton polygon is the genus of a generic hyperplane section \cite{CLS11}*{Proposition 10.5.8}, its general projections
have hyperplane sections of genus $0$. Since the generic hyperplane sections of $X_N$ have genus at least $1$, it is not the Veronese surface or its projections. Therefore by Lemma \ref{lopez} and Corollary \ref{corollary:surface:with:interior:pt:not:rbl}, we can find a generic pencil of hyperplane sections in $X_N$ such that all the members are irreducible. Since being smooth is an open condition (\cite[Theorem 12.2.4]{EGAIV}), we can also assume that the generic member is smooth.

Such a pencil $\Pi$ gives a rational map $X_N \dashrightarrow \mathbb{P}^1$, where the indeterminacy locus consists of points where the two hyperplane sections generating the pencil intersect. We can blow-up the toric surface $\pi:Y \to X_N$ to resolve the indeterminacy locus, and therefore obtain a morphism $f:Y\to \mathbb{P}^1$ whose fibers are the members of the pencil $\Pi$.

\textit{Step 2.} We want to show that $f_*(\mathcal{O}_Y)= \mathcal{O}_{\mathbb{P}^1}$.

The morphism $f$ is:
\begin{enumerate}
    \item flat since it is dominant with target a smooth curve (see \cite[Proposition III.9.7]{Hart}),
    \item proper since the source is proper and the target separated, (see \cite[Corollary II.4.8 (e)]{Hart}), and 
    \item  generically smooth since the generic member of $\Pi$ is smooth. 
\end{enumerate}  Observe that $f_*(\mathcal{O}_Y)$ is a torsion-free sheaf, since already $\mathcal{O}_Y$ has no zero divisors. Therefore, since the local rings of $\mathcal{O}_{\mathbb{P}^1}$ are DVRs and torsion free modules over a DVR are free, the sheaf $f_*(\mathcal{O}_Y)$ is locally free: it is a vector bundle. To check its rank, observe that there is a fiber $Y_p$ of $f$ at a point $p$ which is smooth and connected (the smooth irreducible member of $\Pi$). Therefore $h^0(\mathcal{O}_{Y_p})=1$, and from \cite{Vak17}*{28.1.1} there is an open subset $U\subseteq \mathbb{P}^1$ such that for $x \in U$ we have
$h^0(\mathcal{O}_{Y_x})=1$. Then from \cite{Vak17}*{28.1.5} this is the rank of $f_*(\mathcal{O}_Y)$ at $x$. In particular, the latter is a line bundle. But from the definition of push forward, $H^0(\mathbb{P}^1,f_*(\mathcal{O}_Y)) = H^0(Y,\mathcal{O}_Y)\cong \mathbb{C}$: we have that $f_*(\mathcal{O}_Y)$ is a line bundle on $\mathbb{P}^1$ with a single global section. From the description of the line bundles on $\mathbb{P}^1$ we have the desired isomorphism $f_*(\mathcal{O}_Y)\cong \mathcal{O}_{\mathbb{P}^1}$.

\textit{Step 3.} We prove that $\pi^*L = f^* G$ for a line bundle $G$ on $\mathbb{P}^1.$

Recall that by contradiction we are assuming that for every member $C$ of $\Pi$, $L|_C\cong \mathcal{O}_C$. The members of $\Pi$ are the fibers of $f$, thus for every fiber $F$ of $f$ we have $\pi^*(L)|_F\cong \mathcal{O}_F$. Then from \cite{Vak17}*{Proposition 28.1.11}, there is a line bundle $G$ on $\mathbb{P}^1$ such that $\pi^*(L)\cong f^*(G)$.

\textit{End of the argument.}
We proceed as in the second paragraph of \cite{Sta16}. We report it below for the convenience of the reader. 

Recall that the fibers of $f$ are elements of the linear series of $X_{N}\to \mathbb{P}^n$. We first check that if $U$ is the dense open subset where $\pi$ is an isomorphism, we have $\pi^*\mathcal{O}_{X_N}(1)_{|U} = f^*\mathcal{O}_{\mathbb{P}^1}(1)_{|U}$. Indeed, we check this equality for divisors, so let $H$ be a member of the pencil $\Pi$.
Then $\pi^*\mathcal{O}_{X_N}(1)|_U = \pi^*\mathcal{O}_{X_N}(H)|_U = \mathcal{O}_Y(\pi^{-1}H)|_U $: the line bundle $\pi^*\mathcal{O}_{X_N}(1)|_U$ is the line bundle associated to the divisor being the pull back of a  hyperplane section. When we restrict it to $U$, this agrees with the proper transform of $H$ restricted to $U$. Similarly, $\mathcal{O}_{\mathbb{P}^1}(1)$ is the line bundle associated to the divisor being a point on $\P^1$, so $f^*\mathcal{O}_{\mathbb{P}^1}(1)|_U$ is the line bundle associated to the divisor being the preimage of a point under $f$: a fiber of $f$. The desired equality follows since the proper transform of $H$ is a fiber of $f$.


Then $\pi^*\mathcal{O}_{X_N}(m)|_U = f^*\mathcal{O}_{\mathbb{P}^1}(m)|_U$ for every $m\in \mathbb{Z}$. But $G$ is a line bundle on $\mathbb{P}^1$, so $G = \mathcal{O}_{\mathbb{P}^1}(m)$ for a certain $m$ and $\pi^*L|_U = f^*\mathcal{O}_{\mathbb{P}^1}(m)|_U = \pi^*\mathcal{O}_{X_N}(m)|_U$. Then $$L|_U = L\otimes \pi_*(\mathcal{O}_Y)|_U = \pi_*(\pi^* L \otimes \mathcal{O}_Y)|_U  =  \pi_*\pi^* L|_U = \pi_*\pi^*\mathcal{O}_{X_N}(m)|_U=\mathcal{O}_{X_N}(m)|_U$$
where the first equality follows from \cite[Proposition V.3.4]{Hart}, and the second equality is the projection formula (\cite[Exercise II.5.1]{Hart}).

Now since $X_N$ is normal, if two line bundles are isomorphic on a subset with complement of codimension at least 2, then they are isomorphic. Therefore $L \cong \mathcal{O}_{X_N}(m)$.

But then $\mathcal{O}_{X_{N}}(m)$ restricts to the trivial line bundle on every section, so $m = 0$ and $L$ is trivial: this is the desired contradiction.
Therefore there must be a fiber $F$ of $f$ such that $\pi^*(L)|_F$ is not trivial. Then from Lemma \ref{lemma:being:a:trivial:line:bundle:is:closed}, there is an open subset $U\subseteq \mathbb{P}^1$ where for every $p\in U$ we have $\pi^*(L)|_{Y_p}$ is not trivial, and $Y_p$ is smooth (being smooth is an open condition, see \cite[Theorem 12.2.4]{EGAIV}).
\end{proof}
The following Lemma is well known, we provide a proof for completeness.
\begin{lemma}\label{lemma:being:a:trivial:line:bundle:is:closed}
Consider a flat proper morphism $X\to B$ with integral fibers, and let $L$ be a line bundle on $X$. Then the set $\{b\in B$ such that $L|_{X_b}\cong \cO_{X_b}\}$ is closed.
\end{lemma}
\begin{proof}
From the upper-semicontinuity theorems \cite{Vak17}*{28.1.1} the set $b\in B$ where $h^0(L|_{X_b})>0$ and $h^0(L^{-1}|_{X_b})>0$ is closed. It suffices to prove that if one has a line bundle $G$ on an integral proper (over $\mathbb{C}$) scheme $Y$, then $h^0(G)>0$ and $h^0(
G^{-1})>0$ if and only if $G\cong \cO_Y$.
This is the first paragraph of the proof of the Seesaw theorem \cite{Mum74}.
\end{proof}

We are now ready to prove Theorem \ref{theorem:appendix1}. The argument is roughly as follows. First we consider a parameter space for linear subspaces of $\mathbb{P}^n$: this is the dual projective space, namely $(\mathbb{P}^n)^{\vee}$. Then we construct space a $\cC$ together with a morphism $\pi:\mathcal{C}\to (\mathbb{P}^n)^{\vee}$ whose fiber over the point corresponding to a hyperplane is the curve $H\cap X_{N}$.  We check that $\pi$ is flat and proper.
Then we use Lemma \ref{lemma:being:a:trivial:line:bundle:is:closed} and the fact that being smooth is an open condition for flat and proper morphisms to argue that the set $H\in (\mathbb{P}^n)^{\vee}$ such that $L|_{H\cap X_{N}}$ is not trivial and $H\cap X_{N}$ is smooth is open inside $(\mathbb{P}^n)^{\vee}$. To conclude the argument we just need to check that this open set is not empty: this is Proposition \ref{prop:appendix}.

\begin{proof}[Proof of Theorem \ref{theorem:appendix1}]
 Let $(\mathbb{P}^n)^\vee$ be the projective dual projective space of $\mathbb{P}^n$. Consider the generic hyperplane section  \begin{center}
    $\cH := \{(x,H)\in \mathbb{P}^n\times (\mathbb{P}^n)^\vee: x \in H\}.$
\end{center}
The closed embedding $X_N \hookrightarrow \mathbb{P}^n$ given by the polytope $N$ gives the closed embedding $X_N \times (\mathbb{P}^n)^\vee \hookrightarrow \mathbb{P}^n \times(\mathbb{P}^n)^\vee$. We can construct the fibred product $$\cC:= \cH\times_{\mathbb{P}^n\times (\mathbb{P}^n)^\vee}X_N \times (\mathbb{P}^n)^\vee.$$
From the universal property of a fiber product, a point of $\cC$ corresponds to two points $a\in \mathcal{H}$ and $b \in X_N \times (\mathbb{P}^n)^\vee$ which map to the same point in $\mathbb{P}^n\times (\mathbb{P}^n)^\vee$. We can understand the space $\cC$ via its morphism $\pi:\cC \to (\mathbb{P}^n)^\vee$: a fiber of $\pi $ over the point of $(\mathbb{P}^n)^\vee$ corresponding to the hyperplane $H$ is the intersection $H\cap X_N$, i.e. it is a hyperplane section in $X_N$. We now check that $\pi$ is proper and flat.

For properness, observe that $\cC \to \cH$ is a closed embedding, since being a closed embedding is stable under base change and $X_N \times (\mathbb{P}^n)^\vee \hookrightarrow \mathbb{P}^n \times(\mathbb{P}^n)^\vee$ is a closed embedding. Moreover, $\cH \to (\mathbb{P}^n)^\vee$ is proper. So the composition $\pi:\cC \to (\mathbb{P}^n)^\vee$ is proper as it is the composition of a closed embedding and a proper morphism.

We check that the morphism $\pi$ is flat. From \cite{Vak17}*{24.7.A, (d)} it suffices to check that all the fibers have the same Hilbert polynomial.
For every hyperplane section $H$, we have a morphism $\cdot H:\cO_{\mathbb{P}^n}(-1)\to \cO_{\mathbb{P}^n}$ given by  multiplication by the polynomial that gives the equation of $H$. Since we are embedding $X_N$ in $\P^n$ using a basis of sections of $D_N$, the embedding is non-degenerate, i.e. not contained in a linear subspace of smaller dimension. In particular, the polynomial that gives $H$ is not the zero polynomial.  $\cO_{X_{N}}$ is a domain, so multiplication by a non-zero element is injective. Therefore we have the following exact sequence: 
$$0 \to \cO_{X_N}(-1)\xrightarrow{\cdot H} \cO_{X_N}\to \cO_C \to 0,$$
where $C:=X_N\cap H$. Then by definition of the Hilbert polynomial, we see that the Hilbert polynomial of $C$ does not depend on $H$. Therefore all the fibers of $\pi$ have the same Hilbert polynomial, so $\pi$ is flat.

We can take the pull-back of $L$ to $X_N\times (\mathbb{P}^n)^\vee$ and to $\cC$ to get a line bundle $G$ on $\cC$ which along each fiber $C=H\cap X_N$ of $\pi$ restricts to $L|_C$.
From Proposition \ref{prop:appendix}, there is a smooth fiber $F$ of $\pi$ such that $G|_F \ncong \cO_{F}$. We can replace $(\mathbb{P}^n)^\vee$ with the locus $U\subseteq (\mathbb{P}^n)^\vee$ where $\pi$ is smooth (which is open from \cite[Theorem 12.2.4]{EGAIV}, and contains the fiber $F$). Then Lemma \ref{lemma:being:a:trivial:line:bundle:is:closed} applies, giving the desired result.
\end{proof}
\section{Other interpretations of the (2-2) cluster modular group}
In this section, we discuss two other ways of describing the (2-2) cluster modular group. This section is not essential to the rest of the paper, and may be safely skipped.
\subsection{Translation of a zonotope}\la{sec:ef}
The first alternate description comes from mathematical physics. Eager and Franco \cite{EF}*{Section 3} consider the group $\Z^{V_N}/\Z$ of functions on the vertices of $N$ up to an additive constant and take its quotient by the homology group $H_1(\T,\Z)$. This group is isomorphic to the group $\Z^{\Sigma(1)}_0/jH_1(\T,\Z)$ considered by Fock and Marshakov \cite{FM16} that is defined in Section \ref{sec:fockmarsh}. Recall that $\Sigma(1)$ is in bijection with $E_N$. Given a function $g \in \Z^{\Sigma(1)}_0$, we obtain a function $\widetilde g \in \Z^{V_N}$ defined as follows:  Let $E$ be an edge of $N$ oriented counterclockwise along the boundary of $N$, and let $V_1$ and $V_2$ denote its head and tail respectively. Define $\widetilde g(V_1)-\widetilde g(V_2)=g(E)$. 

Eager and Franco then associate to each vertex $V \in V_N$ an integral vector $\widehat E_V$ in $\Z^{\Sigma(1)}_0/jH_1(\T,\Z)$ defined as follows: Let $\rho, \sigma \in \Sigma(1)$ be the two consecutive rays in counterclockwise cyclic order such that the two dimensional cone spanned by them is dual to the vertex $V$. The vector $\widehat E_V$ is  $\delta_{\rho}-\delta_{\sigma}$, which are the generators we considered in Section \ref{sec:surj}. Eager and Franco then consider the zonotope $Q$ in $\Z^{\Sigma(1)}_0/jH_1(\T,\Z) \otimes_\Z \R$ defined by the vectors $\widehat E_V$ and its dual polytope $\widehat P$. 
\paragraph{Claim 1.}
For a bipartite graph $\Gamma$ with Newton polygon $N$, the interior lattice points of some translate of $\frac{1}{2}\widehat{P}$ are naturally identified with the faces of $\Gamma$ (equivalently the vertices of the underlying quiver). \\

When the edges of $N$ are primitive, this identification is given by the discrete Abel map $f \mapsto {\bf d}(f)$. In \cite{EF}*{Section 8}, they provide the following procedure for generating sequences of spider moves (or \textit{Seiberg duality cascades} as they are called in the mathematical physics literature). As the zonotope $\frac 1 2 \widehat P$ is translated, each time a lattice point moves out of it, another lattice point enters it, and the faces corresponding to these lattice points are related by a spider move. Therefore translations of the zonotope $\frac 1 2 \widehat P$ induce sequences of spider moves. Translations by the vectors $\widehat E_V$ are called \textit{basic periodic transformations}, and as we mentioned, they coincide with the generators of the (2-2) cluster modular group that we use in the proof of surjectivity of $\psi$ in Section \ref{sec:surj}. Therefore Eager and Franco's description agrees with that of Fock and Marshakov, and if Claim 1 is rigorously established, will provide a different proof of surjectivity of $\psi$. 
\subsection{Picard group of the toric stack}\label{picstack}
This section uses technical notions from algebraic geometry and we refer the reader to \cite{TWZ18} for further background. 

In the \cite{FM16}*{Section 7.3}, Fock and Marshakov provide an alternate description of $\Z^{\Sigma(1)}_0/H_1(\mathbb T,\Z)$ as the group of divisor classes on the toric surface $X_N$ that restrict to degree $0$ divisors on a generic spectral curve $C$. While this is true for polygons whose sides are all primitive, the following example shows that it needs to be modified in the general case. The correct general statement is Proposition \ref{propp2} below.
\begin{example}\la{eg:prim}

\begin{figure}
	\centering
	
	\begin{tikzpicture}[scale=1] 
		\draw (2,-1) -- (2,1) -- (-1,1) -- (-1,-1)--(2,-1);
		\draw[fill=black] (0,0) circle (2pt);
		\draw[fill=black] (1,1) circle (2pt);
		\draw[fill=black] (0,1) circle (2pt);
		\draw[fill=black] (2,0) circle (2pt);
		\draw[fill=black] (2,1) circle (2pt);
		\draw[fill=black] (2,-1) circle (2pt);
		\draw[fill=black] (-1,1) circle (2pt);
		\draw[fill=black] (1,0) circle (2pt);
		\draw[fill=black] (-1,0) circle (2pt);
		\draw[fill=black] (0,-1) circle (2pt);
		\draw[fill=black] (-1,-1) circle (2pt);
		\draw[fill=black] (1,-1) circle (2pt);
	\end{tikzpicture}\hspace{1cm}
	\caption{The polygon in Example \ref{eg:prim}.}
	\label{fig:prim}
\end{figure}
Consider the polygon $N$ shown in Figure \ref{fig:prim}. Let
\[
u_1=(-1,0), \quad u_2=(0,-1), \quad u_3=(1,0), \quad u_4=(0,1)
\]
denote the primitive vectors generating the rays of the dual fan $\Sigma(1)$, and let $E_1,E_2,E_3,E_4$ denote the corresponding sides of $N$. In the basis $(u_1,u_2,u_3)$ for $\Z_0^{\Sigma(1)}$, the embedding $j$ is given by
\[
\begin{pmatrix}
-2&0\\
0&-3\\
2&0
\end{pmatrix}.
\]
Computing the Smith normal form, we get 
$
\Z_0^{\Sigma(1)}/jH_1(\T,\Z) \cong  \Z \oplus \Z/6\Z.
$

On the other hand, the divisor class group of the toric surface $X_N$ is
the cokernel of $\Z^2 \xrightarrow[]{\tilde j} \Z^{\Sigma(1)}$ where $\tilde j$ is given by the matrix
\[
\begin{pmatrix}
-1 & 0\\
0&-1\\
1&0\\
0&1
\end{pmatrix}.
\]

For a divisor $D$ represented by the element $(a_1,a_2,\dots,a_4) \in \Z^{\Sigma(1)}$, its degree upon restriction to a generic spectral curve is $\sum_{i=1}^4 |E_i| a_i=2 a_1 + 3 a_2 + 2 a_3+3a_4$.

$ \Z^{\Sigma(1)}/\tilde j \Z^2$ is isomorphic to $\Z^2$ where the isomorphim is $(a_1,a_2,a_3,a_4) \mapsto (a_1+a_3,a_2+a_4)$, so the group of divisors that have degree $0$ when restricted to a generic spectral curve is isomorphic to $\Z$. We can identify this group with the subgroup $\{(2a_1,3a_2,2a_3,3a_4):2 a_1 + 3 a_2 + 2 a_3+3a_4=0\}/jH_1(\T,\Z)$ of $\Z_0^{\Sigma(1)}/jH_1(\T,\Z)$. 
\end{example}
The problem is that when we restrict a toric divisor to a curve, we get the same multiplicity for all points at infinity of the curve intersecting that divisor, so we get only a subgroup of $G_N$. We can fix this by replacing the toric variety $X_N$ with a toric stack $\mathscr{X}_N$ which has divisors that are fractions of toric divisors as we now explain. Associated to $N$ is a a pair of data $\boldsymbol \Sigma=(\Sigma, \beta)$, called a \textit{stacky fan}, defined as follows:
\begin{enumerate}
    \item $\Sigma$ is the normal fan of $N$ in $H_1(\T,\Z)^\vee \otimes \R$;
    \item $\beta :\Z^{|\Sigma(1)|} \ra H_1(\T,\Z)^\vee$ is the homomorphism defined by $\beta(e_\rho):=|E_\rho| u_\rho$.
\end{enumerate}
In other words, the stacky fan is the data of the dual fan along with the lengths of the edges of $N$.

Just as the normal fan $\Sigma$ of $N$ can be used to construct toric surface $X_N$,  the stacky fan $\boldsymbol \Sigma$ gives rise to a stacky toric surface $\mathscr X_N$. The stacky toric surface $\mathscr X_N$ has as coarse moduli space the toric variety $X_N$; let us denote by $\pi:\mathscr X_N \ra X_N$ the projection. Let $\mathcal O_{\mathscr X_N} \left(\frac{1}{|E_\rho|}D_{\rho}\right)$ be the unique line bundle on $\mathscr X_N$ satisfying

$$\mathcal O_{\mathscr X_N} \left(\frac{1}{|E_\rho|}D_{\rho}\right)^{\otimes |E_\rho|} \cong \pi^*\mathcal O_{X_N}(D_\rho).$$

Just as the divisor class group of the toric surface $X_N$ is generated by toric divisors, the Picard group of the toric stack $\mathscr X_N$ is generated by the line bundles $\mathcal O_{\mathscr X_N} \left(\frac{1}{|E_\rho|}D_{\rho}\right)$. 

\begin{theorem}[Borisov and Hua, 2009 \cite{BH09}*{Proposition 3.3}]\label{bhthm}
The following is an isomorphism of groups:
\begin{align*}
    \Z^{\Sigma(1)}/H_1(\T,\Z) &\ra \text{Pic }(\mathscr X_N)\\
    f &\mapsto \mathcal O_{\mathscr X_N} \left(\sum_{\rho}\frac{f(E_\rho)}{|E_\rho|}D_{\rho}\right)
\end{align*} 
\end{theorem}
\begin{proposition}\la{propp2}
This isomorphism identifies $\mathbb Z^{\Sigma(1)}_0/j H_1(\T,\Z)$ with the subgroup of $\text{Pic }(\mathscr X_N)$ of line bundles $\mathcal O_{\mathscr X_N}(D)$ where $D=\sum_\rho{b_\rho}D_\rho$, satisfying 
$$
\sum_\rho |E_\rho| b_\rho=0.
$$
\end{proposition}

There is a version of Theorem \ref{fockthm} which illuminates this correspondence. 
\begin{theorem}[Treumann, Williams and Zaslow, 2018 \cite{TWZ18}*{Proposition 1.2}]
Let $t$ be a seed cluster transformation. Let $\mathscr C = C \times_{X_N} \mathscr X_N$ and let $i:\mathscr C \hookrightarrow \mathscr X_N$ be the embedding. We have:
$$
\mathcal O_{\mathscr C}(S_t) \cong \mathcal O_{\mathscr C}(S) \otimes i^*\mathcal O_{\mathscr X_N}\left( \sum_\rho \frac{\psi(E_\rho)}{|E_\rho|}D_\rho\right).
$$
\end{theorem}

\appendix

\bibliography{references}
\Addresses
\end{document}